\documentclass[a4paper]{amsart}
\usepackage[utf8]{inputenc}
\usepackage[T1]{fontenc}
\usepackage{lmodern}
\usepackage{amssymb,amsxtra}
\usepackage{nicefrac,mathtools}
\usepackage[lite]{amsrefs}

\title{Inverse semigroup actions on groupoids}

\author{Alcides Buss}
\email{alcides.buss@ufsc.br}
\address{Departamento de Matem\'atica\\
 Universidade Federal de Santa Catarina\\
 88.040-900 Florian\'opolis-SC\\
 Brazil}

\author{Ralf Meyer}
\email{rmeyer2@uni-goettingen.de}
\address{Mathematisches Institut\\
 Georg-August-Universit\"at G\"ottingen\\
 Bunsenstra\ss e 3--5\\
 37073 G\"ottingen\\
 Germany}
\thanks{Supported by CNPq/CsF (Brazil) and the German Research Foundation (Deutsche Forschungsgemeinschaft (DFG)) through the grant ``Actions of \(2\)\nb-groupoids on C*\nb-algebras.''}


\newcommand*{\MRref}[2]{ \href{http://www.ams.org/mathscinet-getitem?mr=#1}{MR \textbf{#1}}}
\newcommand*{\arxiv}[1]{\href{http://www.arxiv.org/abs/#1}{arXiv: #1}}

\renewcommand{\PrintDOI}[1]{\href{http://dx.doi.org/\detokenize{#1}}{doi: \detokenize{#1}}}

\BibSpec{book}{%
  +{}  {\PrintPrimary}                {transition}
  +{,} { \textit}                     {title}
  +{.} { }                            {part}
  +{:} { \textit}                     {subtitle}
  +{,} { \PrintEdition}               {edition}
  +{}  { \PrintEditorsB}              {editor}
  +{,} { \PrintTranslatorsC}          {translator}
  +{,} { \PrintContributions}         {contribution}
  +{,} { }                            {series}
  +{,} { \voltext}                    {volume}
  +{,} { }                            {publisher}
  +{,} { }                            {organization}
  +{,} { }                            {address}
  +{,} { \PrintDateB}                 {date}
  +{,} { }                            {status}
  +{}  { \parenthesize}               {language}
  +{}  { \PrintTranslation}           {translation}
  +{;} { \PrintReprint}               {reprint}
  +{.} { }                            {note}
  +{.} {}                             {transition}
  +{} { \PrintDOI}                   {doi}
  +{} { available at \url}            {eprint}
  +{}  {\SentenceSpace \PrintReviews} {review}
}

\usepackage{enumitem} %<--advanced enumerate, which handles label and ref
\setlist[enumerate,1]{label=\textup{(\arabic*)}}% ensure enumerates in theorems are upright

\usepackage{tikz}
\usetikzlibrary{matrix,calc,arrows,decorations.pathmorphing}
\tikzset{node distance=2cm, auto}

\tikzset{cd/.style=matrix of math nodes,row sep=2em,column sep=2em, text height=1.5ex, text depth=0.5ex}
\tikzset{cdar/.style=->,auto}
\tikzset{mid/.style={anchor=mid}} % put labels on the arrow
\tikzset{narrowfill/.style={inner sep=1pt, fill=white}}% style for nodes with filled background

\usepackage{microtype}
\usepackage[pdftitle={Inverse semigroup actions on groupoids},
  pdfauthor={Alcides Buss, Ralf Meyer},
  pdfsubject={Mathematics}
]{hyperref}

\theoremstyle{plain}
\newtheorem{theorem}{Theorem}
\newtheorem{lemma}[theorem]{Lemma}
\newtheorem{proposition}[theorem]{Proposition}
\newtheorem{corollary}[theorem]{Corollary}
\theoremstyle{definition}
\newtheorem{definition}[theorem]{Definition}
\theoremstyle{remark}
\newtheorem{remark}[theorem]{Remark}
\newtheorem{example}[theorem]{Example}
\numberwithin{theorem}{section}

\DeclareMathOperator{\PHomeo}{pHomeo}% partial homeomorphism
\DeclareMathOperator{\Prim}{Prim}    % primitive ideal space
\DeclareMathOperator{\supp}{supp}    % support
\DeclareMathOperator{\Map}{Map}      % bibundle maps
\DeclareMathOperator{\Bis}{Bis}      % inverse semigroup of bisections

\newcommand*{\nb}{\nobreakdash}
\newcommand{\idealin}{\mathrel{\triangleleft}} % relation of being an ideal

\newcommand*{\Star}{\(^*\)\nobreakdash-}
\newcommand*{\blank}{\textup{\textvisiblespace}}% for arguments in functors
\newcommand*{\cstar}{\texorpdfstring{$C^*$\nobreakdash-\hspace{0pt}}{*-}}
\newcommand*{\into}{\hookrightarrow}
\newcommand*{\onto}{\twoheadrightarrow}

\newcommand*{\C}{\mathbb C}
\newcommand*{\Z}{\mathbb Z}
\newcommand*{\R}{\mathbb R}
\newcommand*{\N}{\mathbb N}

\newcommand*{\Bound}{\mathbb B}   % adjointable operators on a Hilbert module
\newcommand*{\Comp}{\mathbb K}    % compact operators on a Hilbert module
\newcommand*{\Mat}{\mathbb M}     % matrices
\newcommand*{\Open}{\mathbb O}    % lattice of open subsets
\newcommand*{\Ideal}{\mathbb I}   % ideal lattice of a C*-algebra

\newcommand*{\Banb}{\mathfrak B}  % bundle of Banach spaces
\newcommand*{\Sect}{\mathfrak S}  % space of quasi-continuous sections
\newcommand*{\OCover}{\mathfrak U}% open cover

\newcommand*{\Peq}{\mathfrak{peq}}% partial equivalence bicategory

\newcommand*{\Mult}{\mathcal M}       % multiplier algebra
\newcommand*{\Cat}{\mathcal C}        % category
\newcommand*{\Hilm}[1][H]{\mathcal #1}% Hilbert module
\newcommand*{\A}{\mathcal A}          % C*-bundle

\newcommand*{\red}{\textup r}         % reduced
\newcommand*{\op}{\textup{op}}        % opposite
\newcommand*{\Cst}{\textup C^*}       % C*-algebra
\newcommand*{\Cont}{\textup C}        % continuous functions
\newcommand*{\Contc}{\Cont_\textup c} % compactly supporte continuous functions
\newcommand*{\dd}{\mathrm d}  % integration
\newcommand*{\Id}{\textup{Id}}% identity
\newcommand*{\defeq}{\mathrel{\vcentcolon=}}% used for definitions
\newcommand*{\congto}{\xrightarrow\sim}

\newcommand*{\norm}[1]{\lVert#1\rVert}% norm
\newcommand*{\cl}[1]{\overline{#1}}% closure
\newcommand*{\braket}[2]{\langle#1{\mid}#2\rangle}% right inner products
\newcommand*{\BRAKET}[2]{\langle\!\langle#1{\mid}#2\rangle\!\rangle}% left inner products
\newcommand*{\ket}[1]{\lvert#1\rangle}% parts of rank-one operators
\newcommand*{\bra}[1]{\langle#1\rvert}

\newcommand*{\s}{s} % source map of groupoids
\newcommand*{\rg}{r}% range map of groupoids

\newcommand{\Link}{L} % linking groupoid

\begin{document}
\begin{abstract}
  We define inverse semigroup actions on topological groupoids by
  partial equivalences.  From such actions, we construct saturated
  Fell bundles over inverse semigroups and non-Hausdorff étale
  groupoids.  We interpret these as actions on \cstar{}algebras by
  Hilbert bimodules and describe the section algebras of these Fell
  bundles.

  Our constructions give saturated Fell bundles over non-Hausdorff
  étale groupoids that model actions on locally Hausdorff spaces.  We
  show that these Fell bundles are usually not Morita equivalent to an
  action by automorphisms.  That is, the Packer--Raeburn Stabilisation
  Trick does not generalise to non-Hausdorff groupoids.
\end{abstract}
\maketitle

\tableofcontents

\section{Introduction}
\label{sec:introduction}

Two of the most obvious actions of a groupoid~\(G\) are those by left
and right translations on its arrow space~\(G^1\).  If~\(G\) is
Hausdorff, they induce continuous actions of~\(G\) on the
\(\Cst\)\nb-algebra \(\Cont_0(G^1)\).  What happens if~\(G\) is
non-Hausdorff?

Let~\(G\) be a non-Hausdorff, étale groupoid with Hausdorff, locally
compact object space~\(G^0\).  Then~\(G^1\) is locally Hausdorff, that
is, it has an open covering \(\OCover = (U_i)_{i\in I}\) by
Hausdorff subsets: we may choose~\(U_i\) so that the range and source
maps restrict to homeomorphisms from~\(U_i\) onto open subsets of the
Hausdorff space~\(G^0\).

The covering~\(\OCover\) yields an étale, locally compact,
Hausdorff groupoid~\(H\) with object space \(H^0 \defeq
\bigsqcup_{i\in I} U_i\), arrow space \(H^1\defeq \bigsqcup_{i,j\in I}
U_i\cap U_j\), range and source maps \(\rg(i,j,x)\defeq (i,x)\) and
\(\s(i,j,x)\defeq (j,x)\), and multiplication \((i,j,x)\cdot (j,k,x) =
(i,k,x)\).  The groupoid~\(H\) is known as the \emph{\v{C}ech
  groupoid} for the covering~\(\OCover\).  In noncommutative
geometry, we view the groupoid \(\Cst\)\nb-algebra \(\Cst(H)\) as the
algebra of functions on the non-Hausdorff space~\(G^1\).  Is there
some kind of action of~\(G\) on~\(\Cst(H)\) that corresponds to the
translation action of~\(G\) on~\(G^1\)?

There is no action of~\(G\) on \(\Cst(H)\) in the usual sense because
there is no action of~\(G\) on~\(H\) by automorphisms.  The problem is
that arrows \(g\in G^1\) have many liftings \((i,g)\in H^0\).  To let
\(g\in G\) act on~\(H\), we must choose \(k\in I\) with \(gh\in U_k\)
for \(h\in U_j\) with \(\rg(h)=\s(g)\).  It may, however, be
impossible to choose~\(k\) continuously when~\(h\) varies in~\(U_j\).
This article introduces \emph{actions by partial equivalences} in
order to make sense of the actions of~\(G\) on \(H\) and~\(\Cst(H)\).

At first, we replace~\(G\) by its inverse semigroup of bisections
\(S=\Bis(G)\).  This inverse semigroup cannot act on~\(H\) by partial
groupoid \emph{isomorphisms} for the same reasons as above.  It does,
however, act on~\(H\) by partial \emph{equivalences} because the
equivalence class of~\(H\) is independent of the covering (see also
\cite{Kasparov-Skandalis:Buildings}*{Lemma 4.1}); thus partial
homeomorphisms on~\(G^1\) lift to partial equivalences of~\(H\) in a
canonical way.  We will see that an \(S\)\nb-action by partial
equivalences on a \v{C}ech groupoid for a locally Hausdorff
space~\(Z\) is equivalent to an \(S\)\nb-action on~\(Z\) by partial
homeomorphisms.

Let~\(S\) act on a groupoid~\(H\) by partial equivalences.  Then we
build a \emph{transformation groupoid} \(H\rtimes S\).  Special cases
of this construction are the groupoid of germs for an action of~\(S\)
on a space by partial homeomorphisms, the semidirect product for a
group(oid) action on another group(oid) by automorphisms, and the
linking groupoid of a single Morita--Rieffel equivalence.  The
original action is encoded in the transformation groupoid
\(L\defeq H\rtimes S\) and open subsets \(L_t\subseteq L\) with
\[
L_t\cdot L_u = L_{tu},\quad
L_t^{-1} = L_{t^*},\quad
L_t\cap L_u = \bigcup_{v\le t,u} L_v,\quad
L^1=\bigcup_{t\in S} L_t
\]
and \(H=L_1\).  We call such a family of subsets an
\emph{\(S\)\nb-grading} on~\(L\) with \emph{unit fibre}~\(H\).  Any
\(S\)\nb-graded groupoid is a transformation groupoid for an
essentially unique action of~\(S\) by partial equivalences on its unit
fibre.  This is a very convenient characterisation of actions by
partial equivalences.

An action of an inverse semigroup~\(S\) on~\(H\) by partial
equivalences cannot induce, in general, an action of~\(S\)
on~\(\Cst(H)\) by partial automorphisms in the usual sense (as defined by
Sieben~\cite{Sieben:crossed.products}).  But we do get an action by
partial Morita--Rieffel equivalences, that is, by Hilbert bimodules.
We show that actions of~\(S\) by Hilbert bimodules are equivalent to
(saturated) Fell bundles over~\(S\).  Along the way, we also
drastically simplify the definition of Fell bundles over inverse
semigroups in~\cite{Exel:noncomm.cartan}.  Our approach clarifies in
what sense a Fell bundle over an inverse semigroup is an ``action'' of
the inverse semigroup on a \(\Cst\)\nb-algebra.

In the end, we want an action of the groupoid~\(G\) itself, not of the
inverse semigroup~\(\Bis(G)\).  For actions by automorphisms, Sieben
and Quigg~\cite{Sieben-Quigg:ActionsOfGroupoidsAndISGs} characterise
which actions of~\(\Bis(G)\) come from actions of~\(G\).  We extend
this characterisation to Fell bundles: a Fell bundle over~\(\Bis(G)\)
comes from a Fell bundle over~\(G\) if and only if the restriction of
the action to idempotents in~\(\Bis(G)\) commutes with suprema of
arbitrarily large subsets.  This criterion only works for~\(\Bis(G)\)
itself.  In practice, we may want to ``model''~\(G\) by a smaller
inverse semigroup~\(S\) such that \(G^0\rtimes S\cong G\).  We
characterise which Fell bundles over such~\(S\) come from Fell bundles
over~\(G\).

In particular, our action of~\(\Bis(G)\) on~\(\Cst(H)\) for the
\v{C}ech groupoid associated to~\(G^1\) does come from an action
of~\(G\), so we get Fell bundles over~\(G\) that describe the left and
the right translation actions on~\(G^1\).  For these Fell bundles, we
show that the section \(\Cst\)\nb-algebras are Morita equivalent to
\(\Cont_0(G^0)\).  More generally, for any principal \(G\)\nb-bundle
\(X\to Z\), the section algebra of the Fell bundle over~\(G\) that
describes the action of~\(G\) on a \v{C}ech groupoid for~\(X\) is
Morita--Rieffel equivalent to~\(\Cont_0(Z)\), just as in the more
classical Hausdorff case (see
Proposition~\ref{pro:Cstar_action_on_G1}).

For any action of an inverse semigroup~\(S\) on a locally compact
groupoid~\(H\) by partial equivalences, we identify the section
\(\Cst\)\nb-algebra of the resulting Fell bundle over~\(S\) with the
groupoid \(\Cst\)\nb-algebra of the transformation groupoid.  In brief
notation,
\[
C^*(H)\rtimes S\cong C^*(H\rtimes S).
\]
This generalises the well-known isomorphism
\[
\Cont_0(X)\rtimes S\cong C^*(X\rtimes S)
\]
for inverse semigroup actions on Hausdorff locally compact spaces by
partial homeomorphisms.

For a Hausdorff locally compact groupoid, any Fell bundle is
equivalent to an ordinary action on a stabilisation (Packer--Raeburn
Stabilisation Trick, see also
\cite{Buss-Meyer-Zhu:Higher_twisted}*{Proposition~5.2}).  In contrast,
our Theorem~\ref{the:no-go-theorem} shows that a non-Hausdorff
groupoid has no action by automorphisms that describes its translation
action on~\(G^1\).  Thus we really need Fell bundles to treat these
actions of a non-Hausdorff groupoid.

Now we explain the results of the individual sections of the paper.

In Section~\ref{sec:partial_equivalence_sub}, we study partial
equivalences between topological groupoids.  We show, in particular, that the
involution that exchanges the left and right actions on a partial
equivalence behaves like the involution in an inverse semigroup.

Section~\ref{sec:ActionsGroupoidsByIso} introduces inverse semigroup
actions by partial equivalences.  We show that the rather
simple-minded definition implies further structure, which is needed
to construct the transformation groupoid.  Once we know that actions
by partial equivalences are essentially the same as \(S\)\nb-graded
groupoids, we treat many examples.  This includes actions on spaces
and \v{C}ech groupoids; in particular, an \(S\)\nb-action on a space
by partial homeomorphisms induces an action by partial equivalences on
any \v{C}ech groupoid for a covering of the space.  We describe a
group action by (partial) equivalences as a kind of extension by the
group.  We show that any (locally) proper Lie groupoid is a
transformation groupoid for an inverse semigroup action on a very
simple kind of groupoid: a disjoint union of transformation groupoids
of the form \(V\rtimes K\), where~\(V\) is a vector space, \(K\) a
compact Lie group, and the \(K\)\nb-action on~\(V\) is by an
\(\R\)\nb-linear representation.  This is meant as an example for
gluing together groupoids along partial equivalences.

In Section~\ref{sec:S_acts_on_Cstar} we define inverse semigroup
actions on \cstar{}algebras by Hilbert bimodules.  The theory is
parallel to that for actions on groupoids by partial equivalences
because both cases have the same crucial algebraic features.  We show
that actions by Hilbert bimodules are equivalent to saturated Fell
bundles.  This simplifies the original definition of Fell bundles over
inverse semigroups in~\cite{Exel:noncomm.cartan}.

In Section~\ref{sec:Fell_from_action}, we turn inverse semigroup
actions on groupoids by partial equivalences into actions on groupoid
\cstar{}algebras by Hilbert bimodules.  We do this in two different
(but equivalent) ways, by using transformation groupoids and abstract
functorial properties of our constructions.  The approach using
transformation groupoids suggests that the section \(\Cst\)\nb-algebra
of the resulting Fell bundle is simply the groupoid
\(\Cst\)\nb-algebra of the transformation groupoid: \(\Cst(H)\rtimes S
\cong \Cst(H\rtimes S)\).  We prove this, and a more general result
for Fell bundles over~\(H\rtimes S\).

In Section~\ref{sec:isg_to_groupoid} we relate inverse semigroup
actions to actions of corresponding étale groupoids.  In particular,
we characterise when an action of~\(\Bis(G)\) comes from an action
of~\(G\).  Finally, we can then treat our motivating example and turn
a groupoid action on a locally Hausdorff space~\(Z\) into a Fell
bundle over the groupoid.  We may also describe the section
\(\Cst\)\nb-algebra in this case, which plays the role of the crossed
product.  If the action is free and proper, then the result is
Morita--Rieffel equivalent to \(\Cont_0(Z/G)\).  We also define
``proper actions'' of inverse semigroups on groupoids.  We show that
a free and proper action can only occur on a groupoid that is
equivalent to a locally Hausdorff and locally quasi-compact space.

Section~\ref{sec:automorphisms_fail} shows that the translation action
of a non-Hausdorff étale groupoid on its arrow space cannot be
described by a groupoid action by automorphisms in the usual sense.
Our previous theory shows, however, that we may describe such actions
by groupoid Fell bundles.  Thus the no-go theorem in
Section~\ref{sec:automorphisms_fail} shows that the Packer--Raeburn
Stabilisation Trick fails for non-Hausdorff groupoids, so Fell bundles
are really more general than ordinary actions in that case.

In Section~\ref{sec:explicit_example}, we examine a very simple
explicit example to illustrate the no-go theorem and to see how our
main results avoid it.

Appendix~\ref{sec:preliminaries} deals with topological groupoids,
their actions on spaces and equivalences between them.  The main point
is to define principal bundles and (Morita) equivalence for
\emph{non-Hausdorff} groupoids in such a way that the theory works
just as well as in the Hausdorff case.  Among others, we show that a
non-Hausdorff space is equivalent to a Hausdorff, locally compact
groupoid if and only if it is locally Hausdorff and locally
quasi-compact, answering a question
in~\cite{Clark-Huef-Raeburn:Fell_algebras}.

Appendix~\ref{sec:Banach_fields} contains a general technical result
about upper semicontinuous fields of Banach spaces over locally
Hausdorff spaces and uses it to prove \(\Cst(H)\rtimes S \cong
\Cst(H\rtimes S)\) for inverse semigroup actions on groupoids and a
more general statement involving Fell bundles over \(H\rtimes S\).

\section{Partial equivalences}
\label{sec:partial_equivalence_sub}

In this section and the next one, we work in the category of
topological spaces and continuous maps, without assuming spaces to
be Hausdorff or locally compact.  Appendix~\ref{sec:preliminaries}
shows how topological groupoids, their actions, principal bundles,
and equivalences between them should be defined so that the theory
goes through smoothly without extra assumptions on the underlying
topological spaces.

Our main applications deal with groupoids that have a Hausdorff,
locally compact object space and a locally Hausdorff, locally
quasi-compact arrow space.  We care about actions of such
groupoids~\(G\) on locally Hausdorff spaces~\(Z\).  It is very
convenient to encode such an action by the transformation groupoid
\(G\ltimes Z\).  Its object space~\(Z\) is only locally Hausdorff.
When we allow such topological groupoids, the usual definition of
equivalence for topological groupoids breaks down because orbit
spaces of proper actions are always Hausdorff, so the actions on an
equivalence bispace cannot be proper unless the object spaces of the
two groupoids are Hausdorff.  Jean-Louis Tu's definition
in~\cite{Tu:Non-Hausdorff} works -- it is equivalent to what we do.
But the theory becomes more elegant if we also drop the local
compactness assumption and thus no longer use proper maps in our
basic definitions.  The replacement for free and proper actions are
``basic'' actions, which are characterised by the map
\[
G\times_{\s,G^0,\rg} X \to X\times X,\qquad
(g,x)\mapsto (gx,x),
\]
being a homeomorphism onto its image with the subspace topology
from~\(X\times X\).

Readers already familiar with the usual theory of locally compact
groupoids may read on and only turn to
Appendix~\ref{sec:preliminaries} in cases of doubt; they should note
that range and source maps of groupoids are assumed to be open,
whereas anchor maps of groupoid actions are not assumed open.  Less
experienced readers should read Appendix~\ref{sec:preliminaries}
first.

\begin{definition}
  \label{def:partial_equivalence}
  Let \(G\) and~\(H\) be topological groupoids.  A \emph{partial
    equivalence} from~\(H\) to~\(G\) is a topological space with
  \emph{anchor maps} \(\rg\colon X\to G^0\) and \(\s\colon X\to
  H^0\) and multiplication maps \(G^1\times_{\s,G^0,\rg} X\to X\)
  and \(X\times_{\s,H^0,\rg} H^1\to X\), which we write
  multiplicatively, that satisfy the following conditions:
  \begin{enumerate}[label=\textup{(P\arabic*)}]
  \item \label{enum:Peq1} \(\s(g\cdot x)=\s(x)\), \(\rg(g\cdot
    x)=\rg(g)\) for all \(g\in G^1\), \(x\in X\) with
    \(\s(g)=\rg(x)\), and \(\s(x\cdot h)=\s(h)\), \(\rg(x\cdot
    h)=\rg(x)\) for all \(x\in X\), \(h\in H^1\) with
    \(\s(x)=\rg(h)\);
  \item \label{enum:Peq2} associativity: \(g_1\cdot (g_2\cdot x)=
    (g_1\cdot g_2)\cdot x\), \(g_2\cdot (x\cdot h_1)= (g_2\cdot
    x)\cdot h_1\), \(x\cdot (h_1\cdot h_2)= (x\cdot h_1)\cdot h_2\)
    for all \(g_1,g_2\in G^1\), \(x\in X\), \(h_1,h_2\in H^1\) with
    \(\s(g_1)=\rg(g_2)\), \(\s(g_2)=\rg(x)\), \(\s(x)=\rg(h_1)\),
    \(\s(h_1)=\rg(h_2)\);
  \item \label{enum:Peq3} the following two maps are homeomorphisms:
    \begin{align*}
      G^1\times_{\s,G^0,\rg} X &\to X\times_{\s,H^0,\s} X,&\qquad
      (g,x)&\mapsto (x,g\cdot x),\\
      X\times_{\s,H^0,\rg} H^1 &\to X\times_{\rg,G^0,\rg} X,&\qquad
      (x,h)&\mapsto (x,x\cdot h);
    \end{align*}
  \item \label{enum:Peq4} \(\s\) and~\(\rg\) are open.
  \end{enumerate}
  The first two conditions say that~\(X\) is a \(G,H\)\nb-bispace.
  The only difference between a partial and a global equivalence is
  whether the anchor maps are assumed surjective or not: conditions
  \ref{enum:Peq1}--\ref{enum:Peq4} are the same as conditions
  \ref{enum:Eq1}--\ref{enum:Eq4} in Proposition~\ref{pro:equivalence}.
\end{definition}

We view a partial equivalence~\(X\) from~\(H\) to~\(G\) as a
generalised map from~\(H\) to~\(G\).  Indeed, there is a bicategory
with partial equivalences as arrows \(H\to G\)
(Theorem~\ref{the:bicategory_peq}).

\begin{definition}
  \label{def:invariant_subset}
  Let~\(G\) be a groupoid.  A subset \(U\subseteq G^0\) is
  \emph{\(G\)\nb-invariant} if \(\rg^{-1}(U)=\s^{-1}(U)\).  In this case,
  \(U\) and \(\rg^{-1}(U)=\s^{-1}(U)\) are the object and arrow
  spaces of a subgroupoid of~\(G\), which we denote by~\(G_U\).
\end{definition}

The canonical projection \(p\colon G^0\to G^0/G\) induces a bijection
between \(G\)\nb-invariant subsets \(U\subseteq G^0\) and subsets
\(p(U)\subseteq G^0/G\).  We are mainly interested in open invariant
subsets.  Since~\(p\) is open and continuous, open \(G\)\nb-invariant
subsets of~\(G^0\) correspond to open subsets of~\(G^0/G\).

\begin{lemma}
  \label{lem:partial_equivalence}
  Let \(G\) and~\(H\) be topological groupoids.  A partial
  equivalence~\(X\) from~\(H\) to~\(G\) is the same as an equivalence
  from~\(H_V\) to~\(G_U\) for open, invariant subsets \(U\subseteq
  G^0\), \(V\subseteq H^0\).  Here \(U=\rg(X)\), \(V=\s(X)\).
\end{lemma}

\begin{proof}
  Let \(U\subseteq G^0\) be \(G\)\nb-invariant.  A left
  \(G_U\)\nb-action is the same as a left \(G\)\nb-action for which
  the anchor map takes values in~\(U\) because \(G_U^0=U\) and
  \(G^1\times_{\s,G^0,\rg} X \cong G_U^1\times_{\s,G^0,\rg} X\) if
  \(\rg(X)\subseteq U\).  Thus the commuting actions of \(G_U\)
  and~\(H_V\) for an equivalence from~\(H_V\) to~\(G_U\) may also be
  viewed as commuting actions of \(G\) and~\(H\), respectively.
  This gives a partial equivalence (see
  Definition~\ref{def:partial_equivalence}).  Conversely, let~\(X\)
  be a partial equivalence.  Let \(U\defeq \rg(X)\subseteq G^0\) and
  \(V\defeq \s(X)\subseteq H^0\).  These are open subsets because
  \(\rg\) and~\(\s\) are open, and they are invariant
  by~\ref{enum:Peq1}.  The actions of \(G\) and~\(H\) are equivalent
  to actions of \(G_U\) and~\(H_V\), respectively.  After replacing
  \(G\) and~\(H\) by \(G_U\) and~\(H_V\), respectively, all
  conditions \ref{enum:Eq1}--\ref{enum:Eq5} in
  Proposition~\ref{pro:equivalence} hold; thus~\(X\) is an
  equivalence from~\(H_V\) to~\(G_U\).
\end{proof}

\begin{lemma}
  \label{lem:restrict_peq}
  Let~\(X\) be a partial equivalence from~\(H\) to~\(G\) and let
  \(U\subseteq G^0\) and \(V\subseteq H^0\) be invariant open
  subsets.  Then
  \[
  {}_U|X|_V \defeq \{x\in X\mid \rg(x)\in U,\ \s(x)\in V\}
  \]
  is again a partial equivalence from~\(H\) to~\(G\).
\end{lemma}

We also write \({}_U|X\) and~\(X|_V\) for \({}_U|X|_{H^0}\)
and~\({}_{G^0}|X|_V\), respectively.

\begin{proof}
  The subset~\({}_U|X|_V\) is open because \(\rg\) and~\(\s\) are
  continuous and \(U\) and~\(V\) are open, and it is invariant under
  the actions of \(G\) and~\(H\) because \(U\) and~\(V\) are
  invariant and the two anchor maps are either invariant or
  equivariant with respect to the two actions.  Hence we may
  restrict the actions of \(G\) and~\(H\) to~\({}_U|X|_V\).
  Conditions \ref{enum:Peq1}--\ref{enum:Peq2} and~\ref{enum:Peq4} in
  Definition~\ref{def:partial_equivalence} are inherited by an open
  invariant subspace.  The inverse to the first homeomorphism
  in~\ref{enum:Peq3} maps \({}_U|X|_V \times_{\s,H^0,\s} {}_U|X|_V\)
  into \(G^1\times_{\s,G^0,\rg} {}_U|X|_V\), and the inverse to the
  second one maps \({}_U|X|_V \times_{\rg,G^0,\rg} {}_U|X|_V\) into
  \({}_U|X|_V\times_{\s,H^0,\rg} H^1\).  Thus~\({}_U|X|_V\) also
  inherits~\ref{enum:Peq3} and is a partial equivalence from~\(H\)
  to~\(G\).
\end{proof}

Equivalences are partial equivalences, of course.
In particular, the \emph{identity equivalence}~\(G^1\) with
\(G\)~acting by left and right multiplication is also a partial
equivalence.

Let \(X\) and~\(Y\) be partial equivalences from~\(H\) to~\(G\) and
from~\(K\) to~\(H\), respectively.  Their composite is defined as for
global equivalences, and still denoted by~\(\times_H\):
\[
X\times_H Y \defeq X\times_{\s,H^0,\rg} Y\mathbin{/}
(x\cdot h,y)\sim (x,h\cdot y),
\]
equipped with the quotient topology and the induced actions of \(G\)
and~\(K\) by left and right multiplication.  The canonical map
\(X\times_{\s,H^0,\rg} Y \to X\times_H Y\) is a principal
\(H\)\nb-bundle for the \(H\)\nb-action defined by \((x,y)\cdot
h\defeq (x\cdot h,h^{-1}\cdot y)\); this follows from the general
theory in~\cite{Meyer-Zhu:Groupoids}.

\begin{example}
  \label{exa:iso_to_partial_equivalence}
  We associate an equivalence~\(H_f\) from~\(G\) to~\(H\) to a
  groupoid isomorphism \(f\colon G\to H\).  The functor~\(f\) consists
  of homeomorphisms \(f^i\colon G^i\to H^i\) for \(i=0,1\).  We take
  \(X=H^1\) with the usual left \(H\)\nb-action and the right
  \(G\)\nb-action by \(h\cdot g\defeq h\cdot f^1(g)\) for all \(h\in
  H^1\), \(g\in G^1\) with \(\s(h) = \rg(f^1(g)) = f^0(\rg(g))\); so
  the right anchor map is~\((f^0)^{-1}\circ\s = \s\circ (f^1)^{-1}\).

  We claim that an equivalence is of this form if and only if it is
  isomorphic to~\(H^1\) as a left \(H\)\nb-space.  Since \(H\backslash
  H^1\cong H^0\), the right anchor map gives a homeomorphism \(H^0\to
  G^0\) in this case; let \(f^0\colon G^0\to H^0\) be its inverse.
  The right action of \(g\in G^1\) on \(h\in H^1\) with
  \(\s(h)=f^0(\rg(g))\) must be of the form \(h\cdot g = h\cdot
  f^1(g)\) for a unique \(f^1(g)\in H^1\) with
  \(\rg(f^1(g))=f^0(\rg(g))\) and \(\s(f^1(g))= \s(h\cdot g) =
  f^0(\s(g))\).  It is routine to check that \(f^0\) and~\(f^1\) give
  a topological groupoid isomorphism.

  When do two isomorphisms \(f,\varphi\colon G\to H\) give isomorphic
  equivalences?  Let \(u\colon H^1_f\congto H^1_\varphi\) be an
  isomorphism.  Define a continuous map \(\sigma\colon G^0\to H^1\) by
  \(\sigma(x)\defeq u(1_{f^0(x)})\) for all \(x\in G^0\).  This
  satisfies \(\rg(\sigma(x))=f^0(x)\) and
  \(\s(\sigma(x))=\varphi^0(x)\) for all \(x\in G^0\) because~\(u\) is
  compatible with anchor maps.  Since~\(u\) is left
  \(H\)\nb-invariant, \(u(h) = u(h\cdot 1_{\s(h)}) = h\cdot
  (\sigma\circ (f^0)^{-1}\circ \s)(h)\) for all \(h\in H^1\),
  so~\(\sigma\) determines~\(u\).  The right \(G\)\nb-invariance
  of~\(u\) translates to \(\sigma(\rg(g))\cdot \varphi^1(g)
  = f^1(g)\cdot \sigma(\s(g))\) for all \(g\in G\).  Thus
  \begin{equation}
    \label{eq:equivalent_isomorphisms}
    \varphi^0(x) = \s(\sigma(x)),\qquad
    \varphi^1(g)
    = \sigma(\rg(g))^{-1} \cdot f^1(g)\cdot \sigma(\s(g)).
  \end{equation}
  Roughly speaking, \(f\) and~\(\varphi\) differ by an inner
  automorphism.

  Let an equivalence \(f\colon G\to H\) and a continuous map \(\sigma\colon
  G^0\to H^1\) with \(\rg(\sigma(x))=f^0(x)\) for all \(x\in H^0\) be
  given.  Assume that \(H^0\to G^0\), \(x\mapsto \s(\sigma(x))\), is a
  homeomorphism.  Then~\eqref{eq:equivalent_isomorphisms} defines an
  isomorphism \(\varphi\colon G\to H\) such that \(h\mapsto h\cdot
  \sigma((f^0)^{-1}(\s(h)))\) is an isomorphism between the equivalences \(H_f\)
  and~\(H_\varphi\).
\end{example}

\begin{example}
  \label{exa:minimal_no_partial}
  If \(G\) and~\(H\) are \emph{minimal} groupoids in the sense that
  \(G^0\) and~\(H^0\) have no proper open invariant subsets, then any
  partial equivalence is either empty or a full equivalence \(G\congto
  H\).  This holds, in particular, if \(G\) and~\(H\) are groups.
\end{example}

\begin{example}
  \label{exa:groups_pe}
  Any non-empty (partial) equivalence between two groups is isomorphic
  to one coming from a group isomorphism \(G\cong H\).  Indeed, since
  \(X/H\cong G^0\) and \(G\backslash X\cong H^0\) are a single point,
  both actions on~\(X\) are free and transitive.  Fix \(x_0\in X\).
  Since the actions are free and transitive and part of principal
  bundles, the maps \(G\to X\), \(g\mapsto g\cdot x_0\), and \(H\to
  X\), \(h\mapsto x_0\cdot h\), are homeomorphisms.  The composite map
  \(G\congto X\congto H\) is an isomorphism of topological groups.
  This isomorphism depends on the choice of~\(x_0\).  The isomorphisms
  \(G\congto H\) for different choices of~\(x_0\) differ by an inner
  automorphism.
\end{example}

\begin{lemma}
  \label{lem:peq_composition_associative_unital}
  The composition~\(\times_H\) is associative and unital with the
  identity equivalence as unit, up to the usual canonical bibundle
  isomorphisms
  \[
  (X\times_H Y)\times_K Z\cong X\times_H (Y\times_K Z),\qquad
  G^1\times_G X\cong X\cong X\times_H H^1.
  \]
\end{lemma}

\begin{proof}
  For global equivalences with arbitrary topological spaces, this is
  contained in \cite{Meyer-Zhu:Groupoids}*{Proposition 7.10}.  The proofs
  in~\cite{Meyer-Zhu:Groupoids} can be extended to the partial case as
  well.  Alternatively, we may reduce the partial to the global case
  by restricting our partial equivalences to global equivalences
  between open subgroupoids as in Lemma~\ref{lem:restrict_peq}.  This
  works because
  \[
  {}_U |(X\times_H Y)|_V \cong ({}_U|X)\times_H (Y|_V)
  \]
  for \(U\subseteq G^0\), \(V\subseteq K^0\) open and invariant and
  partial equivalences~\(X\) from~\(H\) to~\(G\) and~\(Y\) from~\(K\)
  to~\(H\).  Details are left to the reader.
\end{proof}

\begin{proposition}
  \label{pro:characterise_bibundle_isom}
  Let \(G\) and~\(H\) be topological groupoids.  Let \(X_1\)
  and~\(X_2\) be partial equivalences from~\(H\) to~\(G\).  There is
  no bibundle map \(X_1\to X_2\) unless \(\rg(X_1)\subseteq \rg(X_2)\)
  and \(\s(X_1)\subseteq \s(X_2)\).  Any \(G,H\)-bibundle map
  \(\varphi\colon X_1\to X_2\) is an isomorphism onto the open
  sub-bibundle \({}_{\rg(X_1)}|X_2 =X_2|_{\s(X_1)}\).  The
  map~\(\varphi\) is invertible if \(\rg(X_2)\subseteq\rg(X_1)\) or
  \(\s(X_2)\subseteq \s(X_1)\).  In this case, \(\rg(X_2)=\rg(X_1)\)
  and \(\s(X_2)= \s(X_1)\).
\end{proposition}

\begin{proof}
  Since \(\rg_{X_2}\circ\varphi=\rg_{X_1}\) and
  \(\s_{X_2}\circ\varphi=\s_{X_1}\), we must have \(\rg(X_1)\subseteq
  \rg(X_2)\) and \(\s(X_1)\subseteq \s(X_2)\) if there is a bibundle
  map \(\varphi\colon X_1\to X_2\).  Assume this from now on.  The
  image of a
  bibundle map is contained in \({}_{\rg(X_1)}|X_2\) and
  in~\(X_2|_{\s(X_1)}\).  Since \(\rg(X_1)\subseteq \rg(X_2)\) and
  \(\s(X_1)\subseteq \s(X_2)\), we have \(\rg({}_{\rg(X_1)}|X_2)=
  \rg(X_1)\) and \(\s(X_2|_{\s(X_1)})=\s(X_1)\).  All remaining
  assertions now follow once we prove that a bibundle map
  \(\varphi\colon X_1\to X_2\) is invertible if \(\rg(X_1)=\rg(X_2)\)
  or \(\s(X_1)=\s(X_2)\).  We treat the case \(\rg(X_2)=\rg(X_1)\);
  the other one is proved in the same way, exchanging left and right.

  Since~\(X_i\) is a partial equivalence, it is a principal
  \(H\)\nb-bundle over \(X_i/H\cong \rg(X_i)\).  The map~\(\varphi\)
  induces a homeomorphism on the base spaces because
  \(\rg(X_2)=\rg(X_1)\) both carry the subspace topology from~\(G^0\).
  Hence~\(\varphi\) is a homeomorphism by
  \cite{Meyer-Zhu:Groupoids}*{Proposition 5.9}.
\end{proof}

In particular, the restricted multiplication maps \(G^1_U\times_H
X\subseteq G^1\times_G X\to X\) and \(X\times_H H^1_V \subseteq
X\times_H H^1\to X\) are bibundle maps.
Proposition~\ref{pro:characterise_bibundle_isom} shows that they
induce bibundle isomorphisms
\begin{equation}
  \label{eq:restrict_as_product}
  G^1_U \times_G X \cong {}_U|X,\qquad
  X \times_H H^1_V \cong X|_V.
\end{equation}

Partial equivalences carry extra structure similar to an inverse
semigroup.  The adjoint operation is the following:

\begin{definition}
  \label{def:dual_equivalence}
  Given a partial equivalence \(X\) from~\(H\) to~\(G\), we define the
  \emph{dual partial equivalence}~\(X^*\) by exchanging the left and
  right actions on~\(X\).  More precisely, \(X^*\)~is~\(X\) as a
  space, the anchor maps \(\rg^*\colon X^*\to H^0\) and \(\s^*\colon
  X^*\to G^0\) are \(\rg^*=\s_X\) and \(\s^*=\rg_X\), and the
  left \(H\)- and right \(G\)\nb-actions are defined by \(h\cdot^*
  x=x\cdot h^{-1}\) and \(x\cdot^* g\defeq g^{-1}\cdot x\),
  respectively.
\end{definition}

If~\(X\) gives an equivalence from~\(H_V\) to~\(G_U\) for open
invariant subsets \(U\subseteq G^0\), \(V\subseteq H^0\), then~\(X^*\)
gives the ``inverse'' equivalence from~\(G_U\) to~\(H_V\).

The following properties of duals are trivial:
\begin{itemize}
\item naturality: a bibundle map \(X\to Y\) induces a bibundle map
  \(X^*\to Y^*\);
\item \((X^*)^*=X\);
\item there is a natural isomorphism \(\sigma\colon (X\times_H
  Y)^*\cong Y^*\times_H X^*\), \((x,y)\mapsto (y,x)\), with
  \(\sigma^2=\Id\).
\end{itemize}

Let \(\Map(Y_1,Y_2)\) be the space of bibundle maps between two
partial equivalences \(Y_1, Y_2\) from~\(H\) to~\(G\).

\begin{proposition}
  \label{pro:properties_of_dual}
  Let~\(X\) be a partial equivalence from~\(H\) to~\(G\).
  Then there are natural isomorphisms
  \[
  X\times_H X^*\cong G_{\rg(X)}^1,\qquad
  X^*\times_G X\cong H_{\s(X)}^1
  \]
  that make the following diagrams of isomorphisms commute:
  \begin{equation}
    \label{eq:XXX_to_X}
    \begin{gathered}
      \begin{tikzpicture}[baseline=(current bounding box.west)]
        \matrix (m) [cd,column sep=1em] {
          X\times_H X^*\times_G X & X\times_H H^1_{\s(X)} \\
          G^1_{\rg(X)} \times_G X & X,\\
        };
        \draw[cdar] (m-1-1) -- (m-1-2);
        \draw[cdar] (m-1-1) -- (m-2-1);
        \draw[cdar] (m-2-1) -- (m-2-2);
        \draw[cdar] (m-1-2) -- (m-2-2);
      \end{tikzpicture}\\
      \begin{tikzpicture}[baseline=(current bounding box.west)]
        \matrix (m) [cd,column sep=1em] {
          X^*\times_G X\times_H X^* & X^*\times_G G^1_{\rg(X)} \\
          H^1_{\s(X)} \times_G X^* & X^*.\\
        };
        \draw[cdar] (m-1-1) -- (m-1-2);
        \draw[cdar] (m-1-1) -- (m-2-1);
        \draw[cdar] (m-2-1) -- (m-2-2);
        \draw[cdar] (m-1-2) -- (m-2-2);
      \end{tikzpicture}
    \end{gathered}
  \end{equation}
  If~\(K\) is another groupoid and \(Y\) and~\(Z\) are partial
  equivalences from~\(K\) to~\(G\) and from~\(K\) to~\(H\),
  respectively, with \(\rg(Y) \subseteq \rg(X)\) and \(\rg(Z)
  \subseteq \s(X)\), then there are natural isomorphisms
  \begin{align*}
    \Map(X\times_H Z,Y) &\cong \Map(Z, X^*\times_G Y),\\
    \Map(Y,X\times_H Z) &\cong \Map(X^*\times_G Y,Z).
  \end{align*}
  Both map the subsets of bibundle isomorphisms onto each other.
\end{proposition}

\begin{proof}
  Lemma~\ref{lem:partial_equivalence} shows that~\(X\) is an
  equivalence from~\(H_{\s(X)}\) to~\(G_{\rg(X)}\).  Hence the usual
  theory of groupoid equivalence gives canonical isomorphisms
  \(X\times_H X^*\cong G_{\rg(X)}^1\) and \(X^*\times_G X\cong
  H_{\s(X)}^1\).  The first one maps the class of \((x_1,x_2)\) with
  \(\s(x_1)=\s(x_2)\) to the unique \(g\in G^1\) with \(x_1 = g\cdot
  x_2\).  In particular, it maps \([x,x]\mapsto 1_{\rg(x)}\).  The
  second one maps the class of \((x_1,x_2)\) with
  \(\rg(x_1)=\rg(x_2)\) to the unique \(h\in H^1\) with \(x_2 =
  x_1\cdot h\).  In particular, it maps \([x,x]\mapsto 1_{\s(x)}\).
  Then the composite isomorphisms \(X\times_H X^*\times_G X\to X\) and
  \(X^*\times_G X\times_H X^*\to X^*\) map \([x,x,x]\mapsto x\),
  respectively.  Since any element in \(X\times_H X^*\times_G X\) or
  \(X^*\times_G X\times_H X^*\) has a representative of the form
  \((x,x,x)\), we get the two commuting diagrams
  in~\eqref{eq:XXX_to_X}.

  The assumption \(\rg(Y) \subseteq \rg(X)\) implies \(\s(X^*\times_G
  Y) = \s(Y)\) because for any \(y\in Y\) there is \(x\in X^*\) with
  \((x,y)\in X^*\times_G Y\).  Similarly, \(\rg(Z) \subseteq \s(X)\)
  implies \(\s(X\times_H Z) = \s(Z)\).  By
  Proposition~\ref{pro:characterise_bibundle_isom}, a bibundle map
  \(X\times_H Z\to Y\) exists only if \(\s(Z)\subseteq \s(Y)\),
  and then it is an isomorphism onto \(Y|_{\s(Z)}\); and a bibundle map
  \(Z\to X^*\times_G Y\) exists only if \(\s(Z)\subseteq
  \s(Y)\), and then it is an isomorphism onto \(X^*\times_G Y|_{\s(Z)}\).
  Thus we may as well replace~\(Y\) by~\(Y|_{\s(Z)}\) to achieve
  \(\s(Y)=\s(Z)\); then all bibundle maps \(X\times_H Z\to Y\) or
  \(Z\to X^*\times_G Y\) are bibundle isomorphisms.  The second
  isomorphism reduces in a similar way to the case where also
  \(\s(Y)=\s(Z)\) and where we are dealing only with bibundle
  isomorphisms.

  A bibundle map \(\varphi\colon X\times_H Z \to Y\) induces
  \(\Id_{X^*}\times_G\varphi\colon X^*\times_G X\times_H Z\to
  X^*\times_G Z\); we compose this with the natural isomorphism
  \[
  X^*\times_G X \times_H Z \cong H^1_{\s(X)}\times_H Z \cong {}_{\s(X)}|Z
  = Z
  \]
  to get a bibundle map \(Z\to X^*\times_G Y\); here we used
  \(\s(X)\supseteq \rg(Z)\).  We claim that this construction gives
  the desired bijection between \(\Map(X\times_H Z,Y)\) and \(\Map(Z,
  X^*\times_G Y)\).  Since composing with an isomorphism is certainly
  a bijection, it remains to show that
  \[
  \Map(X\times_H Z,Y) \to
  \Map(X^*\times_G X\times_H Z, X^*\times_G Y),\quad
  \varphi\mapsto \Id_{X^*}\times_G\varphi,
  \]
  is bijective.  Since \(X\times_H X^*\cong G^1_{\rg(X)}\) and
  \(\rg(X)\supseteq \rg(Y)\), we have natural isomorphisms \(X\times_H
  X^*\times_G Y\cong Y\) and \(X\times_H X^*\times_G X\times_H Z\cong
  X\times_H Z\).  Naturality means that they intertwine
  \(\varphi\mapsto \Id_{X\times_H X^*}\times_G \varphi\)
  and~\(\varphi\).  Since \(\Id_{X\times_H X^*}\times_G \varphi =
  \Id_X\times_H \Id_{X^*}\times_G \varphi\), we see that
  \(\varphi\mapsto \Id_{X^*}\times_G \varphi\) is injective and has
  \(\psi\mapsto \Id_X\times_H \psi\) for \(\psi\colon Z\to X^*\times_G
  Y\) as a one-sided inverse.  The same argument also shows that
  \(\psi\mapsto \Id_X\times_H \psi\) is injective, so both
  constructions are bijective.
\end{proof}

Applying duality, we also get bijections \(\Map(Z^*\times_H X^*,Y^*)
\cong \Map(Z^*, Y^*\times_G X)\) and \(\Map(Y^*,Z^*\times_H X^*) \cong
\Map(Y^*\times_G X, Z^*)\) under the same hypotheses.

The canonical isomorphisms
\begin{equation}
  \label{eq:dual_like_inverse_sg}
  X\times_H X^*\times_G X \cong X,\qquad
  X^*\times_G X\times_H X^* \cong X^*
\end{equation}
from Proposition~\ref{pro:properties_of_dual} characterise~\(X^*\)
uniquely in the following sense:

\begin{proposition}
  \label{pro:dual_pe_unique}
  Let \(X\) and~\(Y\) be partial equivalences from~\(H\) to~\(G\) and
  from~\(G\) to~\(H\), respectively.  If there are bibundle
  isomorphisms
  \[
  X\times_H Y\times_G X \cong X,\qquad
  Y\times_H X\times_G Y \cong Y,
  \]
  then there is a unique bibundle isomorphism \(X^*\cong Y\) such that
  the composite map
  \begin{equation}
    \label{eq:composite_determines_dual}
    X\cong X\times_H X^*\times_G X \cong X\times_H Y\times_G X \cong X
  \end{equation}
  is the identity map.
\end{proposition}

\begin{proof}
  When we multiply the inverse of the isomorphism \(X\times_H
  Y\times_G X \cong X\) on both sides with~\(X^*\) and
  use~\eqref{eq:restrict_as_product}, we get an isomorphism
  \begin{multline*}
    X^*\cong X^*\times_G X\times_H X^*
    \cong X^*\times_G X\times_H Y\times_G X\times_H X^*
    \\\cong H_{\s(X)}^1 \times_H Y\times_G G_{\rg(X)}^1
    \cong {}_{\s(X)}|Y|_{\rg(X)}.
  \end{multline*}
  This implies \(\s(X)=\rg(X^*)\subseteq \rg(Y)\) and
  \(\rg(X)=\s(X^*)\subseteq \s(Y)\) by
  Proposition~\ref{pro:characterise_bibundle_isom}.  Exchanging \(X\)
  and~\(Y\), the isomorphism \(Y\times_H X\times_G Y \cong Y\) gives
  \(\s(Y)\subseteq \rg(X)\) and \(\rg(Y)\subseteq \s(X)\).  Hence
  \(\rg(Y)=\s(X)\) and \(\s(Y)=\rg(X)\), so \({}_{\s(X)}|Y|_{\rg(X)} =
  Y\).  This gives an isomorphism \(\alpha\colon X^*\congto Y\).

  A diagram chase using the commuting diagrams in~\eqref{eq:XXX_to_X}
  shows that the composite of the map \(X\times_H X^*\times_G X\to
  X\times_H Y\times_G X\) induced by the isomorphism~\(\alpha\) and
  the given isomorphism \(X\times_H Y\times_G X\to Y\) (which we used
  to construct~\(\alpha\)) is the canonical map \(X\times_H
  X^*\times_G X\to X\) as in~\eqref{eq:dual_like_inverse_sg}.  Hence
  the composite in~\eqref{eq:composite_determines_dual} is the
  identity map for the isomorphism~\(\alpha\).

  The isomorphisms in Proposition~\ref{pro:properties_of_dual} give a
  canonical bijection
  \begin{multline*}
    \Map(X^*,Y)
    \cong \Map(X^*\times_G X \times_H X^*,Y)
    \cong \Map(X\times_H X^*,X\times_H Y)
    \\\cong \Map(X,X \times_H Y \times_G X)
    \cong \Map(X,X).
  \end{multline*}
  Inspection shows that it maps an isomorphism \(X^*\congto Y\) to the
  composite map in~\eqref{eq:composite_determines_dual}.  Hence there
  is only one isomorphism \(X^*\congto Y\) for which the composite map
  in~\eqref{eq:composite_determines_dual} is the identity map.
\end{proof}

\begin{proposition}
  \label{pro:idempotents}
  Let~\(X\) be a partial equivalence from~\(G\) to itself and let
  \(\mu\colon X\times_G X\to X\) be a bibundle isomorphism.  Then
  there is a unique isomorphism \(\varphi\colon X\congto G^1_U\) for
  an open \(G\)\nb-invariant subset \(U\subseteq G^0\) such that the
  following diagram commutes:
  \begin{equation}
    \label{eq:idempotent_mu_diagram}
    \begin{tikzpicture}[baseline=(current bounding box.west)]
      \matrix (m) [cd] {
        X\times_G X& X\\
        G^1_U\times_G G^1_U&G^1_U,&\mu_0(g_1,g_2)=g_1\cdot g_2.\\
      };
      \draw[cdar] (m-1-1) -- node {\(\mu\)} (m-1-2);
      \draw[cdar] (m-2-1) -- node {\(\mu_0\)} (m-2-2);
      \draw[cdar] (m-1-1) -- node[swap] {\(\varphi\times_G\varphi\)} (m-2-1);
      \draw[cdar] (m-1-2) -- node {\(\varphi\)} (m-2-2);
    \end{tikzpicture}
  \end{equation}
  Hence \(\rg(X)=\s(X)\) and~\(\mu\) is associative.
\end{proposition}

\begin{proof}
  The isomorphism~\(\mu\) induces an isomorphism
  \[
  X\times_G X\times_G X
  \xrightarrow{\mu\times_G\Id_X} X\times_G X
  \xrightarrow{\mu} X
  \]
  Hence \(Y=X\) satisfies the two conditions in
  Proposition~\ref{pro:properties_of_dual} that ensure \(X=Y\cong
  X^*\).  This gives an isomorphism \(\varphi\colon X\cong X\times_G X
  \cong X\times_G X^*\cong G^1_{\rg(X)}\).  Since~\(\varphi\) is a
  bibundle map, the diagram~\eqref{eq:idempotent_mu_diagram} commutes
  if and only if~\(\mu\) is the composite map
  \[
  X\times_G X \xrightarrow{\varphi\times_G \Id_X}
  G^1_{\rg(X)} \times_G X \cong X,
  \]
  where the map \(G^1_{\rg(X)} \times_G X \cong X\) is the left
  multiplication map, \([g,x]\mapsto g\cdot x\).  Sending an
  isomorphism \(\varphi\colon X\to G^1_{\rg(X)} \cong X\times_G X^*\)
  to this composite map is one of the bijections in
  Proposition~\ref{pro:properties_of_dual}, namely, the first one for
  \(X=Y=Z\):
  \[
  \Map(X,G^1_{\s(X)}) \cong \Map(X,X^*\times_G X)
  \cong \Map(X\times_G X, X).
  \]
  Hence there is exactly one isomorphism~\(\varphi\) that corresponds
  under this bijection to~\(\mu\).
\end{proof}

Proposition~\ref{pro:properties_of_dual} implies that isomorphism
classes of partial equivalences from~\(G\) to itself form an inverse
semigroup~\(\widetilde{\Peq}(G)\).  The idempotents in this inverse
semigroup are in bijection with \(G\)\nb-invariant open subsets
of~\(G^0\) by Proposition~\ref{pro:idempotents}.  These are, in turn,
in bijection with open subsets of the orbit space~\(G^0/G\) by the
definition of the quotient topology on~\(G^0/G\).  These also
correspond to the idempotents of the inverse
semigroup~\(\PHomeo(G^0/G)\) of partial homeomorphisms of the
topological space~\(G^0/G\).

A partial equivalence~\(X\) from~\(H\) to~\(G\) induces a partial
homeomorphism
\[
X_*\colon H^0/H\subseteq \s(X)\to \rg(X)\subseteq G^0/G
\]
by \(X_*([h]) = [g]\) if there is \(x\in X\) with \(\s(x)\in[h]\),
\(\rg(x)\in[g]\).  If~\(Y\) is another partial equivalence from~\(K\)
to~\(H\), then \((X\times_H Y)_* = X_*\circ Y_*\) by definition.  This
gives a canonical homomorphism of inverse semigroups
\[
\widetilde{\Peq}(G)\to \PHomeo(G^0/G).
\]

\begin{remark}
  \label{rem:PHomeo_versus_Peq}
  The homomorphism \(\widetilde{\Peq}(G)\to \PHomeo(G^0/G)\) is
  neither injective nor surjective in general, although it is always
  an isomorphism on the semilattice of idempotents.  Consider, for
  instance, the disjoint union \(G=\Z/3\sqcup \{\textup{pt}\}\).  This
  groupoid is a group bundle, and~\(G^0/G\) has two points.  The
  partial homeomorphism that maps one point to the other does not lift
  to a partial equivalence because the stabilisers are not the same
  and equivalences must preserve the stabiliser groups.  The
  group~\(\Z/3\) has non-inner automorphisms, so there are
  non-isomorphic partial equivalences of~\(G\) defined on~\(\Z/3\)
  that induce the same partial homeomorphism on~\(G^0/G\).
\end{remark}

In our definition of an inverse semigroup action
(see Sections~\ref{sec:ActionsGroupoidsByIso} and~\ref{sec:S_acts_on_Cstar} below), certain isomorphisms
of partial equivalences are a crucial part of the data.  We could not
construct transformation groupoids and Fell bundles without them.  If
we identify isomorphic partial equivalences as above, then we can no
longer talk about two isomorphisms of partial equivalences being
equal.  The correct way to take into account isomorphisms of partial
equivalences is through a bicategory (see \cites{Benabou:Bicategories,
  Leinster:Basic_Bicategories, Buss-Meyer-Zhu:Higher_twisted}).  The
following remarks are intended for readers familiar with bicategories.

Our bicategory has topological groupoids as objects and partial
equivalences as arrows.  Let \(G\) and~\(H\) be topological groupoids
and let \(X_1\) and~\(X_2\) be partial equivalences from~\(H\)
to~\(G\).  As \(2\)\nb-arrows \(X_1\Rightarrow X_2\), we take all
\(G,H\)-bibundle isomorphisms \(X_1\to X_2\), so all \(2\)\nb-arrows
are invertible.  The vertical product of \(2\)\nb-arrows is the
composition of bibundle maps.  Unit \(2\)\nb-arrows are identity maps
on partial equivalences.  The composition of arrows is~\(\times_H\).
The unit arrow on a topological groupoid~\(G\) is~\(G^1\) with the
standard bibundle structure.
Lemma~\ref{lem:peq_composition_associative_unital} provides invertible
\(2\)\nb-arrows
\[
(X\times_H Y)\times_K Z\Rightarrow X\times_H (Y\times_K Z),\qquad
G^1\times_G X\Rightarrow X\Leftarrow X\times_H H^1,
\]
which we take as associator and left and right unit transformations.
Let \(X_1,X_2\) be partial equivalences from~\(H\) to~\(G\) and let
\(Y_1,Y_2\) be partial equivalences from~\(K\) to~\(H\).  The
horizontal product of two bibundle maps \(f\colon X_1\to X_2\) and
\(g\colon Y_1\to Y_2\) is \(f\times_H g\colon X_1\times_H Y_1\to
X_2\times_H Y_2\).

\begin{theorem}
  \label{the:bicategory_peq}
  The data above defines a bicategory~\(\Peq\).
\end{theorem}

\begin{proof}
  It is routine to check that partial equivalences from~\(H\) to~\(G\)
  with bibundle maps between them form a category \(\Cat(G,H)\) for
  the vertical product of bibundle maps, and that the composition of
  partial equivalences with the horizontal product of bibundle maps is
  a functor \(\Cat(G,H)\times\Cat(H,K)\to\Cat(G,K)\).  The associator
  and both unit transformations are natural isomorphisms of functors;
  the associator is clearly compatible with unit transformations and
  makes the usual pentagon commute, see
  \cite{Leinster:Basic_Bicategories}*{p.~2}.
\end{proof}

\begin{remark}
  \label{rem:groupoid_action_bicategory}
  We still get a bicategory if we allow all bibundle maps as
  \(2\)\nb-arrows.  We restrict to invertible \(2\)\nb-arrows to get
  the correct notion of inverse semigroup actions below.
\end{remark}

An arrow \(f\colon x\to y\) in a bicategory is called an
\emph{equivalence} if there are an arrow \(g\colon y\to x\) and
invertible \(2\)\nb-arrows \(g\circ f\Rightarrow \Id_x\) and \(f\circ
g\Rightarrow \Id_y\).  The equivalences in~\(\Peq\) are precisely the
global bibundle equivalences.

The duality \(X\mapsto X^*\) with the canonical flip maps \((X\times_H
Y)^* \congto Y^*\times_H X^*\) gives a functor \(I\colon
\Peq\to\Peq^\op\) with \(I^2=\Id_{\Peq}\).  It seems useful to
formalise the properties of this functor and look for examples in more
general bicategories.  But we shall not go into this question here.

\section{Inverse semigroup actions on groupoids}
\label{sec:ActionsGroupoidsByIso}

We give two equivalent definitions for actions of inverse semigroups
on topological groupoids by partial equivalences.  The first is
exactly what it promises to be.  The second, more elementary, definition does
not mention groupoids or partial equivalences.

Let~\(S\) be an inverse semigroup with unit~\(1\).  Let~\(G\) be a
topological groupoid.

\begin{definition}
  \label{def:S_act_groupoid}
  An \emph{action} of~\(S\) on~\(G\) \emph{by partial equivalences}
  consists of
  \begin{itemize}
  \item partial equivalences~\(X_t\) from~\(G\) to~\(G\) for
    \(t\in S\);

  \item bibundle isomorphisms \(\mu_{t,u}\colon X_t \times_G X_u
    \congto X_{tu}\) for \(t,u\in S\);
  \end{itemize}
  satisfying
  \begin{enumerate}[label=\textup{(A\arabic*)}]
  \item \label{enum:APE2} \(X_1\) is the identity equivalence~\(G^1\)
    on~\(G\);
  \item \label{enum:APE3} \(\mu_{t,1}\colon X_t\times_G G^1\congto
    X_t\) and \(\mu_{1,u}\colon G^1\times_G X_u\congto X_u\) are the
    canonical isomorphisms, that is, the left and right
    \(G\)\nb-actions, for all \(t,u\in S\);
  \item \label{enum:APE4} associativity: for all \(t,u,v\in S\), the
    following diagram commutes:
    \[
    \begin{tikzpicture}[baseline=(current bounding box.west)]
      \node (1) at (0,1) {\((X_t\times_G X_u) \times_G X_v\)};
      \node (1a) at (0,0) {\(X_t\times_G (X_u \times_G X_v)\)};
      \node (2) at (5,1) {\(X_{tu} \times_G X_v\)};
      \node (3) at (5,0) {\(X_t\times_G X_{uv}\)};
      \node (4) at (7,.5) {\(X_{tuv}\)};
      \draw[<->] (1) -- node[swap] {ass} (1a);
      \draw[cdar] (1) -- node {\(\mu_{t,u}\times_G \Id_{X_v}\)} (2);
      \draw[cdar] (1a) -- node[swap] {\(\Id_{X_t}\times_G\mu_{u,v}\)} (3);
      \draw[cdar] (3.east) -- node[swap] {\(\mu_{t,uv}\)} (4);
      \draw[cdar] (2.east) -- node {\(\mu_{tu,v}\)} (4);
    \end{tikzpicture}
    \]
  \end{enumerate}
  If~\(S\) has a zero object~\(0\), then we may also ask
  \(X_0=\emptyset\).
\end{definition}

\begin{remark}
  \label{rem:adjoin_01}
  Let~\(S\) be an inverse semigroup possibly without~\(1\).  We may
  add a unit \(1\) formally and extend the multiplication by \(1\cdot s =
  s=s\cdot 1\) for all \(s\in S\cup\{1\}\).  If partial
  equivalences~\((X_t)_{t\in S}\) and bibundle
  isomorphisms~\((\mu_{t,u})_{t,u\in S}\) are given satisfying
  associativity for all \(t,u,v\in S\), then we may extend this
  uniquely to an action of \(S\cup\{1\}\): we put \(X_1\defeq G^1\)
  and let \(\mu_{t,1}\) and~\(\mu_{1,u}\) be the right and left
  \(G\)\nb-action, respectively.  The associativity condition is
  trivial if one of \(t,u,v\) is~\(1\), so associativity holds for all
  \(t,u,v\in S\cup\{1\}\).  As a result, an action of \(S\cup\{1\}\)
  by partial equivalences is the same as \((X_t)_{t\in S}\) and
  \((\mu_{t,u})_{t,u\in S}\) satisfying only
  Condition~\ref{enum:APE4}.

  Similarly, we may add a zero~\(0\) to~\(S\) and extend the
  multiplication by \(0\cdot s=0=s\cdot 0\) for all \(s\in
  S\cup\{0\}\).  We extend an \(S\)\nb-action by \(X_0\defeq
  \emptyset\), so that \(X_0\times_G X_t=\emptyset=X_t\times_G X_0\),
  leaving no choice for the maps \(\mu_{t,0},\mu_{0,u}\colon
  \emptyset\to\emptyset\).  This gives an action of \(S\cup\{0\}\)
  with \(X_0=\emptyset\).

  If \(0,1\in S\) and we ask no conditions on \(X_0\) and~\(X_1\),
  then \(\rg(X_t),\s(X_t)\subseteq \rg(X_1)=\s(X_1)\) for all \(t\in
  S\), and~\(X_t\) restricted to \(\rg(X_0)=\s(X_0)\) is the trivial
  action where all~\(X_t\) act by the identity equivalence.  Hence all
  the action is on the locally closed, invariant subset
  \(\rg(X_1)\setminus \rg(X_0)\subseteq G^0\).  The conditions on
  \(X_0\) and~\(X_1\) merely rule out such degeneracies.
\end{remark}

\begin{remark}
  \label{rem:bicategory_action}
  An inverse semigroup may be viewed as a special kind of category
  with only one object, which is also a very special kind of
  bicategory.  An inverse semigroup action by partial equivalences is
  exactly the same as a functor from this category to the
  bicategory~\(\Peq\) of partial equivalences
  (see~\cite{Leinster:Basic_Bicategories}).
\end{remark}

\begin{lemma}
  \label{lem:domains_invsg-action}
  For an inverse semigroup action \((X_t,\mu_{t,u})\), we have
  \(\rg(X_t)=\rg(X_{tt^*}) = \s(X_{tt^*}) = \s(X_{t^*})\) and
  \(\s(X_t)=\s(X_{t^*t})=\rg(X_{t^*t}) = \rg(X_{t^*})\) for each
  \(t\in S\).
\end{lemma}

\begin{proof}
  If \(e\in S\) idempotent, then Proposition~\ref{pro:idempotents}
  applied to the isomorphism \(\mu_{e,e}\colon X_e\times_G X_e\cong
  X_e\) gives \(\rg(X_e)=\s(X_e)\).  The existence of an isomorphism
  \(\mu_{t,t^*}\colon X_t\times_G X_{t^*} \cong X_{tt^*}\) implies
  \(\rg(X_t)\supseteq \rg(X_{tt^*})\) and \(\s(X_{t^*})\supseteq
  \s(X_{tt^*})\).  Similarly, the isomorphism \(\mu_{tt^*,t}\) gives
  \(\rg(X_{tt^*})\supseteq \rg(X_t)\), and \(\mu_{t,t^*t}\) gives
  \(\s(X_{t^*t})\supseteq \s(X_t)\).  Now everything follows.
\end{proof}

\begin{definition}
  \label{def:S_act_groupoid_simple}
  Let~\(S\) be an inverse semigroup with unit.  A
  \emph{simplified action} of~\(S\) on a topological groupoid
  consists of
  \begin{itemize}
  \item a topological space~\(G^0\);
  \item topological spaces~\(X_t\) for \(t\in S\);
  \item continuous maps \(\s,\rg\colon X_t\to G^0\);
  \item continuous maps
    \[
    \mu_{t,u}\colon X_t \times_{\s,G^0,\rg} X_u \to X_{tu},
    \qquad (x,y)\mapsto x\cdot y,
    \]
    for \(t,u\in S\);
  \end{itemize}
  satisfying
  \begin{enumerate}[label=\textup{(S\arabic*)}]
  \item \label{e:SAS1} \(\s(x\cdot y)=\s(y)\), \(\rg(x\cdot y)=\rg(x)\) for all
    \(t,u\in S\), \(x\in X_t\), \(y\in X_u\) with \(\s(x)=\rg(y)\);
  \item \label{e:SAS2} \(\rg\colon X_t\to G^0\) and~\(\s\colon X_t\to G^0\) are
    open for all \(t\in S\);
  \item \label{e:SAS3} the maps \(\rg,\s\colon X_1\to G^0\) are surjective;
  \item \label{e:SAS4} \(\mu_{t,u}\) is surjective for each \(t,u\in S\);
  \item \label{e:SAS5} the map
    \[
    X_t\times_{\s,G^0,\rg} X_u\to X_u\times_{\s,G^0,\s} X_{tu},
    \qquad (x,y)\mapsto (y,x\cdot y),
    \]
    is a homeomorphism if \(t=1\) and \(u\in S\);
  \item \label{e:SAS6} the map
    \[
    X_t\times_{\s,G^0,\rg} X_u\to X_t\times_{\rg,G^0,\rg} X_{tu},
    \qquad (x,y)\mapsto (x,x\cdot y),
    \]
    is a homeomorphism if \(t\in S\) and \(u=1\);
  \item \label{e:SAS7} for all \(t,u,v\in S\), the following diagram commutes:
    \begin{equation}
      \label{eq:associativity}
      \begin{tikzpicture}[xscale=1.2,yscale=2,baseline=(current bounding box.west)]
        \node (1) at (0,1) {\((X_t\times_{\s,G^0,\rg} X_u) \times_{\s,G^0,\rg} X_v\)};
        \node (1a) at (0,0) {\(X_t\times_{\s,G^0,\rg} (X_u \times_{\s,G^0,\rg} X_v)\)};
        \node (2) at (5,1) {\(X_{tu} \times_{\s,G^0,\rg} X_v\)};
        \node (3) at (5,0) {\(X_t\times_{\s,G^0,\rg} X_{uv}\)};
        \node (4) at (6,.5) {\(X_{tuv}\)};
        \draw[<->] (1) -- node[swap] {ass} (1a);
        \draw[cdar] (1) -- node {\(\mu_{t,u}\times_{\s,G^0,\rg} \Id_{X_v}\)} (2);
        \draw[cdar] (1a) -- node[swap] {\(\Id_{X_t}\times_{\s,G^0,\rg}\mu_{u,v}\)} (3);
        \draw[cdar] (3) -- node[inner sep=0pt] {\(\mu_{t,uv}\)} (4);
        \draw[cdar] (2) -- node[inner sep=0pt,swap] {\(\mu_{tu,v}\)} (4);
      \end{tikzpicture}
    \end{equation}
  \end{enumerate}
  If~\(S\) has a zero element, we may also ask \(X_0=\emptyset\).
\end{definition}

This definition is more elementary because it does not mention
groupoids or partial equivalences.  It seems less elegant than
Definition~\ref{def:S_act_groupoid}, but is simpler because much of
the complexity of Definition~\ref{def:S_act_groupoid} is hidden in the
conditions \ref{enum:Peq1}--\ref{enum:Peq4} defining partial
equivalences of topological groupoids.

It is clear that an inverse semigroup action by partial equivalences
gives a simplified action: forget the multiplication on~\(G^1\) and
the left and right actions of~\(G\) on the spaces~\(X_t\).  The
isomorphisms in~\ref{e:SAS5} for \(t=1\) and in~\ref{e:SAS6} for
\(u=1\) are those in~\ref{enum:Peq3}, and all other conditions in
Definition~\ref{def:S_act_groupoid_simple} are evident.  The converse
is more remarkable:

\begin{proposition}
  \label{pro:partial_actions_simplify}
  Any simplified inverse semigroup action on groupoids comes from a
  unique action by partial equivalences.  Thus actions and simplified
  actions of inverse semigroups by partial equivalences are
  equivalent.  Furthermore, the maps in \ref{e:SAS5} and~\ref{e:SAS6}
  are isomorphisms and the maps~\(\mu_{t,u}\) are open for all
  \(t,u\in S\).
\end{proposition}

\begin{proof}
  The spaces \(G^0\) and \(G^1\defeq X_1\) with range and source maps
  \(\rg\) and~\(\s\) and multiplication~\(\mu_{1,1}\) satisfy the
  conditions \ref{enum:Gr1}--\ref{enum:Gr4} in
  Proposition~\ref{pro:Top_open_groupoid} because these are special
  cases of our conditions \ref{e:SAS1}--\ref{e:SAS7}.  Hence this data
  defines a topological groupoid.  Similarly, the anchor maps
  \(\rg\colon X_t\to G^0\) and \(\s\colon X_t\to G^0\) and the
  multiplication maps \(\mu_{1,t}\) and~\(\mu_{t,1}\) satisfy
  conditions \ref{enum:Peq1}--\ref{enum:Peq4} in
  Definition~\ref{def:partial_equivalence} and thus turn~\(X_t\)
  into a partial equivalence from~\(G\) to itself.

  Let \(t,u\in S\).  The associativity of the maps~\(\mu\) for
  \(t,1,u\), \(1,t,u\) and \(t,u,1\) implies that~\(\mu_{t,u}\)
  descends to a \(G,G\)-bibundle map \(\bar\mu_{t,u}\colon X_t\times_G
  X_u\to X_{tu}\).  Since~\(\mu_{t,u}\) is surjective by~\ref{e:SAS4},
  so is~\(\bar\mu_{t,u}\).  Hence it is a bibundle isomorphism by
  Proposition~\ref{pro:characterise_bibundle_isom}.

  The groupoid structure on~\(X_1\) and the left and right actions
  on~\(X_t\) are defined so that~\(X_1\) is the identity equivalence
  on~\(G\) and the maps \(\bar\mu_{1,u}\) and~\(\bar\mu_{t,1}\) are
  the canonical isomorphisms.  The associativity condition for the
  bibundle isomorphisms~\(\bar\mu_{t,u}\) follows from the
  corresponding property of the maps~\(\mu_{t,u}\).  Thus we have got
  an action by partial equivalences.  This is the only action that
  simplifies to the given data because of the assumptions about
  \(X_1\), \(\mu_{1,u}\), and~\(\mu_{t,1}\) in
  Definition~\ref{def:S_act_groupoid}.

  By definition, \(X_t\times_G X_u\) is the orbit space of the
  \(G\)\nb-action on \(X_t\times_{\s,G^0,\rg} X_u\) by
  \((x_1,x_2)\cdot g \defeq (x_1\cdot g,g^{-1}\cdot x_2)\).  The
  canonical projection \(X_t\times_{\s,G^0,\rg} X_u\to X_t\times_G
  X_u\) is open by Proposition~\ref{pro:Top_open_orbit}.  The
  map~\(\mu_{t,u}\) is the composite of this projection with the
  homeomorphism \(\bar\mu_{t,u}\colon X_t\times_G X_u\to X_{tu}\),
  hence it is also open.

  Finally, we check that the maps in~\ref{e:SAS5} are isomorphisms for
  all \(t,u\in S\); exchanging left and right gives the same for the
  maps in~\ref{e:SAS6}.  The map in~\ref{e:SAS5} is
  \(G\)\nb-equivariant if we let~\(G\) act on \(X_t\times_{\s,G^0,\rg}
  X_u\) by \(g\cdot (x,y)\defeq (xg^{-1},gy)\) and on
  \(X_u\times_{\s,G^0,\s} X_{tu}\) by \(g\cdot (y,x)\defeq (gy,x)\).
  Both actions are part of principal bundles: the bundle projection on
  \(X_t\times_{\s,G^0,\rg} X_u\) is the canonical map to \(X_t\times_G
  X_u\), and the bundle projection on \(X_u\times_{\s,G^0,\s} X_{tu}\)
  is \(\s\times_{G^0,\s} \Id_{X_{tu}}\) to~\(X_{tu}|_{\rg(X_u)}\).
  Our \(G\)\nb-equivariant map induces the map~\(\mu_{t,u}\) on the
  base spaces, which is a homeomorphism.  Hence so is the map on the
  total spaces by \cite{Meyer-Zhu:Groupoids}*{Proposition 5.9}.
\end{proof}

\subsection{Compatibility with order and involution}
\label{sec:order_involution}

Let~\(S\) be an inverse semigroup with unit.  Define a partial order
on~\(S\) by \(t\le u\) if \(t=tt^*u\) or, equivalently,
\(t=ut^*t\).  The multiplication and involution preserve this order:  \(t_1t_2\le
u_1u_2\) and \(t_1^*\le u_1^*\) if \(t_1\le u_1\) and \(t_2\le u_2\) (see \cite{Lawson:InverseSemigroups}).

Let \((X_t)_{t\in S}\), \((\mu_{t,u})_{t,u\in S}\) be an action
of~\(S\) on~\(G\).  We are going to prove that the action is
compatible with this partial order and the involution on~\(S\).  To
prepare for the proofs of analogous statements for inverse semigroup
actions on \(\Cst\)\nb-algebras, we give rather abstract proofs, which
carry over literally to the \(\Cst\)\nb-algebraic case.

\begin{proposition}
  \label{pro:inclusions_from_action}
  There are unique bibundle maps \(j_{u,t}\colon X_t \to
  X_u\) for \(t,u\in S\) with \(t\le u\) such that the following
  diagrams commute for all \(t_1,t_2,u_1,u_2\in S\) with \(t_1\le
  u_1\), \(t_2\le u_2\):
  \begin{equation}
    \label{eq:inclusions_from_action}
    \begin{tikzpicture}[baseline=(current bounding box.west)]
      \matrix (m) [cd,row sep=1.8em,column sep=4em] {
        X_{t_1}\times_G X_{t_2}& X_{t_1t_2}\\
        X_{u_1}\times_G X_{u_2}& X_{u_1u_2}\\
      };
      \draw[cdar] (m-1-1) -- node {\(\mu_{t_1,t_2}\)} (m-1-2);
      \draw[cdar] (m-2-1) -- node {\(\mu_{u_1,u_2}\)} (m-2-2);
      \draw[cdar] (m-1-1) -- node[swap] {\(j_{u_1,t_1}\times_G
        j_{u_2,t_2}\)} (m-2-1);
      \draw[cdar] (m-1-2) -- node {\(j_{u_1u_2,t_1t_2}\)} (m-2-2);
    \end{tikzpicture}
  \end{equation}
  The map~\(j_{u,t}\) is a bibundle isomorphism onto \(X_u|_{\s(X_t)}
  = {}_{\rg(X_t)}|X_u\).  We have \(j_{t,t}=\Id_{X_t}\) for all \(t\in
  S\) and \(j_{v,u}\circ j_{u,t} = j_{v,t}\) for \(t\le u\le v\)
  in~\(S\).
\end{proposition}

\begin{proof}
  Let \(E(S)\subseteq S\) be the subset of idempotents and let \(e\in
  E(S)\).  Proposition~\ref{pro:idempotents} gives a unique
  isomorphism \(X_e\cong G^1_{U_e}\) intertwining \(\mu_{e,e}\colon
  X_e\times_G X_e\to X_e\) and the multiplication
  in~\(G^1_{U_e}\); here \(U_e\defeq \rg(X_e)=\s(X_e)\) is an open
  \(G\)\nb-invariant subset of~\(G^0\).  The
  diagram~\eqref{eq:inclusions_from_action} for \((e,e)\le (1,1)\)
  shows that~\(j_{1,e}\) has to be this particular isomorphism
  \(X_e\cong G^1_{U_e}\subseteq G^1\).  To simplify notation, we now
  identify~\(X_e\) with~\(G^1_{U_e}\) for all \(e\in E(S)\) using
  these unique isomorphisms, and we transfer the multiplication
  maps~\(\mu_{s,t}\) for idempotent  \(s\), \(t\) or~\(s t\) accordingly.  This
  gives an isomorphic action of~\(S\) by partial equivalences.  So we
  may assume that \(X_e=G_{U_e}^1\) and that \(\mu_{e,e}\colon
  X_e\times_G X_e\to X_e\) is the usual multiplication map
  on~\(G_{U_e}^1\) for all \(e\in E(S)\).

  Let \(e\in E(S)\) and let \(t,u\in S\) satisfy \(t^*t\le e\) and
  \(uu^*\le e\).  Thus \(te=t\), \(eu=u\) and \(teu=tu\).  We show
  that \(\mu_{t,e}\colon X_t\times_G G^1_{U_e}\to X_t\) and
  \(\mu_{e,u}\colon G^1_{U_e}\times_G X_u\congto X_u\) are the obvious
  maps \(\mu^0_{t,e}\) or~\(\mu^0_{e,u}\) from the left and right
  \(G\)\nb-actions in this case.  Associativity of the multiplication
  maps gives us a commuting diagram of isomorphisms
  \[
  \begin{tikzpicture}[baseline=(current bounding box.west)]
    \matrix (m) [cd,row sep=1em, column sep=6em] {
      X_t\times_G X_e\times_G X_u&X_t\times_G X_u\\
      X_t\times_G X_u& X_{tu}\\
    };
    \draw[cdar] (m-1-1) -- node {\(\Id_{X_t}\times_G\mu_{e,u}\)} (m-1-2);
    \draw[cdar] (m-1-2) -- node {\(\mu_{t,u}\)} (m-2-2);
    \draw[cdar] (m-1-1) -- node[swap] {\(\mu_{t,e}\times_G\Id_{X_u}\)} (m-2-1);
    \draw[cdar] (m-2-1) -- node[swap] {\(\mu_{t,u}\)} (m-2-2);
    \draw[double,double equal sign distance] (m-1-2) -- (m-2-1);
  \end{tikzpicture}
  \]
  We may cancel the isomorphism~\(\mu_{t,u}\) to get
  \(\Id_{X_t}\times_G\mu_{e,u} = \mu_{t,e}\times_G\Id_{X_u}\).  Now we
  consider two cases: \(t=e\) or \(e=u\).  If \(t=e\), then
  \(\mu_{t,e} = \mu^0_{t,e}\) is the multiplication map
  on~\(G^1_{U_e}\).  Hence so is \(\mu_{t,e}\times_G\Id_{X_u}\).  Thus
  \(\mu_{e,u}\) and \(\mu^0_{e,u}\) induce the same map \(G^1_{U_e}
  \times_G G^1_{U_e} \times_G X_u\to G^1_{U_e}\times_G X_u\).  We may
  use~\eqref{eq:restrict_as_product} to cancel the
  factor~\(G^1_{U_e}\) because \(\s(X_e)= U_e\supseteq \rg(X_u)
  \supseteq \rg(G^1_{U_e}\times_G X_u)\).  Thus
  \(\mu_{e,u}=\mu^0_{e,u}\) if~\(e\) is idempotent and \(e\ge uu^*\).
  A similar argument in the other case \(e=u\) gives
  \(\mu_{t,e}=\mu^0_{t,e}\) if \(t^*t\le e\).

  Now let \(t\le u\), that is, \(t=tt^*u=ut^*t\).  Then we get two
  candidates for the bibundle map \(j_{u,t}\colon X_t\to X_u\):
  \begin{equation}
    \label{eq:construct_j}
    \begin{gathered}
      X_t \xleftarrow[\cong]{\mu_{tt^*,u}} X_{tt^*} \times_G X_u
      = G^1_{U_{tt^*}} \times_G X_u
      \xrightarrow[\cong]{\mu^0_{tt^*,u}} {}_{U_{tt^*}}|X_u\subseteq X_u,\\
      X_t \xleftarrow[\cong]{\mu_{u,t^*t}} X_u \times_G X_{t^*t}
      = X_u \times_G G^1_{U_{t^*t}}
      \xrightarrow[\cong]{\mu^0_{u,t^*t}} X_u|_{U_{t^*t}} \subseteq X_u.\\
    \end{gathered}
  \end{equation}
  We claim that both maps \(X_t\to X_u\) are equal, so we get only one
  map \(j_{u,t}\colon X_t\to X_u\).  Let \(e=tt^*\) and \(f= t^*t\).
  Then there is a commuting diagram of isomorphisms
  \begin{equation}
    \label{eq:associativity2}
    \begin{tikzpicture}[baseline=(current bounding box.west)]
      \matrix (m) [cd,column sep=.1em,row sep=1.5em] {
        X_e\times_G X_u\times_G X_f&&& X_e\times_G X_t\\
        &&{}_{U_e}|X_u\times_G X_f\\
        &X_e\times_G X_u|_{U_f}\\
        X_t\times_G X_f&&& X_t\\
      };
      \draw[cdar] (m-1-1) -- node {\(\Id_{X_e}\times_G\mu_{u,f}\)} (m-1-4);
      \draw[cdar,mid] (m-1-4) -- node[narrowfill]
      {\(\mu_{e,t} = \mu^0_{e,t}\)} (m-4-4);
      \draw[cdar,mid] (m-1-1) -- node[narrowfill]
      {\(\mu_{e,u}\times_G\Id_{X_f}\)} (m-4-1);
      \draw[cdar] (m-4-1) -- node[swap]
      {\(\mu_{t,f}=\mu^0_{t,f}\)} (m-4-4);
      \draw[cdar,mid] (m-1-1) -- node[narrowfill]
      {\(\Id_{X_e}\times_G\mu^0\)} (m-3-2);
      \draw[cdar,mid] (m-1-1) -- node[fill=white]
      {\(\mu^0\times_G\Id_{X_f}\)} (m-2-3);
      \draw[cdar,mid] (m-3-2) -- node[narrowfill]
      {\(\mu_{e,u}\)} (m-4-4);
      \draw[cdar,mid] (m-2-3) -- node[narrowfill]
      {\(\mu_{u,f}\)} (m-4-4);
    \end{tikzpicture}
  \end{equation}
  The large rectangle commutes by associativity.  The argument above
  gives \(\mu_{e,t} = \mu^0_{e,t}\) and \(\mu_{t,f}=\mu^0_{t,f}\).
  The lower left and upper right triangles commute because
  \(\mu_{e,u}\) and~\(\mu_{u,f}\) are bibundle maps, so they are
  compatible with~\(\mu^0\).  Hence the interior quadrilateral
  commutes.  Thus the two definitions of~\(j_{u,t}\)
  in~\eqref{eq:construct_j} are equal.

  The first construction of~\(j_{u,t}\) in~\eqref{eq:construct_j}
  gives the unique map for which the
  diagram~\eqref{eq:inclusions_from_action} commutes for \((e,t)\le
  (1,u)\) and the inclusion map~\(j_{1,e}\).  Since we already saw
  that~\(j_{1,e}\) is unique, the
  diagrams~\eqref{eq:inclusions_from_action} characterise the bibundle
  maps~\(j_{u,t}\) uniquely for all \(t\le u\) in~\(S\).  The
  map~\(j_{t,t}\) is the identity on~\(X_t\) because
  \(\mu_{tt^*,t}=\mu^0_{tt^*,t}\).

  Now let \(t\le u\le v\), define \(e=t t^*\) and \(f=u u^*\) and
  identify \(X_e\) and~\(X_f\) with subsets of~\(G^1\).  In the
  following diagram, we abbreviate~\(\times_G\) to~\(\ast\), and
  \(\mu^0\) denotes the left and right actions for subsets of~\(G^1\):
  \[
  \begin{tikzpicture}[baseline=(current bounding box.west)]
    \matrix (m) [cd,column sep=4.3em,row sep=3em] {
      X_t & X_e * X_u & {}_{U_e}|X_u\\
      X_e*X_v & X_e * X_f*X_v & {}_{U_e}|X_f*X_v\\
      {}_{U_e}| X_v & X_e * ({}_{U_f}| X_v) & {}_{U_e}|X_v\\
    };
    \draw[cdar] (m-1-2) -- node[swap] {\(\mu_{e,u}\)} (m-1-1);
    \draw[cdar] (m-1-2) -- node {\(\mu^0\)} (m-1-3);
    \draw[cdar] (m-2-2) -- node[swap] {\(\mu_{e,f}*\Id\)} (m-2-1);
    \draw[cdar] (m-2-2) -- node {\(\mu^0*\Id\)} (m-2-3);
    \draw[cdar] (m-3-2) -- node[swap] {\(\mu^0\)} (m-3-1);
    \draw[cdar] (m-3-2) -- node {\(\mu^0\)} (m-3-3);

    \draw[cdar] (m-2-1) -- node[swap] {\(\mu_{e,v}\)} (m-1-1);
    \draw[cdar] (m-2-2) -- node {\(\Id*\mu_{f,v}\)} (m-1-2);
    \draw[cdar] (m-2-3) -- node {\({}_{U_e}|\mu_{f,v}\)} (m-1-3);
    \draw[cdar] (m-2-1) -- node {\(\mu^0\)} (m-3-1);
    \draw[cdar] (m-2-2) -- node[swap] {\(\Id*\mu^0\)} (m-3-2);
    \draw[cdar] (m-2-3) -- node[swap] {\(\mu^0\)} (m-3-3);

    \draw[->,bend left] (m-1-1) to node {\(j_{u,t}\)} (m-1-3);
    \draw[->,bend left] (m-1-3) to[out=50,in=130] node {\({}_{U_e}|j_{v,u}\)} (m-3-3);
    \draw[->] (m-1-1) to[out=220,in=140] node[swap] {\(j_{v,t}\)} (m-3-1);
    \draw[double,double equal sign distance,bend right] (m-3-1) to (m-3-3);
  \end{tikzpicture}
  \]
  The top left square commutes because the multiplication maps are
  associative, the top right square because they are bibundle maps.
  The bottom left square commutes because \(\mu_{e,f}=\mu^0\), and the
  bottom right square commutes for trivial reasons.  The bent
  composite arrows are the maps~\(j\) by construction.  Thus the whole
  diagram commutes, and this means that \(j_{v,u}\circ j_{u,t} =
  j_{v,t}\).

  If \(t_1\le u_1\) and \(t_2\le u_2\) in~\(S\), then there is a
  commuting diagram of isomorphisms
  \begin{equation}
    \label{eq:associativity3}
    \begin{tikzpicture}[baseline=(current bounding box.west)]
      \matrix (m) [cd,row sep=2.5em,column sep=5.6em] {
        X_{t_1}* X_{t_2}& X_{t_1t_2}\\
        X_{t_1t_1^*}* X_{u_1} * X_{u_2} *
        X_{t_2^*t_2} &
        X_{t_1t_1^*}* X_{u_1u_2} * X_{t_2^*t_2} \\
        {}_{U_{t_1t_1^*}}|X_{u_1} * X_{u_2} |_{U_{t_2^*t_2}}&
        {}_{U_{t_1t_1^*}}|X_{u_1u_2} |_{U_{t_2^*t_2}}\\
      };
      \draw[cdar] (m-2-1) -- node[swap] {\(\scriptstyle\mu_{t_1t_1^*,u_1}*
        \mu_{u_2,t_2^*t_2}\)} (m-1-1);
      \draw[cdar] (m-2-1) -- node {\(\scriptstyle\mu^0\)} (m-3-1);
      \draw[cdar] (m-2-2) -- node
      {\(\scriptstyle\mu_{t_1t_1^*,u_1u_2,t_2^*t_2}\)} (m-1-2);
      \draw[cdar] (m-2-2) -- node[swap] {\(\scriptstyle\mu^0\)} (m-3-2);
      \draw[cdar] (m-1-1) -- node {\(\scriptstyle\mu_{t_1,t_2}\)} (m-1-2);
      \draw[cdar] (m-2-1) -- node[swap]
      {\(\scriptstyle\Id*\mu_{u_1,u_2}*\Id\)} (m-2-2);
      \draw[cdar] (m-3-1) -- node[swap]
      {\(\scriptstyle\mu_{u_1,u_2}\)} (m-3-2);
    \end{tikzpicture}
  \end{equation}
  Here we abbreviate~\(\times_G\) to~\(\ast\), \(\mu^0\) denotes the
  left and right actions for subsets of~\(G^1\),
  and~\(\mu_{t_tt_t^*,u_1,u_2,t_2^*t_2}\) denotes the appropriate
  combination of two multiplication maps, which is well-defined by
  associativity.  The upper square commutes by associativity.  The
  lower square commutes because~\(\mu_{u_1,u_2}\) is a bibundle map.
  The left vertical isomorphism from \(X_{t_1}* X_{t_2}\) to
  \({}_{U_{t_1t_1^*}}|X_{u_1} * X_{u_2} |_{U_{t_2^*t_2}}\) is
  \(j_{u_1,t_1}*j_{u_2,t_2}\) because the two constructions
  in~\eqref{eq:construct_j} coincide.  It remains to see that the
  right vertical isomorphism from~\(X_{t_1t_2}\) to
  \({}_{U_{t_1t_1^*}}|X_{u_1u_2}|_{U_{t_2^*t_2}}\)
  is~\(j_{u_1u_2,t_1t_2}\).

  The proof of this is similar to the proof that the two maps
  in~\eqref{eq:construct_j} coincide.  Let \(e=(t_1t_2)(t_1t_2)^*\),
  so \(e\le t_1t_1^*\).  Since \(\rg(X_{t_1t_2})=U_e\)
  and~\eqref{eq:associativity3} is a diagram of isomorphisms, we have
  \(X_e*X_{t_1t_1^*}* X_{u_1u_2} * X_{t_2^*t_2} \cong X_{t_1t_1^*}*
  X_{u_1u_2} * X_{t_2^*t_2}\).  Furthermore, the isomorphism
  \[
  \mu_{e,t_1t_1^*}*\Id\colon X_e*X_{t_1t_1^*}* X_{u_1u_2} * X_{t_2^*t_2}
  \to X_e* X_{u_1u_2} * X_{t_2^*t_2}
  \]
  is equal to the standard multiplication map
  \(\mu^0_{e,t_1t_1^*}*\Id\) because \(e\le t_tt_1^*\).  This fact and
  associativity show that the right vertical isomorphism
  in~\eqref{eq:associativity3} is equal to the composite map
  \[
  X_{t_1t_2} \xleftarrow[\cong]{\mu_{e,u_1u_2,t_2t_2^*}}
  X_e* X_{u_1u_2} * X_{t_2^*t_2}
  \xrightarrow[\cong]{\mu^0}
  {}_{U_e}|X_{u_1u_2}|_{U_{t_2^*t_2}}
  = {}_{U_e}|X_{u_1u_2}.
  \]
  Similarly, we get the same composite map if we
  replace~\(t_2^*t_2\) on the right by the smaller idempotent
  \(f=(t_1t_2)^*(t_1t_2)\).  Now the diagram~\eqref{eq:associativity2}
  shows that the map we get is~\(j_{u_1u_2,t_1t_2}\) as desired.
  Hence~\eqref{eq:inclusions_from_action} commutes.
\end{proof}

\begin{remark}
  \label{rem:ActionsSemilattice}
  Let~\(E\) be a semilattice with unit~\(1\), viewed as an inverse
  semigroup.  An \(E\)\nb-action on a topological groupoid~\(G\) is
  the same as a unital semilattice map from~\(E\) to the lattice of
  open \(G\)\nb-invariant subsets of~\(G^0\), that is, a map
  \(e\mapsto U_e\) satisfying \(U_1=G^0\) and \(U_e\cap U_f=U_{ef}\)
  for all \(e,f\in E\).  The corresponding action by partial
  equivalences is defined by \(X_e\defeq G^1_{U_e}\) and
  \(\mu_{e,f}=\mu^0\colon G^1_{U_e}\times_{G} G^1_{U_f}\to
  G^1_{U_{ef}}\).  Proposition~\ref{pro:inclusions_from_action}
  implies that every action of~\(E\) is isomorphic to one of this
  form.
\end{remark}

\begin{proposition}
  \label{pro:S_action_involution}
  There are unique bibundle isomorphisms \(J_t\colon X^*_t\to
  X_{t^*}\) for which the following composite map is the identity:
  \begin{equation}
    \label{eq:characterise_Jt_1}
    X_t \cong X_t\times_G X_t^*\times_G X_t
    \xrightarrow{\Id_{X_t}\times_G J_t\times_G \Id_{X_t}}
    X_t\times_G X_{t^*}\times_G X_t \xrightarrow{\mu_{t,t^*,t}}
    X_t.
  \end{equation}
  These involutions also make the following diagrams commute:
  \begin{equation}
    \label{eq:characterise_Jt_2}
    \begin{tikzpicture}[baseline=(current bounding box.west)]
      \matrix (m) [cd,row sep=1.5em] {
        X_t \times_G X_t^* & G^1_{U_{tt^*}}\\
        X_t \times_G X_{t^*} & X_{tt^*}\\
      };
      \draw[cdar] (m-2-1) -- node[swap] {\(\mu_{t,t^*}\)} (m-2-2);
      \draw[cdar] (m-1-1) -- node[swap] {\(\Id_{X_t}\times_G J_t\)} (m-2-1);
      \draw[cdar] (m-1-2) -- (m-2-2);
      \draw[cdar] (m-1-1) -- (m-1-2);
    \end{tikzpicture}
    \quad
    \begin{tikzpicture}[baseline=(current bounding box.west)]
      \matrix (m) [cd,row sep=1.5em] {
        X_t^* \times_G X_t & G^1_{U_{t^*t}}\\
        X_{t^*} \times_G X_t & X_{t^*t}\\
      };
      \draw[cdar] (m-2-1) -- node[swap] {\(\mu_{t^*,t}\)} (m-2-2);
      \draw[cdar] (m-1-1) -- node[swap] {\(J_t\times_G \Id_{X_t}\)} (m-2-1);
      \draw[cdar] (m-1-2) -- (m-2-2);
      \draw[cdar] (m-1-1) -- (m-1-2);
    \end{tikzpicture}
  \end{equation}
  Here the unlabelled arrows are the canonical isomorphisms from
  Propositions \textup{\ref{pro:properties_of_dual}}
  and~\textup{\ref{pro:idempotents}}.  Furthermore, \((J_{t^*})^*\circ
  J_t\colon X_t^*\to X_{t^*} \to X_t^*\) is the identity map for all
  \(t\in S\) and the following diagrams commute for all \(t,u,v\in S\)
  with \(t\le u\):
  \begin{equation}
    \label{eq:Jt_3}
    \begin{tikzpicture}[baseline=(current bounding box.west)]
      \matrix (m) [cd,row sep=1.5em] {
        X_u^* \times_G X_v^* & X_{vu}^*\\
        X_{u^*} \times_G X_{v^*} & X_{u^*v^*}\\
      };
      \draw[cdar] (m-1-1) -- node {\(\mu_{v,u}^*\)} (m-1-2);
      \draw[cdar] (m-1-1) -- node[swap] {\(J_u\times_G J_v\)} (m-2-1);
      \draw[cdar] (m-2-1) -- node[swap] {\(\mu_{u^*,v^*}\)} (m-2-2);
      \draw[cdar] (m-1-2) -- node {\(J_{vu}\)} (m-2-2);
    \end{tikzpicture}
    \quad
    \begin{tikzpicture}[baseline=(current bounding box.west)]
      \matrix (m) [cd,row sep=1.5em] {
        X_t^* & X_u^*|_{U_{tt^*}}\\
        X_{t^*} & X_{u^*}|_{U_{tt^*}}\\
      };
      \draw[cdar] (m-1-1) -- node {\(j_{u,t}^*\)} (m-1-2);
      \draw[cdar] (m-1-1) -- node[swap] {\(J_t\)} (m-2-1);
      \draw[cdar] (m-2-1) -- node[swap] {\(j_{u^*,t^*}\)} (m-2-2);
      \draw[cdar] (m-1-2) -- node {\(J_u|_{U_{tt^*}}\)} (m-2-2);
    \end{tikzpicture}
  \end{equation}
\end{proposition}

Write \(x^*\defeq J_t(x)\) for \(x\in X_t\) and
\(\mu_{t,u}(x,y)=x\cdot y\) for \(x\in X_t\), \(y\in X_u\) with
\(\s(x)=\rg(y)\).  The above diagrams and equations of maps mean that
the involution is characterised by \(x\cdot x^*\cdot x=x\) for all
\(x\in X_t\) and has the properties \(x\cdot x^*=1_{\rg(x)}\),
\(x^*\cdot x=1_{\s(x)}\), \((x^*)^*=x\), \((x\cdot y)^* = y^*\cdot
x^*\), and \(j_{u^*,t^*}(x^*) = j_{u,t}(x)^*\).

\begin{proof}
  The two isomorphisms \(\mu_{t,t^*,t}\colon X_t\times_G
  X_{t^*}\times_G X_t\to X_t\) and \(\mu_{t^*,t,t^*}\colon
  X_{t^*}\times_G X_t\times_G X_{t^*}\to X_{t^*}\) that we may build
  from~\(\mu\) are equal by associativity.
  Proposition~\ref{pro:dual_pe_unique} for these isomorphisms gives a
  unique isomorphism \(J_t\colon X_t^*\cong X_{t^*}\) for
  which~\eqref{eq:characterise_Jt_1} becomes the identity map.

  We claim that~\eqref{eq:characterise_Jt_1} is the identity if and
  only if either of the diagrams in~\eqref{eq:characterise_Jt_2}
  commutes.  The proofs for both cases differ only by exchanging left
  and right, so we only write down one of them.  Assume that the first
  diagram in~\eqref{eq:characterise_Jt_2} commutes.  Applying the
  functor \(\blank \times_G \Id_{X_t}\) to it, we get that the
  isomorphism~\eqref{eq:characterise_Jt_1} is the identity map because
  the multiplication map \(\mu_{tt^*,t}\colon X_{tt^*}\times_G X_t\to
  X_t\) is just the left action if we identify \(X_{tt^*}\cong
  G^1_{U_{tt^*}}\) as usual.  Conversely, assume that the isomorphism
  in~\eqref{eq:characterise_Jt_1} is the identity map.  Take a further
  product with~\(X_{t^*}\) and then identify \(X_t\times_G
  X_{t^*}\cong X_{tt^*}\) via~\(\mu_{t,t^*}\).  Using again that the
  multiplication with~\(X_{tt^*}\) is just the \(G\)\nb-action, this
  gives the first diagram in~\eqref{eq:characterise_Jt_2}.

  Next we show that \(J_{t^*} = (J_t^{-1})^*\), which implies
  \(J_{t^*}^*\circ J_t = \Id_{X_t}\).  We use the commuting diagram
  \[
  \begin{tikzpicture}[baseline=(current bounding box.west)]
    \matrix (m) [cd,row sep=1.5em] {
      X^*_{t^*} \times_G X_{t^*} & G^1_{U_{tt^*}}\\
      X_t \times_G X^*_t & G^1_{U_{tt^*}}\\
      X_t \times_G X_{t^*} & X_{tt^*}\\
    };
    \draw[cdar] (m-1-1) -- (m-1-2);
    \draw[cdar] (m-2-1) -- (m-2-2);
    \draw[cdar] (m-3-1) -- node {\(\mu_{t,t^*}\)} (m-3-2);
    \draw[cdar] (m-1-1) -- node[swap] {\((J_t^{-1})^*\times_G
      J_t^{-1}\)} (m-2-1);
    \draw[cdar] (m-2-1) -- node[swap] {\(\Id_{X_t}\times_G J_t\)} (m-3-1);
    \draw[double,double equal sign distance] (m-1-2) -- (m-2-2);
    \draw[cdar] (m-2-2) -- (m-3-2);
  \end{tikzpicture}
  \]
  The top rectangle commutes because the pairing \(X\times_G X^*\to
  G^1_{\rg(X)}\) is natural.  The bottom diagram is the first one
  in~\eqref{eq:characterise_Jt_2}.  The large rectangle is the second
  diagram in~\eqref{eq:characterise_Jt_2} for~\(t^*\) with
  \((J_t^{-1})^*\) instead of \(J_{t^*}\).  Since this diagram
  characterises~\(J_{t^*}\), we get \(J_{t^*} = (J_t^{-1})^*\) as
  asserted.

  Since the involution~\(J_{vu}\) is uniquely characterised by a
  diagram like the first one in~\eqref{eq:characterise_Jt_2}, we may
  prove the first diagram in~\eqref{eq:Jt_3} by showing that the
  composite map \(\mu_{u^*,v^*}\circ (J_u\times_G J_v)\circ
  (\mu_{v,u}^*)^{-1}\colon X_{vu}^* \to X_{u^*v^*}\) also makes the
  diagram in~\eqref{eq:characterise_Jt_2} for \(t=vu\) commute.  This
  is a routine computation using the same diagrams for \(J_u\)
  and~\(J_v\) and that the multiplication maps involving~\(X_e\) for
  idempotent \(e\in S\) are always given by the left or right action
  because of the compatibility with~\(j_{1,e}\).  This proof is a
  variant of the usual proof that \((xy)^{-1}=y^{-1}x^{-1}\) in a
  group because \(y^{-1}x^{-1}\cdot (xy) = 1\).

  Similarly, we get the second diagram in~\eqref{eq:Jt_3} by showing
  that the composite map \(j_{u^*,t^*}^{-1} \circ J_u\circ
  j_{u,t}^*\colon X_t^*\to X_{t^*}\) satisfies the defining condition
  for~\(J_t\) because \(j_{u,t}\) and~\(j_{u^*,t^*}\) are compatible
  with the multiplication maps.
\end{proof}

\subsection{Transformation groupoids}
\label{sec:transformation}

Let~\((X_t,\mu_{t,u})_{t,u\in S}\) be an action of a unital inverse
semigroup~\(S\) on a topological groupoid~\(G\) by partial
equivalences.  Define the embeddings \(j_{u,t}\colon X_t\to X_u\) for
\(t\le u\) in~\(S\) and the involutions \(X_t^*\to X_{t^*}\) as in
Propositions \ref{pro:inclusions_from_action}
and~\ref{pro:S_action_involution}.

Let \(X\defeq \bigsqcup_{t\in S} X_t\) and define a relation~\(\sim\)
on~\(X\) by \((t,x)\sim (u,y)\) for \(x\in X_t\), \(y\in X_u\) if
there are \(v\in S\) with \(v\le t,u\) and \(z\in X_v\) with
\(j_{t,v}(z)=x\) and \(j_{u,v}(z)=y\).

\begin{lemma}
  \label{lem:transformation_gp_open_equivalence}
  The relation~\(\sim\) is an equivalence relation.  Equip \(X_\sim
  \defeq X/{\sim}\) with the quotient topology.  The quotient map \(\pi\colon
  X\to X_\sim\) is a local homeomorphism.  It restricts to a
  homeomorphism from~\(X_t\) onto an open subset of~\(X_\sim\) for
  each \(t\in S\).  Thus~\(X_\sim\) is locally quasi-compact or
  locally Hausdorff if and only if all~\(X_t\) are.
\end{lemma}

\begin{proof}
  It is clear that~\(\sim\) is reflexive and symmetric.  For
  transitivity, take \((t_1,x_1)\sim (t_2,x_2)\sim (t_3,x_3)\).  Then
  there are \(t_{12}\le t_1,t_2\), \(t_{23}\le t_2,t_3\), \(x_{12}\in
  X_{t_{12}}\), and \(x_{23}\in X_{t_{23}}\) with
  \(j_{t_i,t_{12}}(x_{12}) = x_i\) for \(i=1,2\) and
  \(j_{t_i,t_{23}}(x_{23}) = x_i\) for \(i=2,3\).  Thus \(\s(x_{12}) =
  \s(x_2) = \s(x_{23})\in \s(t_{23})= \s(t_{23}^*t_{23})\).  Let
  \(t\defeq t_{12}t_{23}^*t_{23}\), so that \(t\le t_{12}\) and \(t\le
  t_2t_2^* t_{23} = t_{23}\).  We have \(x_{12}\in
  X_{t_{12}}|_{U_{t_{23}^*t_{23}}} \cong X_{t_{12}}\times_G
  X_{t_{23}^*t_{23}} \cong X_t\).  Let~\(x\) be the image
  of~\(x_{12}\) under this isomorphism.  Then
  \(j_{t_{12},t}(x)=x_{12}\).  Hence \(j_{t_i,t}(x)=j_{t_i,t_{12}}
  (j_{t_{12},t}(x)) = x_i\) for \(i=1,2\).  Since
  \(j_{t_2,t_{23}}(x_{23})=x_2= j_{t_2,t_{23}}(j_{t_{23},t}(x))\) and
  \(j_{t_2,t_{23}}\) is injective by
  Proposition~\ref{pro:characterise_bibundle_isom}, we get
  \(j_{t_{23},t}(x)=x_{23}\) and hence also \(j_{t_3,t}(x)=x_3\).
  Thus \(x_1\sim x_3\) as desired.

  We prove that~\(\pi\) is open.  Any open subset of~\(X\) is a
  disjoint union of open subsets of the spaces~\(X_t\); so~\(\pi\) is
  open if and only if all the maps \(X_t\to X_\sim\) are open.  Let
  \(U\subseteq X_t\) be open, then we must check that
  \(\pi^{-1}(\pi(U))\) is open.  This set is a union over the set of
  triples \(t,v,w\in S\) with \(w\le t,v\), where the set for
  \(t,v,w\) is contained in~\(X_v\) and consists of all~\(j_{v,w}(x)\)
  with \(x\in j_{t,w}^{-1}(U)\).  The map~\(j_{v,w}\) is open by
  Proposition~\ref{pro:characterise_bibundle_isom}, and~\(j_{t,w}\) is
  continuous, so \(j_{v,w}(j_{t,w}^{-1}(U))\) is open.  Hence
  \(\pi^{-1}(\pi(U))\) is open as a union of open subsets of~\(X\),
  showing that~\(\pi\) is open.

  If \((t,x)\sim (t,y)\), then there are \(u\le t\) and \(z\in X_u\)
  with \(x=j_{t,u}(z)=y\); so the map from~\(X_t\) to~\(X_\sim\) is
  injective.  Since~\(\pi\) is open and continuous, it restricts to a
  homeomorphism from~\(X_t\) onto an open subset of~\(X_\sim\).
  Thus~\(\pi\) is a local homeomorphism.  Since being locally
  Hausdorff or locally quasi-compact are local properties and~\(\pi\)
  is a local homeomorphism, \(X_\sim\) has one of these two properties
  if and only if~\(X\) has, if and only if each~\(X_t\) has.
\end{proof}

The space~\(X_\sim\) need not be Hausdorff, just as for étale
groupoids constructed from inverse semigroup actions on spaces,
where~\(X_\sim\) will be the groupoid of germs of the action (by
Theorem~\ref{the:action_on_space}).

From now on, we identify~\(X_t\) with its image in~\(X_\sim\), using
that \(\pi|_{X_t}\colon X_t\to X_\sim\) is a homeomorphism onto an
open subset by Lemma~\ref{lem:transformation_gp_open_equivalence}.

We are going to turn~\(X_\sim\) into a topological groupoid with the
same object space~\(G^0\) as~\(G\).  Since~\(j_{u,t}\) is a bibundle
map, it is compatible with range and source maps.  So the maps
\(\rg,\s\colon X_t\rightrightarrows G^0\) induce well-defined maps
\(\rg,s\colon X_\sim\rightrightarrows G^0\).

The multiplication maps~\(\mu_{t,u}\) give a continuous map
\(X\times_{\s,G^0,\rg} X\to X\), by mapping the \(t,u\)-component of
\(X\times_{\s,G^0,\rg} X\) to the \(tu\)-component of~\(X\)
by~\(\mu_{t,u}\).  Equation~\eqref{eq:inclusions_from_action} shows
that this descends to a well-defined continuous map \(\mu\colon
X_\sim\times_{\s,G^0,\rg} X_\sim\to X_\sim\).

\begin{lemma}
  \label{lem:transformation_gp_groupoid}
  The maps \(\rg,\s\colon X_\sim\rightrightarrows G^0\) and
  \(\mu\colon X_\sim\times_{\s,G^0,\rg} X_\sim\to X_\sim\) define a
  topological groupoid~\(X_\sim\).  It contains~\(G\) as an open
  subgroupoid.  Hence~\(X_\sim\) is étale if and only if~\(G\) is.
\end{lemma}

\begin{proof}
  The multiplication is associative already on~\(X\)
  by~\ref{enum:APE4} and the associativity of~\(S\).  The maps \(\rg\)
  and~\(\s\) are open on~\(X_\sim\) because they are so on
  each~\(X_t\).  The maps \(\rg,\s,\mu\) restricted to \(G^1=X_1\)
  reproduce the groupoid structure on~\(G\) by~\ref{enum:APE2}.  Even
  more, \ref{enum:APE3} implies that multiplication
  in~\(X_\sim\) with elements of~\(X_1\) is the same as the
  \(G\)\nb-action.  In particular, unit elements in~\(G^1\) act
  identically, so they remain unit elements in~\(X_\sim\).  If \(x\in
  X_t\), then \(x^*\in X_{t^*}\) satisfies
  \(\mu_{t,t^*}(x,x^*)=1_{\rg(x)}\) by
  Proposition~\ref{pro:S_action_involution}.  Hence
  \[
  \pi(x,t)\cdot \pi(x^*,t^*)
  \defeq \pi(\mu_{t,t^*}(x,x^*),tt^*)
  = \pi(1_{\rg(x)},tt^*).
  \]
  This is equivalent to the unit element~\((1_{\rg(x)},1)\)
  in~\(X_\sim\) because~\(j_{1,tt^*}\) is the usual inclusion map
  (more precisely, the computation above assumes that we identify
  \(X_{tt^*}\cong G^1_{U_{tt^*}}\subseteq G^1\) using~\(j_{1,tt^*}\)).
  Similarly, \(\pi(x,t)\cdot \pi(x^*,t^*) \sim (1_{\s(x)},1)\) is a
  unit element.  Thus \(\pi(x^*,t^*)\) is inverse to~\(\pi(x,t)\).
  The map \(\pi(x,t)\mapsto \pi(x^*,t^*)\) is continuous.  Thus we
  have a topological groupoid.  We have seen above that it
  contains~\(G\) as an open subgroupoid.  Therefore, \(X_\sim\) is
  étale if and only if~\(G\) is.
\end{proof}

\begin{definition}
  \label{def:transformation_gp}
  The groupoid~\(X_\sim\) is called the \emph{transformation
    groupoid} of the \(S\)\nb-action \((X_t,\mu_{t,u})\) on~\(G\) and
  denoted by \(G\rtimes S\), or by \(G\rtimes_{X_t,\mu_{t,u}} S\) if
  the action must be specified.
\end{definition}

Our proof shows that~\(G\rtimes S\) with the family of open
subsets~\((X_t)_{t\in S}\) encodes all the algebraic structure of our
action by partial equivalences.  The next definition characterises
when a groupoid~\(H\) with a family of subsets~\((H_t)_{t\in S}\) is
the transformation groupoid of an inverse semigroup action.

\begin{definition}
  \label{def:graded_groupoid}
  Let~\(S\) be an inverse semigroup.  A (saturated)
  \emph{\(S\)\nb-grading} on a topological groupoid~\(H\) is a family
  of open subsets~\((H_t)_{t\in S}\) of~\(H^1\) such that
  \begin{enumerate}[label=\textup{(Gr\arabic*)}]
  \item \label{enum:Gra1} \(H_t\cdot H_u=H_{tu}\) for all \(t,u\in S\);
  \item \label{enum:Gra2} \(H_t^{-1}=H_{t^*}\) for all \(t\in S\);
  \item \label{enum:Gra4} \(H_t\cap H_u = \bigcup_{v\le t,u} H_v\) for
    all \(t,u\in S\);
  \item \label{enum:Gra6} \(H^1=\bigcup_{t\in S} H_t\).
  \end{enumerate}
  If~\(S\) has a zero element~\(0\), we may also require
  \(H_0=\emptyset\).
\end{definition}

The conditions \ref{enum:Gra1} and~\ref{enum:Gra2} imply that~\(H_1\)
is a subgroupoid of~\(H\), called the \emph{unit fibre} of the
grading.  \ref{enum:Gra6} and~\ref{enum:Gra1} imply that
\(\s(H_1)=\rg(H_1)=H^0\).  \ref{enum:Gra4} implies \(H_v\subseteq
H_u\) for \(v\le u\).

A non-saturated \(S\)\nb-grading would be defined by
weakening~\ref{enum:Gra1} to \(H_t\cdot H_u\subseteq H_{tu}\) for all
\(t,u\in S\).  We only use saturated gradings and drop the adjective.

\begin{theorem}
  \label{the:graded_groupoid_versus_action}
  Let~\(S\) be an inverse semigroup with unit.  The transformation
  groupoid~\(G\rtimes S\) of an \(S\)\nb-action on a groupoid~\(G\) by
  partial equivalences is an \(S\)\nb-graded groupoid.  Any
  \(S\)\nb-graded groupoid \((H,(H_t)_{t\in S})\) is isomorphic to one
  of this form, where \(G^0=H^0\) and \(G^1=H_1\subseteq H\).  Two
  actions by partial equivalences are isomorphic if and only if their
  transformation groupoids are isomorphic in a grading-preserving way.
\end{theorem}

Here an isomorphism between actions \((X_t)_{t\in S}\)
and~\((Y_t)_{t\in S}\) by partial equivalences on two groupoids \(G\)
and~\(H\) means the obvious thing: a family of homeomorphisms \(X_t
\cong Y_t\) compatible with the range, source, and multiplication maps
in Definition~\ref{def:S_act_groupoid_simple}.

\begin{proof}
  It follows directly from our construction that the subspaces
  \(X_t\subseteq G\rtimes S\) for an \(S\)\nb-action by partial
  equivalences satisfy \ref{enum:Gra1}--\ref{enum:Gra6}.  It is also
  clear that the transformation groupoid construction is natural for
  isomorphisms of \(S\)\nb-actions.

  Let~\(H\) with the subspaces~\(H_t\) for \(t\in S\) be an
  \(S\)\nb-graded topological groupoid.  Then \(G^1\defeq H_1\) with
  \(G^0=H^0\) is an open subgroupoid of~\(H\).  Let \(X_t=H_t\) with
  the restriction of the range and source map of~\(H\), and with the
  \(G\)\nb-action and maps \(\mu_{t,u}\colon X_t\times_{\s,G^0,\rg}
  X_u\to X_{tu}\) from the multiplication map in~\(H\).  This
  satisfies \ref{e:SAS3} by definition, \ref{e:SAS4} because
  \(H_t\cdot H_u= H_{tu}\), \ref{e:SAS1} and~\ref{e:SAS7}
  because~\(H\) is a groupoid, and~\ref{e:SAS2} because~\(X_t\) is
  open in~\(H\) and the range and source maps of~\(H\) are open.  If
  \((y,z)\in X_u\times_{\s,G^0,\s} X_{tu}\), then \(zy^{-1}\in
  X_{tu}X_{u^*}= X_{tuu^*}\subseteq X_t\) because \(tuu^*\le t\).
  Hence \((y,z)\mapsto (zy^{-1},y)\) gives a continuous inverse for
  the map in~\ref{e:SAS5}, so that the latter is a homeomorphism.  A
  similar argument shows that the map in~\ref{e:SAS6} is a
  homeomorphism.  Thus we get an \(S\)\nb-action by partial
  equivalences.  This construction is natural in the sense that
  isomorphic \(S\)\nb-graded groupoids give isomorphic actions by partial
  equivalences actions.

  If we start with an action by partial equivalences, turn it into a
  graded groupoid, and then back into an action by partial
  equivalences, then we get an isomorphic action by construction.
  When we start with a graded groupoid, go to an action by partial
  equivalences and back to a graded groupoid, then we also get back
  our original \(S\)\nb-graded groupoid.  The only non-trivial point
  is that the map \(\pi\colon \bigsqcup_{t\in S} H_t\to
  \bigl(\bigsqcup_{t\in S} H_t\bigr)_\sim\) identifies \(x\in H_t\)
  and \(y\in H_u\) for \(t,u\in S\) if and only if \(x=y\) in~\(H\);
  this is exactly the meaning of~\ref{enum:Gra4}.
\end{proof}

\subsection{Examples: group actions and actions on spaces}
\label{sec:exa_action}

The equivalence between actions by partial equivalences and graded
groupoids makes it easy to describe all actions of groups on groupoids
and all actions of inverse semigroups on spaces.

\begin{theorem}
  \label{the:group_action_extension}
  Let~\(G\) be a topological groupoid and let~\(S\) be a group, viewed
  as an inverse semigroup.  Then an \(S\)\nb-action on~\(G\) by
  \textup{(}partial\textup{)} equivalences is equivalent to a
  groupoid~\(H\) containing~\(G\) as an open subgroupoid with
  \(H^0=G^0\), and with a continuous groupoid homomorphism \(\pi\colon
  H\onto S\) such that \(\pi^{-1}(1)=G\) and, for each \(x\in H^0\)
  and \(t\in S\) there is \(h\in H^1\) with \(\s(h)=x\) and
  \(\pi(h)=t\).  In this situation, \(H\) is the transformation
  groupoid~\(G\rtimes S\).  If~\(G\) is also a group, this is the same
  as a group extension \(G\rightarrowtail H\onto S\).
\end{theorem}

\begin{proof}
  Since \(tt^*=1\) for any \(t\in S\), any action of~\(S\) by partial
  equivalences will be an action by global equivalences.  By
  Theorem~\ref{the:graded_groupoid_versus_action}, we may replace an
  \(S\)\nb-action by partial equivalences by an \(S\)\nb-graded
  groupoid \((H,(H_t)_{t\in S})\).  We have \(H^0=G^0\) by
  construction.  Since~\(S\) is a group, \ref{enum:Gra4} says that
  \(H_t\cap H_u = \emptyset\) for \(t\neq u\).  Thus we get a
  well-defined map \(\pi\colon H^1\to S\) with \(\pi^{-1}(t)=H_t\); in
  particular, \(G=\pi^{-1}(1)\).  The map~\(\pi\) is continuous
  because the subsets~\(H_t\) are open.  The condition on the
  existence of~\(h\) for given~\(x,t\) says that the map \(\s\colon
  H_t\to H^0\) is onto, that is, \(H_t\) is a global equivalence.
  Thus an \(S\)\nb-action on~\(G\) gives \(\pi\colon H\to S\) with the
  asserted properties.

  For the converse, let \(\pi\colon H\to S\) be a groupoid
  homomorphism as in the statement.  Define \(H_t\defeq
  \pi^{-1}(t)\subseteq H^1\).  These are open subsets because~\(\pi\)
  is continuous.  If \(t,u\in S\), then \(H_t H_u\subseteq H_{tu}\) is
  trivial.  If \(h\in H_{tu}\), then our technical assumption gives
  \(h_2\in H_u\) with \(\s(h_2)=\s(h)\).  Then \(h_1\defeq
  hh_2^{-1}\in H_t\), so \(h\in H_tH_u\).  Thus~\ref{enum:Gra1} holds.
  The remaining conditions for an \(S\)\nb-grading are trivial in this
  case, and~\(H\) is the transformation groupoid~\(G\rtimes S\) by
  construction.

  If~\(G\) is also a group, then so is~\(H\) because \(G^0=H^0\), and
  then the condition on~\(\pi\) simply says that it is a surjection
  \(\pi\colon H\onto S\) with kernel~\(G\).  This is the same as a
  group extension.
\end{proof}

The obvious definition of a group action by automorphisms on another
group only covers \emph{split} group extensions.  We need some kind of
twisted action by automorphisms to allow for non-trivial group
extensions as well.  Our notion of action by equivalences achieves
this very naturally.

For groupoid extensions, one usually requires the kernel to be a group
bundle; this need not be the case here.  There are many examples of
groupoid homomorphisms (or \(1\)\nb-cocycles) with the properties
required in Theorem~\ref{the:group_action_extension}.  We mention one
typical case:

\begin{example}
  \label{exa:groupoid_of_self-cover}
  Let~\(H\) be the groupoid associated to a self-covering
  \(\sigma\colon X\to X\) of a compact space~\(X\) as
  in~\cite{Deaconu:Endomorphisms}.  The canonical \(\Z\)\nb-valued
  cocycle \(\pi\colon H\to \Z\) on it clearly has the properties needed to define a
  \(\Z\)\nb-grading on~\(H\).  The subgroupoid \(G\defeq \pi^{-1}(0)\)
  is the groupoid that describes the equivalence relation generated by
  \(x\sim y\) if \(\sigma^k(x)=\sigma^k(y)\) for some \(k\in\N\).  The
  action of~\(\sigma\) on~\(X\) preserves this equivalence relation
  and hence gives an endomorphism of~\(G\); this endomorphism is an
  equivalence, and our \(\Z\)\nb-action on~\(G\) by equivalences is
  generated by this self-equivalence of~\(G\).  But unless~\(\sigma\)
  is a homeomorphism, \(\sigma\) is not invertible on~\(G\), so it
  gives no action of~\(\Z\) by automorphisms.
\end{example}

Now we turn to actions of inverse semigroups on topological spaces.
Let~\(S\) be an inverse semigroup with unit and let~\(Z\) be a
topological space.  First we recall Exel's construction of the
\emph{groupoid of germs} for an inverse semigroup action by partial
homeomorphisms~\cite{Exel:Inverse_combinatorial}.

Let \(\PHomeo(Z)\) be the inverse semigroup of partial homeomorphisms
of~\(Z\).  An action of~\(S\) on~\(Z\) by partial homeomorphisms is a
monoid homomorphism \(\theta\colon S\to \PHomeo(Z)\).  This gives
partial homeomorphisms \(\theta_{t}\colon D_{t^*t}\to D_{tt^*}\) for
\(t\in S\) with open subsets \(D_e\subseteq Z\) for \(e\in E(S)\).
The groupoid of germs has object space~\(Z\), and its arrows are the
``germs'' \([t,z]\) for \(t\in S\), \(z\in D_{t^*t}\); by definition,
\([t,z]=[u,z']\) if and only if there is \(e\in E(S)\) with \(z=z'\in
D_e\) and \(te=ue\).  The groupoid structure is defined by
\(\s[t,z]=z\), \(\rg[t,z]=\theta_t(z)\), \([t,z]\cdot [u,z']=[tu,z']\)
if \(z=\theta_u(z')\), and \([t,z]^{-1}=[t^*,\theta_t(z)]\).  The
subsets \(\{[t,z]\mid z\in U\}\) for \(t\in S\) and an open subset
\(U\subseteq D_{t^*t}\) form a basis for the topology on the arrow
space.

\begin{remark}
  Many authors use another germ relation that only requires an open
  subset~\(V\) of~\(Z\) with \(z\in V\) and \(\theta_t|_V =
  \theta_u|_V\).  This may give a different groupoid, of course.
  Exel's germ groupoids need not be essentially principal (see \cite{Renault:Cartan.Subalgebras}).
\end{remark}

\begin{theorem}
  \label{the:action_on_space}
  Let~\(Z\) be a topological space viewed as a topological groupoid,
  and let~\(S\) be an inverse semigroup with unit.  Isomorphism
  classes of actions of~\(S\) on~\(Z\) by partial equivalences are in
  natural bijection with actions of~\(S\) on~\(Z\) by partial
  homeomorphisms.  The transformation groupoid \(Z\rtimes S\) for an
  action by partial equivalences is the groupoid of germs
  defined by Exel~\textup{\cite{Exel:Inverse_combinatorial}}.
\end{theorem}

\begin{proof}
  Let \(\theta\colon S\to\PHomeo(Z)\) be an action of~\(S\) by partial
  homeomorphisms.  Exel's groupoid of germs carries an obvious
  \(S\)\nb-grading by the open subsets \(X_t\defeq \{[t,z] \mid z\in
  D_{t^*t}\}\) with \(X_1=Z\).  The conditions in
  Definition~\ref{def:graded_groupoid} are trivial to check.  Hence
  Exel's groupoid is the transformation groupoid \(Z\rtimes S\) for an
  action of~\(S\) on~\(Z\) by partial equivalences by
  Theorem~\ref{the:graded_groupoid_versus_action}.  Conversely, an
  \(S\)\nb-action on~\(Z\) is equivalent to the \(S\)\nb-graded
  groupoid \(Z\rtimes S\).  This groupoid is étale.  The assumptions
  of an \(S\)\nb-grading imply that the subsets \(X_t\subseteq
  Z\rtimes S\) form an inverse semigroup of bisections that satisfies
  the assumptions in
  \cite{Exel:Inverse_combinatorial}*{Proposition~5.4}, which ensures
  that the groupoid of germs is \(Z\rtimes S\).
\end{proof}

\begin{corollary}
  \label{cor:wide_action}
  Let~\(G\) be an étale groupoid, let~\(S\) be an inverse semigroup,
  and let \(f\colon S\to\Bis(G)\) be a semigroup homomorphism.  This
  induces an isomorphism \(G^0\rtimes S\cong G\) if and only if
  \(\bigcup_{t\in S} f(t)=G\) and \(f(t)\cap f(u) = \bigcup_{v\in S,
    v\le t,u} f(v)\) for all \(t,u\in S\).
\end{corollary}

Here \(G^0\rtimes S\) uses the action of~\(S\) on~\(G^0\) induced
by~\(f\) and the usual action of~\(\Bis(G)\).

\begin{proof}
  Add a unit to~\(S\) and map it to the unit bisection \(G^0\subseteq
  G\), so that we may apply
  Theorem~\ref{the:graded_groupoid_versus_action}.  For \(t\in S\),
  let \(G_t\defeq f(t)\subseteq G^1\); these are open subsets because
  each~\(f(t)\) is a bisection.  Since~\(f\) is a semigroup
  homomorphism, \ref{enum:Gra1} and~\ref{enum:Gra2} hold.  The other
  two conditions are exactly the technical assumptions of the
  corollary.  Thus these two assumptions are equivalent to
  \((G_t)_{t\in S}\) being an \(S\)\nb-grading on~\(G\).  If they
  hold, then Theorem~\ref{the:graded_groupoid_versus_action} says that
  \(G\cong G_1\rtimes S = Z\rtimes S\).  Conversely, the
  transformation groupoid \(Z\rtimes S\) is \(S\)\nb-graded by
  Theorem~\ref{the:graded_groupoid_versus_action}, so if \(G\cong
  Z\rtimes S\), then it satisfies the two technical assumptions.
\end{proof}

A subsemigroup \(S\subseteq\Bis(G)\) with the properties required in
Corollary~\ref{cor:wide_action} is called \emph{wide}.
Corollary~\ref{cor:wide_action} explains why they appear so frequently
(see, for instance, \cites{Exel:Inverse_combinatorial,
  Sieben-Quigg:ActionsOfGroupoidsAndISGs,BussExel:Regular.Fell.Bundle}).
\cite{Exel:Inverse_combinatorial}*{Proposition~5.4} already shows that
\(Z\rtimes S=G\) if~\(S\) is wide, but we have not seen the converse
statement yet.

Since the proof of Theorem~\ref{the:action_on_space} is not explicit,
we give another pedestrian proof.

Let \(\theta_t\colon D_{t^*t}\to D_{tt^*}\) for \(t\in S\) give an
action of~\(S\) on~\(Z\) by partial homeomorphisms.  This is a
groupoid isomorphism from~\(D_{t^*t}\) to~\(D_{tt^*}\), which we turn
into a partial equivalence from~\(Z\) to itself as in
Example~\ref{exa:iso_to_partial_equivalence}.  Here this means that we
take \(X'_t\defeq D_{tt^*}\) with anchor maps \(\rg'(z)\defeq z\),
\(\s'(z)\defeq \theta_t^{-1}(z)\).  Since all arrows in~\(Z\) are
units, the range and source maps determine the partial equivalence.
The homeomorphism~\(\theta_t\) gives a bibundle isomorphism
from~\(X_t'\) to \(X_t= D_{t^*t}\) with \(\rg(z)\defeq \theta_t(z)\)
and \(\s(z)\defeq z\).  The comparison with Exel's groupoid is more
obvious for the second choice, which we take from now on.

There is an obvious homeomorphism
\[
X_t\times_Z X_u\congto \{z\in D_{tt^*}\mid
\theta_t(z)\in D_{uu^*}\} = D_{(tu)^*(tu)},
\]
such that the range and source maps are \(\theta_{tu}\) and the
inclusion map, respectively.  We choose this isomorphism
for~\(\mu_{t,u}\) to define our action by partial equivalences.
Actually, this is no choice at all because the range and source maps
are injective here, so there is at most one bibundle map between any
two partial equivalences.  (We will see more groupoids with this
property in Section~\ref{sec:local_centraliser}.)  Hence the
associativity condition in the definition of an inverse semigroup
action holds automatically.  Thus we have turned an action by partial
homeomorphisms on~\(Z\) into an action by partial equivalences
on~\(Z\), viewed as a topological groupoid.

Our construction of the transformation groupoid above is
\emph{exactly} the construction of the groupoid of germs in this
special case, so the isomorphism between \(Z\rtimes S\) and the
groupoid of germs from~\cite{Exel:Inverse_combinatorial} is trivial.

Next we check that every partial equivalence~\(X\) of~\(Z\) is
isomorphic to one coming from a partial homeomorphism.  Since all
arrows in~\(Z\) are units, we have \(X/Z=X=Z\backslash X\).  Hence the
anchor maps \(G^0\leftarrow X\to H^0\) are continuous, open and
injective by condition~\ref{enum:Peq3} in
Proposition~\ref{pro:equivalence}.  The map \(\theta\defeq
\rg\circ\s^{-1}\colon \s(X)\to \rg(X)\) is a partial homeomorphism
from~\(G^0\) to~\(H^0\), and \(\rg\colon X\to \rg(X)\) is an
isomorphism of partial equivalences from~\(X\) to~\(X_\theta\).

Since there is always only one isomorphism between partial
equivalences coming from the same partial homeomorphism, an inverse
semigroup action on~\(Z\) is determined uniquely by the isomorphism
classes of the~\(X_t\), which are in bijection with partial
homeomorphisms of~\(Z\).  This proves the first statement in
Theorem~\ref{the:action_on_space}.

\subsection{Morita invariance of actions by partial equivalences}
\label{sec:Morita_invariance}

\begin{proposition}
  \label{pro:actions_equivalent}
  Let~\(Y\) be an equivalence from~\(H\) to~\(G\), and let
  \((X_t,\mu_{t,u})\) be an action of an inverse semigroup with unit
  on~\(G\).  Let \(X'_t\defeq Y\times_G X_t\times_G Y^*\) and let
  \(\mu'_{t,u}\colon X'_t\times_H X'_u\to X'_{tu}\) be the composite
  isomorphism
  \begin{multline*}
    Y\times_G X_t \times_G Y^*\times_H Y\times_G X_u\times_G Y^*
    \congto Y\times_G X_t \times_G G^1\times_G X_u\times_G Y^*
    \\ \congto Y\times_G X_t \times_G X_u\times_G Y^*
    \xrightarrow[\cong]{\mu_{t,u}} Y\times_G X_{tu}\times_G Y^*,
  \end{multline*}
  where the first two isomorphisms are the canonical ones from
  Proposition~\textup{\ref{pro:properties_of_dual}} and
  Lemma~\textup{\ref{lem:peq_composition_associative_unital}}.
  Then~\(\mu'_{t,u}\) is an action of~\(S\) on~\(H\) by partial
  equivalences.  Its transformation groupoid \(H\rtimes S\) is
  equivalent to \(G\rtimes S\).

  When we translate the action on~\(Y\) back to~\(X\) using the
  inverse equivalence~\(Y^*\), we get an action on~\(G\) that is
  isomorphic to the original one.
\end{proposition}

\begin{proof}
  More precisely, \(X'_1\) as defined above is only isomorphic
  to~\(H^1\) in a very obvious way.  We should only use the above
  definition of~\(X'_t\) for \(t\neq1\) and let \(X'_1\defeq H^1\) for
  \(t=1\), and let \(\mu'_{1,t}\) and~\(\mu'_{t,1}\) be the canonical
  isomorphisms.  We should also put in associators for the composition
  of partial equivalences, which only cause notational complications,
  however.  Up to these technicalities, it is clear that the
  maps~\(\mu'\) inherit associativity from the maps~\(\mu\).  The
  action that we get by translating~\(\mu'\) back to~\(G\)
  with~\(Y^*\) is canonically isomorphic to the original action
  because \(Y^*\times_H Y\cong G^1\).

  It remains to prove the equivalence of the transformation groupoids
  \(G\rtimes S\) and \(H\rtimes S\).  Here we use the linking
  groupoid~\(L\) of the equivalence; its object space is
  \(L^0=G^0\sqcup H^0\), its arrow space is \(G^1\sqcup Y \sqcup
  Y^*\sqcup H^1\), its range and source maps are \(\rg\) and~\(\s\) on
  each component, and its multiplication consists of the
  multiplications in \(G\) and~\(H\), the \(G,H\)-bibundle structure
  on~\(Y\), the \(H,G\)-bibundle structure on~\(Y^*\), and the
  canonical isomorphisms \(Y\times_G Y^*\congto H^1\) and
  \(Y^*\times_H Y\congto G^1\) from
  Proposition~\ref{pro:properties_of_dual}.  This gives a topological
  groupoid~\(L\).  There is a canonical right action of~\(L\) on
  \(G^1\sqcup Y = \rg^{-1}(G^1)\subseteq L^1\) that provides an
  equivalence from~\(L\) to~\(G\) when combined with the left
  actions of~\(G\) on \(G^1\) and~\(Y\); there is a similar canonical
  equivalence \(H^1\sqcup Y^*\) from~\(L\) to~\(H\).

  We may transport the \(S\)\nb-action on~\(G\) to~\(L\) because it is
  equivalent to~\(G\).  When we transport this action on~\(L\) further
  to~\(H\), we get the action described above because the composite
  equivalence \((G^1\sqcup Y)\times_L (H^1\sqcup Y^*)^*\) from~\(H\)
  to~\(G\) is isomorphic to~\(Y\).

  The action on~\(L\) is given by bibundles
  \begin{multline*}
    (G^1\sqcup Y)^* \times_G X_t \times_G (G^1\sqcup Y)
    \\\cong X_t
    \sqcup (X_t \times_G Y)
    \sqcup (Y^* \times_G X_t)
    \sqcup (Y^* \times_G X_t \times_G Y),
  \end{multline*}
  where we cancelled factors of~\(G^1\) using
  Lemma~\ref{lem:peq_composition_associative_unital}.  When we
  restrict the transformation groupoid \(L\rtimes S\) to
  \(G^0\subseteq L^0\) or to \(H^0\subseteq L^0\), then we only pick
  the components \(X_t\) and \(Y^* \times_G X_t \times_G Y\) in the
  above decomposition, so we get the transformation groupoids
  \(G\rtimes S\) and \(H\rtimes S\), respectively.  Routine
  computations show that the other two parts \(\rg^{-1}(G^0)\cap
  \s^{-1}(H^0)\) and \(\rg^{-1}(H^0)\cap \s^{-1}(G^0)\) of \(L\rtimes
  S\) give an equivalence from~\(H\rtimes S\) to~\(G\rtimes S\), such
  that \(L\rtimes S\) is the resulting linking groupoid.

  It can be shown with less routine computations that the embedding
  \(G\rtimes S\into L\rtimes S\) is fully faithful and essentially
  surjective.  We checked ``fully faithful'' above.  Being
  ``essentially surjective'' means that the map
  \(G^0\times_{\subset,L^0,\rg} L^1\to L^0\), \((x,l)\mapsto \s(l)\),
  is open and surjective.  It is open because \(\rg\colon L^1\to L^0\)
  is open and \(G^0\subset L^0\) is open, and surjective because
  already \(G^0\times_{L^0} Y\subset G^0\times_{L^0} L^1\) surjects
  onto~\(H^0\).  Since both \(G\rtimes S\into L\rtimes S\) and
  \(H\rtimes S\into L\rtimes S\) are fully faithful and essentially
  surjective, they induce equivalence bibundles by
  \cite{Meyer-Zhu:Groupoids}*{Proposition 6.8}, which we may compose
  to an equivalence from \(H\rtimes S\) to \(G\rtimes S\).  Of course,
  this gives the same equivalence as the argument above.
\end{proof}

\begin{corollary}
  \label{cor:action_on_covering_groupoid}
  Let~\(S\) be an inverse semigroup with unit.  Let \(f\colon X\to Z\)
  be an open continuous surjection and let~\(G(f)\) be its covering
  groupoid, see Definition~\textup{\ref{def:covering_groupoid}}.  Then
  \(S\)\nb-actions by partial equivalences on~\(G(f)\)
  are canonically equivalent to \(S\)\nb-actions on~\(Z\) by partial
  homeomorphisms, such that \(G(f)\rtimes S\) is equivalent to
  \(Z\rtimes S\).
\end{corollary}

Here ``equivalent'' means an equivalence of categories, where the
arrows are isomorphisms of \(S\)\nb-actions that fix the underlying
groupoid.

\begin{proof}
  \(G(f)\) is canonically equivalent to~\(Z\) viewed as a groupoid, so
  the assertion follows from Theorem~\ref{the:action_on_space} and
  Proposition~\ref{pro:actions_equivalent}.
\end{proof}

In particular, Corollary~\ref{cor:action_on_covering_groupoid} applies
to the \v{C}ech groupoid~\(G_\OCover\) of an open
covering~\(\OCover\) of a locally Hausdorff space~\(Z\) by
Hausdorff open subsets.  Thus we may replace an \(S\)\nb-action by
partial homeomorphisms on a locally Hausdorff space~\(Z\) by an
``equivalent'' action by partial equivalences on a Hausdorff
groupoid~\(G_\OCover\), and the resulting transformation groupoids
\(Z\rtimes S\) and \(G_\OCover\rtimes S\) are equivalent.

The quickest way to describe the resulting \(S\)\nb-action
on~\(G_\OCover\) explicitly is by describing
\(G_\OCover\rtimes S\) and an \(S\)\nb-grading on it.  Let
\(X\defeq \bigsqcup_{U\in\OCover} U\) and let \(p\colon X\to Z\)
be the canonical map, which is an open surjection.  The pull-back
\(p^*(Z\rtimes S)\) of~\(Z\rtimes S\) along~\(p\) is a groupoid
with object space~\(X\), arrow space \(X\times_{p,Z,\rg} (Z\rtimes
S)^1 \times_{\s,Z,p} X\), \(\rg(x_1,g,x_2) = x_1\),
\(\s(x_1,g,x_2)=x_2\), and \((x_1,g,x_2)\cdot (x_2,h,x_3) =
(x_1,g\cdot h,x_3)\) (see \cite{Meyer-Zhu:Groupoids}*{Example 3.13}).
Let
\[
X_t \defeq \{(x_1,g,x_2)
\in X\times_{p,Z,\rg} (Z\rtimes S)^1 \times_{\s,Z,p} X \mid
g\in t\}.
\]

\begin{proposition}
  \label{pro:desingularise_transformation_groupoid}
  The subspaces \(X_t\subseteq p^*(Z\rtimes S)^1\) form an
  \(S\)\nb-grading on~\(p^*(Z\rtimes S)\).  The resulting
  \(S\)\nb-graded groupoid is the transformation groupoid for the
  \(S\)\nb-action on~\(G_\OCover\) that we get by translating the
  \(S\)\nb-action on~\(Z\) along the equivalence to~\(G_\OCover\).
\end{proposition}

\begin{proof}
  The subspaces \(X_t\) form an \(S\)\nb-grading because the
  bisections \(t\in S\) give an \(S\)\nb-grading on \(Z\rtimes S\)
  and~\(p\) is surjective.  Hence they describe an \(S\)\nb-action
  on~\(G_\OCover\).  The equivalence from~\(G_\OCover\)
  to~\(Z\) is given by the canonical action of~\(G_\OCover\)
  on~\(G_\OCover^0=X\) and the projection \(p\colon X\to Z\).
  Hence~\(X_t\) is exactly what we get when we translate \(t\subseteq
  Z\rtimes S\) along the equivalence.
\end{proof}

For instance, let~\(H\) be an étale groupoid with locally Hausdorff
arrow space and let~\(S\) be some inverse semigroup of bisections with
\(H\cong H^0\rtimes S\); we could take \(S=\Bis(H)\).  Let \(Z=H^1\)
with the action of~\(H\) by left multiplication.  This induces an
action of~\(S\) on~\(Z\).  Its transformation groupoid \(H^1\rtimes
S\) is \(H^1\rtimes H\) with the obvious \(S\)\nb-grading by
\(H^1\times_{H^0} t\) for \(t\in S\).

The left multiplication action of~\(H\) on~\(H^1\) with the bundle
projection \(\s\colon H^1\to H^0\) is a trivial principal bundle.  In
particular, the transformation groupoid \(H^1\rtimes S \cong
H^1\rtimes H\) is isomorphic to the covering groupoid of the cover
\(\s\colon H^1\onto H^0\).  Hence it is equivalent to the
space~\(H^0\), viewed as a groupoid with only unit arrows.  The
\(S\)\nb-grading on \(H^1\rtimes S\) does, however, not carry over
to~\(H^0\).

If we replace the \(S\)\nb-action on~\(H^1\) by an equivalent
\(S\)\nb-action on~\(G_\OCover\) for a Hausdorff open cover
of~\(H^1\), then the transformation groupoid \(G_\OCover\rtimes S\) is
equivalent to \(H^1\rtimes S\) and hence also equivalent to the
space~\(H^0\).  In particular, the groupoid \(G_\OCover\rtimes S\) is
\emph{basic} (see Section~\ref{sec:basic_vs_free_proper}).  If a
groupoid is equivalent to a space, then this space has to be its orbit
space.  So if~\(H^0\) is Hausdorff, then the groupoid
\(G_\OCover\rtimes S\) is a free and proper, Hausdorff groupoid by
Proposition~\ref{pro:proper_action}.

\subsection{Local centralisers}
\label{sec:local_centraliser}

We are going to show that for many groupoids~\(G\) the
bibundles~\(X_t\) already determine the multiplication
maps~\(\mu_{t,u}\) and thus the inverse semigroup action.  This
happens, among others, for essentially principal topological groupoids
(meaning that the isotropy group bundle has no interior; see \cite{Renault:Cartan.Subalgebras}) and for
groups with trivial centre.

\begin{definition}
  \label{def:local_centraliser}
  A \emph{local centraliser} of~\(G\) is a map \(\gamma\colon U\to
  G^1\) defined on an open \(G\)\nb-invariant subset~\(U\) of~\(G^0\)
  with \(\s(\gamma(x))=\rg(\gamma(x))=x\) for all \(x\in U\) and
  \(\gamma(\rg(g))\cdot g=g\cdot \gamma(\s(g))\) for all \(g\in G\).
  We say that~\(G\) has \emph{no local centralisers} if all local
  centralisers are given by \(\gamma(x)=1_x\) for \(x\in U\) and
  some~\(U\) as above.
\end{definition}

Local centralisers defined on the same subset~\(U\) form an Abelian
group under pointwise multiplication.  All local centralisers form an
Abelian inverse semigroup.  It is the centre of~\(\Bis(G)\) if~\(G\)
is étale.

\begin{lemma}
  \label{lem:centralisers_bibundles}
  Let~\(X\) be a partial equivalence from~\(H\) to~\(G\).  Then
  \(\Map(X,X)\) is isomorphic to the group of local centralisers
  of~\(G\) defined on~\(\rg(X)\), and to the group of local
  centralisers of~\(H\) defined on~\(\s(X)\).

  If~\(G\) has no local centralisers and \(X\) and~\(Y\) are partial
  equivalences from~\(G\) or to~\(G\), then there is at most one
  bibundle map \(X\to Y\), so bibundle isomorphisms are unique if they
  exist.
\end{lemma}

\begin{proof}
  The two descriptions of \(\Map(X,X)\) are equivalent by
  taking~\(X^*\), so we only prove one.  Every bibundle map \(X\to X\)
  is invertible by Proposition~\ref{pro:characterise_bibundle_isom}.
  The canonical group homomorphisms
  \begin{multline*}
    \Map(X,X) \xrightarrow{\blank\times_H X^*}
    \Map(X\times_H X^*,X\times_H X^*)
    \cong \Map(G^1_{\rg(X)},G^1_{\rg(X)})
    \\ \xrightarrow{\blank\times_G X}
    \Map(X,X)
  \end{multline*}
  are inverse to each other by the proof of
  Proposition~\ref{pro:properties_of_dual}.  Thus it remains to
  identify the set \(\Map(G^1_U,G^1_U)\) for an open
  \(G\)\nb-invariant subset~\(U\) of~\(G^0\) with the group of local
  centralisers defined on~\(U\).  We may view~\(G^1_U\) as the
  equivalence from~\(G^1_U\) to itself associated to the identity
  functor on~\(G^1_U\).  We described all bibundle isomorphisms
  between such equivalences in
  Example~\ref{exa:iso_to_partial_equivalence}.  Specialising
  Example~\ref{exa:iso_to_partial_equivalence} to the automorphisms of
  the identity functor gives exactly the local centralisers defined
  on~\(U\).  A quick computation shows that the composition of
  bibundle isomorphisms corresponds to the pointwise multiplication of
  local centralisers.

  Let \(f_1,f_2\colon X\to Y\) be bibundle maps.  Then both are
  bibundle isomorphisms \(X\to Y|_{\s(X)}\), and we may form a
  composite bibundle isomorphism \(f_2^{-1}\circ f_1\colon X\to X\).
  Since there are no local centralisers, the first part of the lemma
  shows that this is the identity map, so \(f_1=f_2\).  In particular,
  if two partial equivalences \(G\to H\) or \(H\to G\) are isomorphic,
  then the isomorphism is unique.
\end{proof}

Recall that~\(\widetilde{\Peq}(G)\) denotes the inverse semigroup of
\emph{isomorphism classes} of partial equivalences on~\(G\).

\begin{theorem}
  \label{the:no_local_centralisers}
  Let~\(G\) be a topological groupoid without local centralisers.  An
  action of an inverse semigroup~\(S\) on~\(G\) is equivalent to a
  homomorphism \(S\to\widetilde{\Peq}(G)\).  More precisely,
  isomorphism classes of \(S\)\nb-actions on~\(G\) by partial
  equivalences are in canonical bijection with homomorphisms
  \(S\to\widetilde{\Peq}(G)\).
\end{theorem}

\begin{proof}
  A homomorphism \(f\colon S\to\widetilde{\Peq}(G)\) gives us
  bibundles~\(X_t\) with \(X_t\times_G X_u\cong X_{tu}\) and
  \(X_1\cong G^1\); we may as well assume \(X_1=G^1\).  By
  Lemma~\ref{lem:centralisers_bibundles}, the isomorphisms
  \(\mu_{t,u}\colon X_t\times_G X_u\congto X_{tu}\) above are unique,
  so there is no need to specify them.  The conditions \ref{enum:APE3}
  and~\ref{enum:APE4} hold because any two parallel bibundle
  isomorphisms are equal.  Thus~\(f\) determines an \(S\)\nb-action by
  partial equivalences.  Conversely, an action by partial equivalences
  determines such a homomorphism by taking the isomorphism classes of
  the~\(X_t\) and forgetting the~\(\mu_{t,u}\).  Since isomorphisms of
  partial equivalences are unique if they exist, this forgetful
  functor is actually not forgetting anything here, so we get a
  bijection between isomorphism classes of actions by partial
  equivalences and homomorphisms \(S\to\widetilde{\Peq}(G)\).
\end{proof}

The results in this section are inspired by the notion of a
``quasi-graphoid'' used by Debord in~\cite{Debord:Holonomy}.  Debord
already treats partial equivalences of groupoids as arrows between
groupoids and uses them to glue together groupoids constructed
locally.  She restricts, however, to a situation where bibundle
isomorphisms are uniquely determined.  Even more, she wants the range
and source maps to already determine a partial equivalence uniquely.
For this, she assumes that a smooth map \(\gamma\colon U\to G^1\)
defined on an open subset~\(U\) of~\(G^0\) has to be the unit section
already if it only satisfies \(\s(\gamma(x))=\rg(\gamma(x))=x\) for
all \(x\in U\).  This condition holds for holonomy groupoids of
foliations -- even for the mildly singular foliations that she is
considering.

\subsection{Decomposing proper Lie groupoids}
\label{sec:proper_Lie}

A manifold may be constructed by taking a disjoint union of local
charts and gluing them together along the coordinate change maps,
which are partial homeomorphisms, or diffeomorphisms in the smooth
case.  When constructing groupoids locally, it is more likely that the
coordinate change maps are no longer partial isomorphisms but only
partial equivalences.  Actually, it may well be that the local pieces
are, to begin with, only local groupoids and not groupoids
(see~\cite{Debord:Holonomy}); this is not covered by our theory.
Therefore, we know no good examples where groupoids have been
constructed by gluing together smaller groupoids along partial
equivalences.

Instead, we take a groupoid as given and analyse it using local
information.  The local information should say that the groupoid
locally is \emph{equivalent} to one of a particularly simple form.
Then the groupoid is globally equivalent to a transformation groupoid
for an inverse semigroup action by partial equivalences on a disjoint
union of groupoids having the desired simple form.

We now get more concrete and consider a proper Lie groupoid~\(H\).  To
formulate stronger results, we shall work with (partial) equivalences
of Lie groupoids in this section; that is, spaces are replaced by
smooth manifolds, continuous maps by smooth maps, and open maps by
submersions.  This does not change the theory significantly,
see~\cite{Meyer-Zhu:Groupoids}.

First we formulate the local linearisability of proper Lie groupoids.
This was conjectured by Weinstein~\cite{Weinstein:Linearization} and
proved by Zung~\cite{Zung:Proper_linearization}.  Both authors try to
describe the local structure of proper Lie groupoids \emph{up to
  isomorphism}.  Following Trentinaglia~\cite{Trentinaglia:Thesis}, we
only aim for a description up to Morita equivalence:

\begin{theorem}
  \label{the:local_linearizable}
  Let~\(H\) be a proper Lie groupoid.  For every \(x\in H^0\) there
  are an open \(H\)\nb-invariant neighbourhood~\(U_x\) of~\(x\)
  in~\(H^0\), a linear representation of the stabiliser group~\(H_x\)
  on a finite-dimensional vector space~\(W_x\), and a Lie groupoid
  equivalence from the transformation groupoid~\(W_x\rtimes H_x\)
  to~\(H_{U_x}\).
\end{theorem}

The vector space~\(W_x\) is the normal bundle to the \(H\)\nb-orbit
\(H x\) of~\(x\), with its canonical representation of~\(H_x\).

Weinstein and Zung impose extra assumptions on~\(H\) to
describe~\(H_{U_x}\) up to isomorphism.  The argument in
\cite{Trentinaglia:Thesis}*{Section 4} shows how to deduce
Theorem~\ref{the:local_linearizable} quickly from
\cite{Zung:Proper_linearization}*{Theorem 2.3} without extra
assumptions.

Actually, we do not need~\(H\) to be proper.  Since we only need local
structure, it is enough for~\(H\) to be \emph{locally proper}, that
is, each \(x\in H^0\) has an \(H\)\nb-invariant open
neighbourhood~\(U\) such that~\(H_U\) is proper; this allows the orbit
space~\(H^0/H\) to be a locally Hausdorff but non-Hausdorff manifold.

Assume now that~\(H\) is a locally proper groupoid.  By
Theorem~\ref{the:local_linearizable}, there is a
covering~\(\OCover\) of~\(H^0\) by open, \(H\)\nb-invariant
subsets and, for each \(U\in\OCover\), a Lie groupoid
equivalence~\(X_U\) from a transformation groupoid \(W_U\rtimes G_U\)
for a compact Lie group~\(G_U\) and a linear representation~\(W_U\)
of~\(G_U\) to the restriction~\(H_U\).  Now let
\[
G\defeq \bigsqcup_{U\in \OCover} W_U\rtimes G_U.
\]
This disjoint union is a groupoid with object space~\(\bigsqcup W_U\).

Let~\(K\) be the covering groupoid of~\(H^0\) for the
covering~\(\OCover\).  Since \(H_U|_{U\cap V}=H_{U\cap V}=
H_V|_{U\cap V}\), the inverse semigroup \(S\defeq \Bis(K)\) acts on
\(\bigsqcup_{U\in\OCover} H_U\): each element of~\(\Bis(K)\) acts
by the identity equivalence between the appropriate restrictions of
\(H_U\) and~\(H_V\), and all the multiplication maps are the canonical
isomorphisms.  The disjoint union \(X\defeq
\bigsqcup_{U\in\OCover} X_U\) gives an equivalence from~\(G\) to
\(\bigsqcup_{U\in\OCover} H_U\), so we may transfer this
\(S\)\nb-action to~\(G\).

We make the action on~\(G\) more concrete.  Any bisection of~\(K\) is
a disjoint union of bisections of the form
\[
(U_1,D,U_2) \defeq \{(U_1,x,U_2) \mid x\in D\}
\]
for \(U_1,U_2\in\OCover\) and an open subset \(D\subseteq U_1\cap
U_2\).  The product \((U_1,D_1,U_2)\cdot (U'_2,D_2,U_3)\) is empty if
\(U_2\neq U_2'\), and is equal to \((U_1,D_1\cap D_2,U_3)\) if \(U_2=
U_2'\).

The partial equivalence~\(X_{U_1,D,U_2}\) on~\(G\) associated to
\((U_1,D,U_2)\) is the composite partial equivalence
\[
G\supseteq W_{U_1}\rtimes G_{U_1}
\xrightarrow{{}_D|X_{U_1}^*} H_D
\xrightarrow{X_{U_2}|_D} W_{U_2}\rtimes G_{U_2} \subseteq G.
\]
The composite of \(X_{U_1,D_1,U_2}\) and \(X_{U_2',D_2,U_3}\) is
clearly empty for \(U_2\neq U_2'\), as it should be.  If \(U_2=U_2'\),
then there is a canonical isomorphism of partial equivalences
\[
\mu_{(U_1,D_1,U_2),(U_2,D_2,U_3)}\colon X_{U_1,D_1,U_2}\times_G
X_{U_2,D_2,U_3} \to X_{U_1,D_1\cap D_2,U_3},
\]
using the restriction of the canonical pairing \(X_{U_2}\times_G
X_{U_2}^* \to H_{U_2}\) to remove the extra two factors in the middle.
This is exactly what happens if we translate the ``trivial'' action
of~\(S\) on \(\bigsqcup H_U\) described above to~\(G\) along the
equivalence \(\bigsqcup X_U\).

\begin{theorem}
  \label{the:global_structure_proper_Lie}
  The locally proper Lie groupoid~\(H\) is equivalent to the
  transformation groupoid~\(G\rtimes S\) for the action of~\(S\)
  on~\(G\) described above.
\end{theorem}

\begin{proof}
  Since we constructed the action of~\(S\) on~\(G\) by translating the
  action on \(\bigsqcup_{U\in\OCover} H_U\),
  Proposition~\ref{pro:actions_equivalent} shows that \(G\rtimes S\)
  is equivalent to \(\bigsqcup_{U\in\OCover} H_U\rtimes S\).
  Since~\(S\) acts ``trivially'' on \(\bigsqcup H_U\), this
  transformation groupoid is easy to understand: it is the
  pull-back~\(p^*(H)\) of~\(H\) for the canonical map \(p\colon
  \bigsqcup_{U\in\OCover} U\to H^0\).  Since~\(p\) is a surjective
  submersion, \(p^*(H)\) is equivalent to~\(H\).
\end{proof}

As a result, any locally proper Lie groupoid is equivalent to a
transformation groupoid for an inverse semigroup action on a disjoint
union of linear actions of compact groups.  Such transformation
groupoids need not be locally proper, however, so we do not have a
characterisation of locally proper Lie groupoids.  The groupoid
\(G\rtimes S\) is étale if and only if~\(G\) is, if and only if the
stabilisers~\(H_x\) are finite.  This means that~\(H\) is an orbifold
(see~\cite{Moerdijk:Orbifolds_groupoids}).

\section{Inverse semigroup actions on \texorpdfstring{$C^*$}{C*}-algebras}
\label{sec:S_acts_on_Cstar}

We now define inverse semigroup actions on \(\Cst\)\nb-algebras by
Hilbert bimodules, in parallel to actions on groupoids by partial
equivalences.

\begin{definition}
  \label{def:Hilbert_bimodule}
  A \emph{Hilbert \(A,B\)-bimodule}~\(\Hilm\) is a left Hilbert
  \(A\)\nb-module and a right Hilbert \(B\)\nb-module such that the
  left and right multiplications commute, and \(\BRAKET{x}{y}_A\cdot z
  = x\cdot \braket{y}{z}_B\) for all \(x,y,z\in \Hilm\).  A
  \emph{Hilbert \(A,B\)-bimodule map} is a bimodule map that also
  intertwines both inner products.
\end{definition}

Let~\(\Hilm\) be a Hilbert \(A,B\)-bimodule.  Let \(I\idealin A\)
and \(J\idealin B\) be the closed linear spans of the elements
\(\BRAKET{x}{y}_A\) and~\(\braket{x}{y}_B\) with \(x,y\in\Hilm\),
respectively.  These are closed ideals in \(A\) and~\(B\),
and~\(\Hilm\) is an \(I,J\)\nb-imprimitivity bimodule by restricting
the left multiplications to \(I\) and~\(J\).  Ideals in a
\(\Cst\)\nb-algebra are in bijection with open subsets of its
primitive ideal space, so ideals are the right analogues of open
invariant subsets of groupoids.  Hence we denote the ideals \(I\)
and~\(J\) above as \(I\defeq \rg(\Hilm)\) and \(J\defeq \s(\Hilm)\),
and we think of Hilbert \(A,B\)-bimodules as partial Morita
equivalences from~\(B\) to~\(A\).

Given an ideal \(K\idealin A\), we define the restriction of a
Hilbert bimodule~\(\Hilm\) to~\(K\) as \({}_K|\Hilm\defeq K\cdot
\Hilm\subseteq \Hilm\), which is canonically isomorphic to
\(K\otimes_A \Hilm\).  We restrict to ideals in~\(B\) in a similar
way.

The left action of~\(A\) on a Hilbert bimodule is by a nondegenerate
\Star{}ho\-mo\-mor\-phism \(A\to\Bound(\Hilm)\) into the adjointable
operators on~\(\Hilm\).  Thus a Hilbert \(A,B\)-bimodule becomes a
correspondence by forgetting the left inner
product.

\begin{lemma}
  \label{lem:corr_to_Hil_bimod}
  A correspondence~\(\Hilm\) carries a Hilbert
  bimodule structure if and only if there is an ideal \(I\idealin
  A\) such that the left action \(\varphi\colon A\to\Bound(\Hilm)\)
  restricts to an isomorphism from~\(I\) onto~\(\Comp(\Hilm)\).  This
  ideal and the left inner product are uniquely determined by the
  correspondence.
\end{lemma}

\begin{proof}
  First let~\(\Hilm\) be a Hilbert bimodule.  Then~\(\Hilm\) is an
  imprimitivity bimodule from~\(\s(\Hilm)\) to~\(\rg(\Hilm)\),
  so~\(\varphi|_{\rg(\Hilm)}\) is an isomorphism from~\(\rg(\Hilm)\)
  onto~\(\Comp(\Hilm)\).  If \(I\idealin A\) is another ideal
  with \(\varphi(I) = \Comp(\Hilm)\), then \(\varphi(\rg(\Hilm)\cdot I)
  = \Comp(\Hilm)\) as well.  Thus~\(\rg(\Hilm)\) is the minimal ideal
  that~\(\varphi\) maps onto~\(\Comp(\Hilm)\), and the only one on
  which this happens isomorphically.  Thus~\(\rg(\Hilm)\) is already
  determined by the underlying correspondence.

  Let~\(\Hilm'\) be another Hilbert \(A,B\)-bimodule with the same
  underlying correspondence as~\(\Hilm\) and with left
  \(A\)\nb-valued inner product~\(\BRAKET{x}{y}'_A\).  Then
  \[
  \varphi(\BRAKET{x}{y}'_A)z = x\cdot \braket{y}{z}_B
  = \varphi(\BRAKET{x}{y}_A)z
  \]
  for all \(x,y,z\in \Hilm\).  Since \(\rg(\Hilm)=\rg(\Hilm')\)
  depends only on the correspondence and the restriction
  of~\(\varphi\) to~\(\rg(\Hilm)\) is faithful, we get
  \(\Hilm=\Hilm'\) as Hilbert bimodules.

  Now let~\(\Hilm\) be a correspondence and let \(I\idealin A\)
  be an ideal that is mapped isomorphically onto~\(\Comp(\Hilm)\).
  Transfer the usual \(\Comp(\Hilm)\)-valued left inner product
  on~\(\Hilm\) through this isomorphism to one with values in
  \(A\supseteq I\).  This turns~\(\Hilm\) into a Hilbert
  \(A,B\)-bimodule.
\end{proof}

\begin{proposition}
  \label{pro:Hilbert_bimodule_map}
  Let \(\Hilm\) and~\(\Hilm'\) be Hilbert \(A,B\)-bimodules.  If there
  is a Hilbert bimodule map \(f\colon \Hilm\to\Hilm'\), then
  \(\s(\Hilm)\subseteq \s(\Hilm')\) and \(\rg(\Hilm)\subseteq
  \rg(\Hilm')\).  Such a Hilbert bimodule map is an isomorphism
  from~\(\Hilm\) onto the submodule \(\Hilm'\cdot \s(\Hilm) =
  \rg(\Hilm)\cdot \Hilm'\) in~\(\Hilm'\).  So it is an isomorphism
  onto~\(\Hilm'\) if and only if \(\s(\Hilm')\subseteq \s(\Hilm)\), if
  and only if \(\rg(\Hilm')\subseteq \rg(\Hilm)\), if and only if the
  map \(\Comp(\Hilm)\to \Comp(\Hilm')\) induced by~\(f\) is an
  isomorphism.
\end{proposition}

\begin{proof}
  Since the norm on a Hilbert bimodule is generated by the
  inner products, Hilbert bimodule maps are norm isometries and thus
  injective.  Moreover,
  \[
  f(\Hilm)
  = f(\rg(\Hilm)\cdot \Hilm)
  = \rg(\Hilm)\cdot f(\Hilm)
  \subseteq \rg(\Hilm)\cdot\Hilm'.
  \]
  Thus \(\rg(\Hilm')\subseteq \rg(\Hilm)\) is necessary for~\(f\) to
  be an isomorphism.  Conversely, if
  \(\rg(\Hilm')\subseteq\rg(\Hilm)\), then even \(\rg(\Hilm') =
  \rg(\Hilm)\) because a bimodule map preserves the left inner
  product.  Then the map from \(\rg(\Hilm)\cong\Comp(\Hilm)\) to
  \(\rg(\Hilm')\cong\Comp(\Hilm')\) that sends \(\ket{\xi}\bra{\eta}\)
  to \(\ket{f(\xi)}\bra{f(\eta)}\) for \(\xi,\eta\in\Hilm\) is an
  isomorphism \(\Comp(\Hilm)\cong\rg(\Hilm')\cong\Comp(\Hilm')\).
  Since \(\Comp(\Hilm')\cdot\Hilm'=\Hilm'\), the linear span of
  elements of the form \(\ket{f(\xi)}\bra{f(\eta)}\zeta' = f(\xi)\cdot
  \braket{f(\eta)}{\zeta'}\) for \(\xi,\eta\in\Hilm\),
  \(\zeta'\in\Hilm'\) is dense in~\(\Hilm'\).  Since \(f(\Hilm)\) is a
  right \(B\)\nb-module, this implies that~\(f\) is surjective.  Hence
  it is an isomorphism of Hilbert bimodules.  A similar argument for
  the right inner product instead of the left one shows that all the
  listed conditions for~\(f\) are indeed equivalent to~\(f\) being an
  isomorphism.

  If \(\rg(\Hilm)\neq\rg(\Hilm')\), then we may restrict~\(f\) to a
  Hilbert bimodule map \(\Hilm\to \rg(\Hilm)\cdot\Hilm'\).  Since
  \(\rg(\Hilm)\cdot \rg(\Hilm)\cdot \Hilm' = \rg(\Hilm)\cdot\Hilm'\),
  this is an isomorphism by the first statement.  A similar argument
  on the other side shows that~\(f\) is an isomorphism onto
  \(\Hilm'\cdot\s(\Hilm)\), so \(\Hilm'\cdot \s(\Hilm) =
  \rg(\Hilm)\cdot \Hilm'\).
\end{proof}

A Hilbert \(A,B\)-bimodule~\(\Hilm\) has a dual Hilbert
\(B,A\)-bimodule~\(\Hilm^*\), where we exchange left and right
structures using adjoints: \(b\cdot x^* \cdot a \defeq (a^*\cdot
x\cdot b^*)^*\) for \(a\in A\), \(b\in B\), \(x\in\Hilm\), and
\(\braket{x^*}{y^*}_A = \BRAKET{y}{x}_A\), \(\BRAKET{x^*}{y^*}_B =
\braket{y}{x}_B\).  We will see that this construction has the same
formal properties as the dual for partial equivalences of groupoids.
To begin with, a Hilbert bimodule map \(X\to Y\) remains a Hilbert
bimodule map \(X^*\to Y^*\), and \((X^*)^*=X\).  Furthermore,
\((\xi\otimes\eta)^*\mapsto \eta^*\otimes\xi^*\) defines a Hilbert
bimodule map \(\sigma\colon (X\otimes_B Y)^* \to Y^*\otimes_B X^*\)
with dense range, hence an isomorphism.  Applying~\(\sigma\) twice
gives the identity map.  (More precisely,
\(\sigma_{Y^*,X^*}\circ\sigma_{X,Y}=\Id_{(X\otimes_B Y)^*}\).)

\begin{proposition}
  \label{pro:Hilbert_bimodule_dual}
  Let~\(\Hilm\) be a Hilbert \(A,B\)-bimodule.  The inner products
  on~\(\Hilm\) give Hilbert bimodule isomorphisms
  \(\Hilm\otimes_B\Hilm^* \cong \rg(\Hilm)\) and \(\Hilm^*\otimes_A
  \Hilm\cong \s(\Hilm)\), and the restrictions of the left and right
  actions give Hilbert bimodule isomorphisms
  \[
  \rg(\Hilm)\otimes_A \Hilm \cong \Hilm \cong \Hilm\otimes_B \s(\Hilm),\quad
  \s(\Hilm)\otimes_B \Hilm^* \cong \Hilm^* \cong \Hilm^*\otimes_A \rg(\Hilm).
  \]
  that make the following diagrams of isomorphisms commute:
  \begin{equation}
    \label{eq:HHH_to_H}
    \begin{tikzpicture}[baseline=(current bounding box.west)]
      \matrix (m) [cd,column sep=1em] {
        \Hilm\otimes_B \Hilm^*\otimes_A \Hilm & \Hilm\otimes_B \s(\Hilm) \\
        \rg(\Hilm) \otimes_A \Hilm & \Hilm,\\
      };
      \draw[cdar] (m-1-1) -- (m-1-2);
      \draw[cdar] (m-1-1) -- (m-2-1);
      \draw[cdar] (m-2-1) -- (m-2-2);
      \draw[cdar] (m-1-2) -- (m-2-2);
    \end{tikzpicture}
    \begin{tikzpicture}[baseline=(current bounding box.west)]
      \matrix (m) [cd,column sep=1em] {
        \Hilm^*\otimes_A \Hilm\otimes_B \Hilm^* & \Hilm^*\otimes_A \rg(\Hilm) \\
        \s(\Hilm) \otimes_B \Hilm^* & \Hilm^*.\\
      };
      \draw[cdar] (m-1-1) -- (m-1-2);
      \draw[cdar] (m-1-1) -- (m-2-1);
      \draw[cdar] (m-2-1) -- (m-2-2);
      \draw[cdar] (m-1-2) -- (m-2-2);
    \end{tikzpicture}
  \end{equation}

  Let~\(D\) be another \(\Cst\)\nb-algebra, let~\(\Hilm[K]\) be a
  Hilbert \(A,D\)-bimodule and let~\(\Hilm[L]\) be a Hilbert
  \(B,D\)-bimodule with \(\rg(\Hilm[K])\subseteq \rg(\Hilm)\) and
  \(\rg(\Hilm[L])\subseteq \s(\Hilm)\).  Then Hilbert \(A,D\)-bimodule
  maps \(\Hilm\otimes_B \Hilm[L] \to \Hilm[K]\) are naturally in
  bijection with Hilbert \(B,D\)-bimodule maps \(\Hilm[L] \to
  \Hilm^*\otimes_A \Hilm[K]\), and this bijection maps isomorphisms
  again to isomorphisms.  Similarly, Hilbert \(A,D\)-bimodule maps
  \(\Hilm\otimes_B \Hilm[L] \leftarrow \Hilm[K]\) are naturally in
  bijection with Hilbert \(B,D\)-bimodule maps \(\Hilm[L] \leftarrow
  \Hilm^*\otimes_A \Hilm[K]\).
\end{proposition}

\begin{proof}
  The Hilbert bimodule isomorphisms \(\Hilm\otimes_B\Hilm^* \cong
  \rg(\Hilm)\), \(\Hilm^*\otimes_A \Hilm\cong \s(\Hilm)\),
  \(\rg(\Hilm)\otimes_A \Hilm \cong \Hilm \cong \Hilm\otimes_B
  \s(\Hilm)\) and \(\s(\Hilm)\otimes_B \Hilm^* \cong \Hilm^* \cong
  \Hilm^*\otimes_A \rg(\Hilm)\) are routine to check using
  that~\(\Hilm\) is full as a Hilbert
  \(\rg(\Hilm),\s(\Hilm)\)-bimodule.  The diagrams
  in~\eqref{eq:HHH_to_H} are
  equivalent to the requirement \(\BRAKET{x}{y}_A\cdot z = x\cdot
  \braket{y}{z}_B\) in the definition of a Hilbert bimodule.  The
  claim about Hilbert bimodule maps is proved like the analogous one
  about partial equivalences of groupoids in
  Proposition~\ref{pro:dual_pe_unique}; now we use the canonical
  isomorphisms just established and
  Proposition~\ref{pro:Hilbert_bimodule_map} instead of
  Proposition~\ref{pro:characterise_bibundle_isom}.
\end{proof}

\begin{proposition}
  \label{pro:Hilbert_bimodule_dual_unique}
  Up to isomorphism, \(\Hilm^*\) is the unique Hilbert
  \(B,A\)-bimodule~\(\Hilm[K]\) for which there are isomorphisms
  \[
  \Hilm\otimes_B \Hilm[K]\otimes_A \Hilm \cong \Hilm,\qquad
  \Hilm[K]\otimes_A \Hilm\otimes_B \Hilm[K] \cong \Hilm[K].
  \]
  More precisely, if there are such isomorphisms then there is a
  unique Hilbert bimodule isomorphism \(\Hilm^*\congto \Hilm[K]\)
  such that the following map is the identity map:
  \[
  \Hilm \congto \Hilm\otimes_B \Hilm[K] \otimes_A \Hilm
  \congto \Hilm\otimes_B \Hilm^* \otimes_A \Hilm
  \congto \rg(\Hilm) \otimes_A \Hilm
  \congto \Hilm.
  \]
\end{proposition}

\begin{proof}
  Repeat the proof of Proposition~\ref{pro:dual_pe_unique},
  replacing~\(\times_G\) by~\(\otimes_A\).
\end{proof}

\begin{proposition}
  \label{pro:idempotent_Hilbert_bimodules}
  Let \(\Hilm\) be a Hilbert \(A,A\)-bimodule and let \(\mu\colon
  \Hilm\otimes_A \Hilm\to \Hilm\) be a bimodule isomorphism.  Then there is a
  unique isomorphism from~\(\Hilm\) onto an ideal \(I\idealin A\)
  that intertwines~\(\mu\) and the multiplication map \(I\otimes_A I
  \congto I\).  We have \(I=\rg(\Hilm)=\s(\Hilm)\), and the
  multiplication~\(\mu\) is associative.
\end{proposition}

\begin{proof}
  This is proved exactly like Proposition~\ref{pro:idempotents}.
\end{proof}

\begin{definition}
  \label{def:S_act_Cstar}
  Let~\(S\) be an inverse semigroup with unit and let~\(A\)
  be a \(\Cst\)\nb-algebra.  An \emph{\(S\)\nb-action on~\(A\) by
    Hilbert bimodules} consists of
  \begin{itemize}
  \item Hilbert \(A,A\)-bimodules~\(\Hilm_t\) for \(t\in S\);
  \item bimodule isomorphisms \(\mu_{t,u}\colon
    \Hilm_t\otimes_A \Hilm_u\congto \Hilm_{tu}\) for \(t,u\in S\);
  \end{itemize}
  satisfying
  \begin{enumerate}[label=\textup{(AH\arabic*)}]
  \item \(\Hilm_1\) is the identity Hilbert \(A,A\)-bimodule~\(A\);
  \item \(\mu_{t,1}\colon \Hilm_t\otimes_A A\congto \Hilm_t\) and
    \(\mu_{1,u}\colon A\otimes_A \Hilm_u\congto \Hilm_u\) are the canonical
    isomorphisms for all \(t,u\in S\);
  \item associativity: for all \(t,u,v\in S\), the following diagram
    commutes:
    \[
    \begin{tikzpicture}[baseline=(current bounding box.west)]
      \node (1) at (0,1) {\((\Hilm_t\otimes_A \Hilm_u) \otimes_A \Hilm_v\)};
      \node (1a) at (0,0) {\(\Hilm_t\otimes_A (\Hilm_u \otimes_A \Hilm_v)\)};
      \node (2) at (5,1) {\(\Hilm_{tu} \otimes_A \Hilm_v\)};
      \node (3) at (5,0) {\(\Hilm_t\otimes_A \Hilm_{uv}\)};
      \node (4) at (7,.5) {\(\Hilm_{tuv}\)};
      \draw[<->] (1) -- node[swap] {ass} (1a);
      \draw[cdar] (1) -- node {\(\mu_{t,u}\otimes_A \Id_{\Hilm_v}\)} (2);
      \draw[cdar] (1a) -- node[swap] {\(\Id_{\Hilm_t}\otimes_A\mu_{u,v}\)} (3);
      \draw[cdar] (3.east) -- node[swap] {\(\mu_{t,uv}\)} (4);
      \draw[cdar] (2.east) -- node {\(\mu_{tu,v}\)} (4);
    \end{tikzpicture}
    \]
  \end{enumerate}
  If~\(S\) has a zero element~\(0\), we may also require
  \(\Hilm_0=\{0\}\).
\end{definition}

\begin{theorem}
  \label{the:S_act_Cstar_Fell_bundle}
  Let~\(S\) be an inverse semigroup with unit, let~\(A\) be
  a \(\Cst\)\nb-algebra.  Then actions of~\(S\) on~\(A\) by Hilbert
  bimodules are equivalent to saturated Fell bundles over~\(S\)
  \textup(as defined in~\textup{\cite{Exel:noncomm.cartan}}\textup)
  with unit fibre~\(A\).

  More precisely, let \((\Hilm_t)_{t\in S}\) and \((\mu_{t,u})_{t,u\in
    S}\) be an \(S\)\nb-action by Hilbert bimodules on~\(A\).  Then
  there are unique Hilbert bimodule maps \(j_{u,t}\colon \Hilm_t\to \Hilm_u\)
  for \(t\le u\) that make the following diagrams commute for all
  \(t_1,t_2,u_1,u_2\in S\) with \(t_1\le u_1\), \(t_2\le u_2\):
  \begin{equation}
    \label{eq:inclusions_from_bimodule_action}
    \begin{tikzpicture}[baseline=(current bounding box.west)]
      \matrix (m) [cd,column sep=3em] {
        \Hilm_{t_1}\otimes_A \Hilm_{t_2}& \Hilm_{t_1t_2}\\
        \Hilm_{u_1}\otimes_A \Hilm_{u_2}& \Hilm_{u_1u_2}\\
      };
      \draw[cdar] (m-1-1) -- node {\(\mu_{t_1,t_2}\)} (m-1-2);
      \draw[cdar] (m-2-1) -- node {\(\mu_{u_1,u_2}\)} (m-2-2);
      \draw[cdar] (m-1-1) -- node[swap] {\(j_{u_1,t_1}\otimes_A j_{u_2,t_2}\)} (m-2-1);
      \draw[cdar] (m-1-2) -- node {\(j_{u_1u_2,t_1t_2}\)} (m-2-2);
    \end{tikzpicture}
  \end{equation}
  The map~\(j_{u,t}\) is a Hilbert bimodule isomorphism onto
  \(\Hilm_u\cdot \s(\Hilm_t) = \rg(\Hilm_t)\cdot \Hilm_u\).  We have
  \(j_{t,t}=\Id_{\Hilm_t}\) for all \(t\in S\) and \(j_{v,u}\circ j_{u,t}
  = j_{v,t}\) for \(t\le u\le v\) in~\(S\).  And there are unique
  Hilbert bimodule isomorphisms \(J_t\colon \Hilm_t^*\congto \Hilm_{t^*}\),
  \(x\mapsto x^*\), such that \(\mu_{t,t^*,t}(x,x^*,x)=x\cdot \braket{x}{x}_A=\BRAKET{x}{x}_A\cdot x\) for all
  \(x\in \Hilm_t\).  These also satisfy \(\mu_{t,t^*}(x\otimes
  x^*)=\BRAKET{x}{x}_A\), \(\mu_{t^*,t}(x^*,x) = \braket{x}{x}_A\) and
  \((x^*)^* = x\) for all \(x\in \Hilm_t\); \(\mu_{t,u}(x,y)^* =
  \mu_{u^*,t^*}(y^*,x^*)\) for all \(x\in \Hilm_t\), \(y\in \Hilm_u\),
  \(t,u\in S\); and \(j_{u,t}(x)^*=j_{u^*,t^*}(x^*)\) for all \(t\le
  u\) in~\(S\), \(x\in \Hilm_t\).

  Conversely, a saturated Fell bundle \((\A_t)_{t\in S}\) over~\(S\)
  with \(A=\A_1\) becomes an \(S\)\nb-action by Hilbert bimodules by
  taking~\(\Hilm_t=\A_t\) with the multiplication maps~\(\mu_{t,u}\) and
  the \(A\)\nb-bimodule structure induced by the Fell bundle
  multiplication, and the left and right inner products
  \(\BRAKET{x}{y}_A\defeq x\cdot y^*\), \(\braket{x}{y}_A\defeq
  x^*\cdot y\) for \(x,y\in \Hilm_t\).
\end{theorem}

\begin{proof}
  We construct the inclusion maps~\(j_{t,u}\) and the
  involutions~\(J_t\) and show their properties exactly as in the
  proofs of Propositions \ref{pro:inclusions_from_action}
  and~\ref{pro:S_action_involution}.
\end{proof}

With Theorem~\ref{the:S_act_Cstar_Fell_bundle}, it becomes easier to
construct saturated Fell bundles over inverse semigroups because
Definition~\ref{def:S_act_Cstar} needs far less data and has
correspondingly fewer conditions to check.

\begin{remark}
  \label{rem:idempotent_correspondence}
  The correspondence bicategory introduced
  in~\cite{Buss-Meyer-Zhu:Higher_twisted} is not suitable for our
  purposes by the following observation:
  Let \(I\into A \onto A/I\) be a split extension of
  \(\Cst\)\nb-algebras.  Then \(p\colon A\to A/I \to A\) is an
  idempotent endomorphism.  It remains an idempotent arrow in the
  correspondence bicategory.  More generally, if~\(A\) is Morita
  equivalent to an ideal in a \(\Cst\)\nb-algebra~\(B\), then we can
  translate~\(p\) to a correspondence~\(\Hilm\) from~\(B\) to itself
  that is idempotent in the sense that \(\Hilm\otimes_B \Hilm\cong
  \Hilm\) with an associative isomorphism.  Thus there are more
  idempotent endomorphisms in the correspondence bicategory than usual
  for inverse semigroup actions.  Furthermore, the idempotent arrows
  no longer commute up to isomorphism; thus a very basic assumption
  for inverse semigroups fails in this case.  This is why we only
  allowed Hilbert bimodules above.
\end{remark}

\begin{proposition}
  \label{pro:Hilbert_bimdule_bicategory}
  There is a bicategory with \(\Cst\)\nb-algebras as objects, Hilbert
  bimodules as arrows, Hilbert bimodule isomorphisms as
  \(2\)\nb-arrows, and~\(\otimes_B\) as composition of arrows.
\end{proposition}

\begin{proof}
  The correspondence bicategory is constructed already
  in~\cite{Buss-Meyer-Zhu:Higher_twisted}.
  Lemma~\ref{lem:corr_to_Hil_bimod} allows to identify Hilbert
  bimodules with a subset of correspondences.  It is well-known that
  composites of Hilbert bimodules are again Hilbert bimodules.  Hence
  the Hilbert bimodules form a sub-bicategory in the opposite of the
  correspondence bicategory.
\end{proof}

\section{Fell bundles from actions of inverse semigroups}
\label{sec:Fell_from_action}

All groupoids in this section are assumed to be locally quasi-compact, locally
Hausdorff and with (locally compact) Hausdorff object space and a Haar system, so
that they have groupoid \cstar{}algebras.  Let~\(G\) be such a
groupoid and let~\(S\) be a unital inverse semigroup acting on~\(G\)
by partial equivalences.  We want to turn this into an action of~\(S\)
on~\(\Cst(G)\) by Hilbert bimodules; equivalently, we want a Fell
bundle over~\(S\) with unit fibre~\(\Cst(G)\).  There are two closely
related ways to construct this.  We are going to explain one approach
in detail and only sketch the other one briefly in
Section~\ref{sec:other_Fell_bundle_construction}.

We give details for the construction of the Fell bundle using the
transformation groupoid \(L=G\rtimes S\) because this also suggests
how to describe the section \(\Cst\)\nb-algebra of the resulting Fell
bundle.  The transformation groupoid~\(L\) comes with an
\(S\)\nb-grading~\((L_t)_{t\in S}\).  Roughly speaking, our Fell
bundle over~\(S\) will involve the subspaces of \(\Cst(L)\) of
elements supported on the open subsets~\(L_t\).  Since \(G^1=L_1\),
the unit fibre of the Fell bundle will be~\(\Cst(G)\).  This also
suggests that the section \(\Cst\)\nb-algebra of the Fell bundle
over~\(S\) is~\(\Cst(L)\).  This is indeed the case, but the technical
details need some care.

First, we need a Haar system on~\(L\).  We show in
Proposition~\ref{pro:Haar_on_transformation} that the Haar system
on~\(G\) extends uniquely to a Haar system on~\(L\).  Secondly, it is
non-trivial that~\(\Cst(G)\) is contained in~\(\Cst(L)\): this means
that the maximal \(\Cst\)\nb-norm that defines~\(\Cst(G)\) extends to a
\(\Cst\)\nb-norm on~\(\Cst(L)\).  A related issue is to show that an
element of~\(\Cst(L)\) supported in~\(G\) actually belongs
to~\(\Cst(G)\).  These problems become clearer if we construct
a pre-Fell bundle using the dense \Star{}algebra that
defines~\(\Cst(L)\) and then complete it.

In the non-Hausdorff case, continuous functions with compact support
are replaced by finite linear combinations of certain functions that
are not continuous.  The identification of~\(\Cst(L)\) with the
section \cstar{}algebra of the Fell bundle requires a technical result about these functions.
We prove it in Appendix~\ref{sec:Banach_fields} in the more general
setting of sections of upper semicontinuous Banach bundles because
this is not more difficult and allows us to generalise our main
results to Fell bundles over groupoids.

We write~\(\Sect(X)\) for the space of linear combinations of
compactly supported functions on Hausdorff open subsets of a locally
Hausdorff, locally quasi-compact space~\(X\).  This is the space of
compactly supported continuous functions on~\(X\) if and only if~\(X\)
is Hausdorff, and it is often denoted by~\(\Contc(X)\).  We find this
notation misleading, however, because its elements are not continuous
functions.

\subsection{A Haar system on the transformation groupoid}
\label{sec:Haar_system}

Before we enter the construction of Haar systems, we mention an
important trivial case: if~\(G\) is étale, then so is~\(L\).
Therefore, \(L\) certainly has a canonical Haar system if~\(G\) is
étale.  This already covers many examples, and the reader only
interested in étale groupoids may skip the construction of the Haar
system on~\(L\).

We define Haar systems as
in~\cite{Renault:Representations}*{Section~1}.  Thus our Haar system
\((\lambda_G^x)_{x\in G^0}\) on~\(G\) is left invariant, so \(\supp
\lambda_G^x = G^x = \{g\in G^1\mid \rg(g)=x\}\) and
\(g_*\lambda_G^{\s(g)} = \lambda_G^{\rg(g)}\) for all \(g\in G\).  The
continuity requirement for~\((\lambda_G^x)_{x\in G^0}\) is that the
function~\(\lambda_G(f)\) on~\(G^0\) defined by
\(\lambda_G(f)(x)\defeq \int_G f(g) \,\dd\lambda_G^x(g)\) is
continuous on~\(G^0\) for all \(f\in\Sect(G)\).  By the definition
of~\(\Sect(G)\) (see Definition~\ref{def:continuous_Banach_field}), it suffices to
check continuity if~\(f\) is a continuous function with compact
support on a Hausdorff open subset~\(U\) of~\(G\).

\begin{proposition}
  \label{pro:Haar_on_transformation}
  The Haar system on~\(G\) extends uniquely to a Haar system on the
  transformation groupoid~\(L\).
\end{proposition}

\begin{proof}
  Fix \(x\in G^0=L^0\).  We are going to describe the
  measure~\(\lambda_L^x\) on~\(L^x\) in the Haar system.  Since
  \(L=\bigcup_{t\in S}
  L_t\) is an open cover, the measure~\(\lambda_L^x\) is determined by
  its restrictions to~\(L_t\) for all \(t\in S\).  If
  \(x\notin\rg(L_t)\), then there is nothing to do, so consider \(t\in
  S\) with \(x\in\rg(L_t)\), and fix \(g\in L_t\) with \(\rg(g)=x\).
  If \(A\subseteq L_t^x \defeq L_t\cap L^x\) is measurable, then \(A=
  g\cdot (g^{-1}\cdot A)\) with \(g^{-1}\cdot A\subseteq L_t^{-1}\cdot
  L_t = L_1 = G\).  Since we want~\((\lambda_L^x)\) to be
  left invariant and to extend~\((\lambda_G^x)\), we must have
  \(\lambda_L^x(A) = \lambda_G^{\s(g)}(g^{-1}\cdot A)\) if \(g\in
  L_t\) satisfies \(\rg(g)=x\) and \(A \subseteq L_t^x\) is
  measurable.  Hence there is at most one Haar measure on~\(L\)
  extending the given Haar measure on~\(G\).

  If \(g_1,g_2\in L_t\) satisfy \(\rg(g_1)=\rg(g_2)=x\), then \(g_1^{-1}\cdot
  g_2\in L_t^{-1} L_t = L_1=G\); the left invariance
  of~\((\lambda_G^x)\) with respect to~\(G\) implies that
  \(\lambda_G^{\s(g)}(g^{-1}\cdot A)\) does not depend on the choice
  of~\(g\).  If \(\emptyset\neq A\subseteq L_t^x\cap L_u^x\), then we
  may pick the same element \(g\in A\) to define the measure of~\(A\)
  as a subset of \(L_t^x\) and of~\(L_u^x\).  Thus the definitions
  of~\(\lambda_L^x\) on the sets \(L_t^x\) for \(t\in S\) are
  compatible.  Thus there is a unique measure~\(\lambda_L^x\)
  on~\(L^x\) with \(\lambda_L^x(A) = \lambda_G^{\s(g)}(g^{-1}\cdot
  A)\) whenever \(A \subseteq L_t^x\) is measurable and \(g\in L_t\)
  satisfies \(\rg(g)=x\).  If \(l\in L\) has \(\s(l)=x\), then
  \(l_*(\lambda^x)\) is a measure on~\(L^{\rg(l)}\) with the same
  properties that characterise~\(\lambda^{\rg(l)}\) uniquely; so we
  get the left invariance of our family of measures:
  \(l_*(\lambda^{\s(l)}) = \lambda^{\rg(l)}\) for all \(l\in L\).

  Checking continuity by hand is unpleasant, so we use a different
  description of the same Haar system for this purpose.  Recall
  that~\(L_t\) is an equivalence between restrictions of~\(G\) to open
  invariant subsets of~\(G^0\).  The proof that equivalent groupoids
  have Morita--Rieffel equivalent groupoid \(\Cst\)\nb-algebras uses a
  family of measures on the equivalence bibundle in order to define
  the right inner product; this measure on~\(L_t\) is exactly the one
  described above (see the proof of
  \cite{Renault:Representations}*{Corollaire 5.4}), and its continuity
  is known, even in the non-Hausdorff case.  Thus our family of
  measures restricts to a continuous family on each~\(L_t\).  Since
  the map \(\bigoplus \Sect(L_t) \to \Sect(L)\) in
  Proposition~\ref{pro:Banb_and_cover} is surjective, the family of
  measures~\((\lambda_L^x)\) is continuous.
\end{proof}

\subsection{Construction of the Fell bundle}
\label{sec:construct_Fell}

We know now that~\(L\) has a Haar system.  So we get a \Star{}algebra
structure on~\(\Sect(L)\) as in~\cites{Renault:Representations,
  Muhly-Williams:Renaults_equivalence}.  Since the Haar measure
on~\(L\) extends the one on~\(G\), the map \(\Sect(G)\to\Sect(L)\)
induced by the open embedding \(G\to L\) is a \Star{}algebra
isomorphism onto its image.  The groupoid \(\Cst\)\nb-algebras of
\(L\) and~\(G\) are the completions of \(\Sect(L)\) and~\(\Sect(G)\)
for suitable \(\Cst\)\nb-norms.

\begin{lemma}
  \label{lem:multiply_in_Sect_L}
  The involution on \(\Sect(L)\) maps \(\Sect(L_t)\)
  onto~\(\Sect(L_t^{-1}) = \Sect(L_{t^*})\).  The convolution product
  maps \(\Sect(L_t)\times\Sect(L_u)\) to~\(\Sect(L_{tu})\).
\end{lemma}

\begin{proof}
  The claim for the involution is trivial.  The claim for the
  convolution product follows, of course, from \(L_t\cdot L_u\subseteq
  L_{tu}\), but requires some care in the non-Hausdorff case because
  the convolution product is not defined directly, see the proof of
  \cite{Muhly-Williams:Renaults_equivalence}*{Proposition~4.4}.  If
  \(f_1\in\Sect(U)\), \(f_2\in\Sect(V)\) for Hausdorff open subsets
  \(U\subseteq L_t\) and \(V\subseteq L_u\), and if \(U\cdot V\) is
  also Hausdorff, then we directly get \(f_1*f_2\in \Sect(U\cdot V)\)
  with \(U\cdot V\subseteq L_{tu}\).  If \(U\cdot V\) is
  non-Hausdorff, a partition of unity is used to write \(f_1\)
  and~\(f_2\) as finite sums of functions on \emph{smaller} Hausdorff
  open subsets \(U'\subseteq U\), \(V'\subseteq V\) for which
  \(U'\cdot V'\) is Hausdorff.  Since \(U'\cdot V'\subseteq U\cdot
  V\subseteq L_{tu}\), we get \(\Sect(L_t)*\Sect(L_u)\subseteq
  \Sect(L_{tu})\) as desired.
\end{proof}

Lemma~\ref{lem:multiply_in_Sect_L} gives
\(\Sect(G)*\Sect(L_t)\subseteq\Sect(L_t)\) and
\(\Sect(L_t)*\Sect(G)\subseteq\Sect(L_t)\), so \(\Sect(L_t)\) is a
\(\Sect(G)\)-bimodule; it also implies \(f_1^**f_2\in\Sect(G)\) and
\(f_1*f_2^*\in\Sect(G)\) for all \(f_1,f_2\in\Sect(L_t)\), which gives
\(\Sect(G)\)-valued left and right inner products on \(\Sect(L_t)\).
We also have \(f_1*f_2\in\Sect(L_{tu})\) for \(f_1\in\Sect(L_t)\) and
\(f_2\in\Sect(L_u)\), and these multiplication maps are associative
and ``isometric'' with respect to the \(\Sect(G)\)-valued inner
products.  We put ``isometric'' in quotation marks because we have not
yet talked about norms.

\begin{lemma}
  \label{lem:positivity_pre-Fell}
  \(f^* * f\in\Sect(G)\) is positive in~\(\Cst(G)\) for each \(t\in
  S\), \(f\in\Sect(L_t)\), and the closed linear span of \(f_1^**
  f_2\) for \(f_1,f_2\in\Sect(L_t)\) is dense in
  \(\Cst(G_{\rg(L_t)})\).
\end{lemma}

\begin{proof}
  We have already used in the proof of
  Proposition~\ref{pro:Haar_on_transformation} that the Haar measure
  on~\(L\) restricts to the usual family of measures on the partial
  equivalence space~\(L_t\).  In that context, the positivity of such
  inner products is already proved in
  \cites{Renault:Representations,Muhly-Williams:Renaults_equivalence}
  in order to show that~\(\Sect(L_t)\) may be completed to a Hilbert
  \(\Cst(G)\)-bimodule.  The proof that an equivalence induces a
  Morita--Rieffel equivalence also shows that the inner product
  defined above is full, that is, the closed linear span of \(f_1^**
  f_2\) for \(f_1,f_2\in\Sect(L_t)\) is dense in
  \(\Cst(G_{\rg(L_t)})\).
\end{proof}

Hence we may complete~\(\Sect(L_t)\) to a Hilbert
bimodule~\(\Cst(L_t)\) over~\(\Cst(G)\).  The densely defined
convolution map \(\Sect(L_t)\times\Sect(L_u) \to \Sect(L_{tu})\)
extends to a Hilbert bimodule map
\[
\mu_{t,u}\colon \Cst(L_t)\otimes_{\Cst(G)}\Cst(L_u) \to \Cst(L_{tu})
\]
because it is isometric for the \(\Sect(G)\)\nb-valued inner products.
Since \(\Cst(L_t)\) is full as a Hilbert bimodule over
\(\Cst(G_{\rg(L_t)})\) and \(\Cst(G_{\s(L_t)})\), it follows that the
maps~\(\mu_{t,u}\) above are surjective.

The associativity of the multiplication on the dense
subspaces~\(\Sect(L_t)\) extends to~\(\Cst(L_t)\).  Thus we have
constructed an action of~\(S\) by Hilbert bimodules on~\(\Cst(G)\).
By Theorem~\ref{the:S_act_Cstar_Fell_bundle}, this is equivalent to a
saturated Fell bundle \(\Cst(L_t)_{t\in S}\) over~\(S\).

\begin{theorem}
  \label{the:iterated_crossed_1}
  The section \(\Cst\)\nb-algebra~\(\Cst(S,\Cst(L_t)_{t\in S})\) is
  naturally isomorphic to the groupoid
  \(\Cst\)\nb-algebra~\(\Cst(L)\).
\end{theorem}

This theorem looks almost trivial from our construction; but the proof
requires a technical result about~\(\Sect(L)\) to be proved in
Appendix~\ref{sec:Banach_fields}.  Before we turn to that, we first add
coefficients in a Fell bundle over~\(L\).

The above construction still works in almost literally the same way if
we replace \(\Sect(L_t)\) by \(\Sect(L_t,\Banb)\) everywhere,
where~\(\Banb\) is a Fell bundle over the groupoid~\(L\).
Unfortunately, we could not find a reference for the generalisation of
Lemma~\ref{lem:positivity_pre-Fell} to this context.  The references
on groupoid crossed products we could find consider \emph{either} Fell
bundles over Hausdorff groupoids (such
as~\cite{Muhly-Williams:Equivalence.FellBundles}) \emph{or} a more
restrictive class of actions for non-Hausdorff groupoids (such as
\cites{Renault:Representations, Muhly-Williams:Renaults_equivalence}),
but not both.  In particular, the positivity of the inner product
on~\(\Sect(L_t,\Banb)\) for a partial equivalence~\(L_t\) is only
proved in some cases: for arbitrary upper semicontinuous Fell bundles
over Hausdorff groupoids
in~\cite{Muhly-Williams:Equivalence.FellBundles}; for Green twisted
actions of non-Hausdorff groupoids on continuous fields of
\(\Cst\)\nb-algebras over~\(G^0\) in~\cite{Renault:Representations};
and for untwisted actions by automorphisms of non-Hausdorff groupoids
on \(\Cont_0(G^0)\)-algebras
in~\cite{Muhly-Williams:Renaults_equivalence}.  This is probably only
a technical issue that will be resolved eventually, but not in this
paper.  So we add an assumption about it in our next theorem.

\begin{theorem}
  \label{the:iterated_crossed_2}
  Let~\(\Banb\) be a Fell bundle over~\(L\).  Assume that
  \(f^**f\in\Sect(G,\Banb)\) is positive in~\(\Cst(G,\Banb)\) for all
  \(f\in \Sect(L_t,\Banb)\), \(t\in S\), and that the linear span of
  these inner products is dense in~\(\Cst(G_{\s(L_t)},\Banb)\).  Then
  there is a Fell bundle
  \(\Cst(L_t,\Banb)_{t\in S}\) over~\(S\) that has the section
  \(\Cst\)\nb-algebra of the restriction \(\Cst(G,\Banb|_G)\) as unit
  fibre.  The section \(\Cst\)\nb-algebra
  \(\Cst(S,\Cst(L_t,\Banb)_{t\in S})\) is naturally isomorphic to the
  section \(\Cst\)\nb-algebra of the groupoid Fell
  bundle~\(\Cst(L,\Banb)\).
\end{theorem}

Theorem~\ref{the:iterated_crossed_1} is a special case of
Theorem~\ref{the:iterated_crossed_2} for the constant Fell
bundle~\(\C\).  It remains to prove
Theorem~\ref{the:iterated_crossed_2}.  This will be done in
Appendix~\ref{sec:proof_iterated_crossed}, after some preliminary
results about Banach bundles in Appendix~\ref{sec:Banach_fields}.

\begin{corollary}
  \label{cor:compare_groupoid_isg}
  Let~\(L\) be an étale topological groupoid with Hausdorff locally
  compact object space and with a Haar system.  Let~\(S\) be a wide
  inverse subsemigroup of~\(\Bis(L)\), that is, \(\bigcup_{t\in S} t =
  L\) and \(\bigcup_{t\in S, t\subseteq t_1\cap t_2} = t_1\cap t_2\)
  for all \(t_1,t_2\in S\).  Then the groupoid \(\Cst\)\nb-algebra
  of~\(L\) is isomorphic to the crossed product \(\Cont_0(L^0)\rtimes
  S\).

  More generally, if~\(\Banb\) is a Fell bundle over~\(L\), then the
  section \(\Cst\)\nb-algebra \(\Cst(L,\Banb)\) is isomorphic to the
  section \(\Cst\)\nb-algebra of the associated Fell bundle
  over~\(S\).
\end{corollary}

\begin{proof}
  The assumptions on~\(S\) ensure that~\(L\) is an \(S\)\nb-graded
  groupoid by \(L_t \defeq t\) with unit fibre \(G=L^0\).  So
  Theorem~\ref{the:iterated_crossed_1} gives the first assertion, and
  Theorem~\ref{the:iterated_crossed_2} gives the second one.  In this
  case, positivity is not an issue because we are dealing with a
  space~\(G\), so positivity in \(\Cst(G,\Banb)\) is equivalent to
  pointwise positivity in all \(x\in L^0=G^0\).  The value
  \((f^**f)(x)\) for \(f\in\Sect(L_t,\Banb)\) is either zero or
  \(f(l)^*f(l)\) for the unique \(l\in L_t\) with \(\s(l)=x\).  This
  is assumed to be positive in the definition of a Fell bundle over a
  groupoid.
\end{proof}

The isomorphism \(\Cst(L) \cong \Cont_0(L^0)\rtimes S\) is already
proved in \cite{Exel:Inverse_combinatorial}*{Theorem~9.8} (if~\(L^0\) is
second countable and~\(S\) is countable).  The more general result for
(separable) Fell bundles over (second countable) étale groupoids is
proved in
\cite{BussExel:Fell.Bundle.and.Twisted.Groupoids}*{Theorem~2.13}.

Another special case worth mentioning are group extensions.  Let
\(G\rightarrowtail H\onto S\) be an extension of locally compact
groups with discrete~\(S\).  This gives an action of~\(S\), viewed as
an inverse semigroup, on~\(G\) by
Theorem~\ref{the:group_action_extension}.  We get a Fell bundle
over~\(S\) with unit fibre~\(\Cst(G)\) and section
\(\Cst\)\nb-algebra~\(\Cst(H)\).  More generally, we get a similar
result for a Fell bundle over~\(H\) (compare with
\cite{Buss-Meyer:Crossed_products}*{Example~3.9}).  Our Fell bundle
also comes from a Green twisted action of~\((H,G)\), and in this
formulation, our theorem is well-known in this case (see
\cites{Green:Local_twisted, Chabert-Echterhoff:Twisted}).

\begin{corollary}
  \label{cor:CstarG_embeds}
  In the situation of Theorem~\textup{\ref{the:iterated_crossed_2}},
  the canonical map from~\(\Cst(G,\Banb)\) to~\(\Cst(L,\Banb)\) is
  injective.
\end{corollary}

\begin{proof}
  The unit fibre of the Fell bundle in
  Theorem~\ref{the:iterated_crossed_2} is \(\Cst(G,\Banb)\) and the
  section \(\Cst\)\nb-algebra is \(\Cst(L,\Banb)\).  The unit fibre
  always embeds into the section \(\Cst\)\nb-algebras of a Fell bundle
  over an inverse semigroup, see
  \cite{Exel:noncomm.cartan}*{Corollary~8.10}.
\end{proof}

Next, we note a useful variant of Theorem~\ref{the:iterated_crossed_1} for
group-valued cocycles.

Let~\(L\) be a locally quasi-compact, locally Hausdorff groupoid,
let~\(S\) be a group, and let \(c\colon L\to S\) be a
\(1\)\nb-cocycle.  Let \(L_t\defeq c^{-1}(t)\subseteq L\) for \(t\in
S\), and let \(G=L_1= c^{-1}(1)\).  Since we do not assume anything
about~\(c\), this need \emph{not} be an \(S\)\nb-grading (compare
Theorem~\ref{the:group_action_extension}).  Nevertheless, we may
complete~\(\Sect(L_t)\) to a Hilbert bimodule over~\(\Cst(G)\) and
thus get a Fell bundle over~\(S\).  The difference to the situation
above is that this Fell bundle need not be saturated any more.

\begin{theorem}
  \label{the:group_cocycle_Fell}
  The section \cstar{}algebra of the Fell bundle over~\(S\) with unit
  fibre~\(\Cst(G)\) just described is isomorphic to~\(\Cst(L)\).
  Hence the canonical map \(\Cst(G)\to\Cst(L)\) is faithful.
\end{theorem}

\begin{proof}
  The proofs of Theorem~\ref{the:iterated_crossed_1} and
  Corollary~\ref{cor:CstarG_embeds} still work for non-saturated Fell
  bundles (even over inverse semigroups).  Alternatively, we may
  replace our non-saturated Fell bundle over~\(G\) by a saturated Fell
  bundle over an inverse semigroup associated to~\(G\), just as for
  partial actions
  (see~\cite{Exel:Partial_actions}).  This
  does not change the section \(\Cst\)\nb-algebra, and afterwards
  Theorem~\ref{the:iterated_crossed_1} applies literally.
\end{proof}

\subsection{Another construction of the Fell bundle}
\label{sec:other_Fell_bundle_construction}

The construction of the Fell bundle over~\(S\) in
Section~\ref{sec:construct_Fell} used the transformation groupoid.
Now we construct this Fell bundle using the abstract functorial
properties of actions on groupoids and their corresponding actions on
\cstar{}algebras.  Actually, some aspects of this have been used to
prove Lemma~\ref{lem:positivity_pre-Fell} above.

It is well-known that two equivalent groupoids have Morita--Rieffel
equivalent \(\Cst\)\nb-algebras
(see~\cite{Muhly-Renault-Williams:Equivalence}), even in the
non-Hausdorff case (see~\cite{Renault:Representations}).  The proof is
constructive: given an equivalence~\(X\) from~\(H\) to~\(G\), the
space~\(\Sect(X)\) is completed to a
\(\Cst(G)\)-\(\Cst(H)\)-imprimitivity bimodule, using certain natural
formulas for a \(\Sect(G)\)-\(\Sect(H)\)-bimodule structure and
\(\Sect(G)\)- and \(\Sect(H)\)-valued inner products.  An important
ingredient here is that the Haar measures on \(G\) and~\(H\) give
canonical families of measures on the fibres of the range and source
maps of~\(X\), which may be used to integrate functions on~\(X\).

Even if~\(X\) is only a \emph{partial} equivalence, the same formulas
still work and give a Hilbert bimodule~\(\Cst(X)\) from~\(\Cst(H)\)
to~\(\Cst(G)\) by completing~\(\Sect(X)\).  If \(f\colon X\to X'\) is
an isomorphism between two partial equivalences, then \(f_*\colon
\Sect(X)\to\Sect(X')\) defined by \(f_*(h) = h\circ f^{-1}\) is an
isomorphism that preserves all structure, so it extends to an
isomorphism \(\Cst(X)\congto \Cst(X')\).

\begin{theorem}
  \label{the:functor_gpd_Cstar}
  The maps \(G\mapsto\Cst(G)\) from groupoids to \(\Cst\)\nb-algebras,
  \(X\mapsto \Cst(X)\) from partial equivalences to Hilbert bimodules,
  and \(f\mapsto f_*\) from bibundle isomorphisms to Hilbert bimodule
  isomorphisms are part of a functor from the bicategory of partial
  groupoid equivalences to the bicategory of \(\Cst\)\nb-algebras and
  Hilbert bimodules.
\end{theorem}

\begin{proof}
  The above map is strictly compatible with unit arrows: the unit
  arrow~\(G^1\) on~\(G\) is sent to \(\Cst(G^1)=\Cst(G)\), and the
  unit transformations in both bicategories are also preserved.  To
  complete the above data to a functor of bicategories, it remains to
  give natural isomorphisms \(\Cst(X)\otimes_{\Cst(H)} \Cst(Y)\cong
  \Cst(X\times_H Y)\) and check that they satisfy the expected
  associativity condition for three composable partial equivalences.
  They are constructed by writing down the ``convolution map''
  \(\Sect(X)\odot \Sect(Y)\to \Sect(X\times_H Y)\) given by the formula
  \begin{equation}
    \label{eq:FormulaForProductGeneral}
    (\xi\cdot \eta)(x,y)\defeq
    \int_{H^1}\xi(x\cdot h)\eta(h^{-1}\cdot y) \,\dd{\lambda}^u(h),
  \end{equation}
  for all \(\xi\in \Sect(X)\), \(\eta\in \Sect(Y)\) and \((x,y)\in
  X\times_H Y\), where \(u=\s(x)=\rg(y)\).  It is routine to check
  that the map~\eqref{eq:FormulaForProductGeneral} has dense range and
  is a bimodule map and an isometry for both inner products; thus it
  extends to an isomorphism between the completions:
  \(\Cst(X)\otimes_{\Cst(H)} \Cst(Y)\cong \Cst(X\times_H Y)\).

  One way to construct the convolution maps and check their properties
  is like our construction above using the transformation groupoid:
  build an appropriate linking groupoid containing all the data.  For
  two composable equivalences \(Y\) and~\(X\) from~\(K\) to~\(H\)
  and from~\(H\) to~\(H\), this linking groupoid has object space
  \(G^0\sqcup H^0\sqcup K^0\); its arrow space is a disjoint union of
  \(G^1\), \(H^1\), \(K^1\), \(X\), \(Y\), \(X^*\), \(Y^*\),
  \(X\times_H Y\), and \(Y^*\times_H X^*\), the source and range maps
  are the obvious ones, and the multiplication is defined using the
  left and right actions of \(G\), \(H\) and~\(K\) and canonical maps.
  This is indeed a topological groupoid, and it inherits a canonical
  Haar system if \(G\), \(H\) and~\(K\) have Haar systems.  The
  convolution map is the restriction of the convolution in this larger
  groupoid to \(X\times_H Y\).  Given three composable partial
  equivalences, there is a similar linking groupoid combining all the
  relevant data, and the associativity of its convolution product on
  \(X\times_H Y\times_K Z\) gives the associativity coherence of the
  isomorphisms \(\Cst(X)\otimes_{\Cst(H)} \Cst(Y)\cong \Cst(X\times_H
  Y)\).
\end{proof}

\begin{remark}
  The above theorem is extended in the thesis of Rohit
  Holkar~\cite{Holkar:Thesis}, where a similar functor from a
  bicategory of groupoid correspondences to the bicategory of
  \(\Cst\)\nb-correspondences is constructed.  This construction is
  more difficult because the family of measures needed to write down
  the right inner product is no longer canonical and becomes part of
  the data.  Hence the behaviour of the measures under composition has
  to be studied as well.
\end{remark}

An inverse semigroup action by partial equivalences may be defined as
a functor (of bicategories) from the inverse semigroup to the
bicategory of groupoids and partial equivalences.  Composing it with
the functor in the theorem gives a functor from the inverse semigroup
to the bicategory of Hilbert bimodules, which is the same thing as an
action by Hilbert bimodules.  This is the same as a saturated Fell
bundle over the inverse semigroup by
Theorem~\ref{the:S_act_Cstar_Fell_bundle}.  This is the second
construction of the Fell bundle over~\(S\).  It gives an isomorphic
Fell bundle because the Haar measure on~\(L\) used above is the same
as the combination of the measure families on the partial
equivalences~\(L_t\) that are used to define the convolution maps in
Theorem~\ref{the:functor_gpd_Cstar}.

More concretely, an action \((X_t,\mu_{t,u})\) of~\(S\) on~\(G\)
yields the action on~\(\Cst(G)\) given by the Hilbert
bimodules~\(\Cst(X_t)\) with the multiplication maps
\[
\Cst(X_t)\otimes_A \Cst(X_u) \congto
\Cst(X_t\times_G X_u) \xrightarrow[\cong]{\Cst(\mu_{t,u})}
\Cst(X_{tu}),
\]
which involve the convolution isomorphisms \(\Cst(X_t)\otimes_A
\Cst(X_u) \congto \Cst(X_t\times_G X_u)\).  This is associative by the
associativity coherence of these convolution isomorphisms.

\section{Actions of inverse semigroups and groupoids}
\label{sec:isg_to_groupoid}

Let~\(H\) be an étale groupoid with locally compact Hausdorff object
space.  So far, we have constructed actions of the inverse
semigroup~\(\Bis(H)\) on certain \(\Cst\)\nb-algebras.  Instead, we
would like to construct actions of~\(H\) itself.  In this section, we
are going to see that both kinds of actions are very closely related.
Here an action of~\(\Bis(H)\) is as above: an action by Hilbert
bimodules or, equivalently, a saturated Fell bundle over~\(\Bis(H)\).
The corresponding ``actions'' of~\(H\) are saturated Fell bundles
over~\(H\).

First we explain how to turn a Fell bundle over~\(H\) into one
over~\(\Bis(H)\).  So let \(\Banb = (B_h)_{h\in H}\) be a Fell bundle
over~\(H\) (see \cites{Kumjian:Fell_bundles,
  BussExel:Fell.Bundle.and.Twisted.Groupoids}).  Let \(A\defeq
\Cont_0(H^0,\Banb)\) be the \(\Cst\)\nb-algebra of
\(\Cont_0\)\nb-sections of~\(\Banb\) over~\(H^0\); this is a
\(\Cont_0(H^0)\)-\(\Cst\)\nb-algebra by construction.  If
\(t\in\Bis(H)\), then the Fell bundle operations turn \(\Hilm_t\defeq
\Cont_0(t,\Banb)\) into a Hilbert
\(\Cont_0(\rg(t),\Banb)\)-\(\Cont_0(\s(t),\Banb)\)-bimodule.  The
multiplication in the Fell bundle induces multiplication maps
\(\mu_{t,u}\colon \Hilm_t\otimes_A \Hilm_u \to \Hilm_{tu}\).  This
gives an action of~\(\Bis(H)\) on~\(A\) by Hilbert bimodules.

Not every action of~\(\Bis(H)\) by Hilbert bimodules is of this form.
The obstruction lies in how idempotents in~\(\Bis(H)\) act.
Idempotents in~\(\Bis(H)\) are the same as open subsets of~\(H^0\).
We identify the idempotent semilattice~\(E(\Bis(H))\) with the
complete lattice~\(\Open(H^0)\) of open subsets of~\(H^0\).  So the
action of idempotents in~\(\Bis(H)\) becomes a map from~\(\Open(H^0)\)
to the complete lattice~\(\Ideal(A)\) of ideals in~\(A\).

\begin{theorem}
  \label{the:action_Bis_groupoid}
  An action \((\Hilm_t,\mu_{t,u})_{t\in\Bis(H)}\) of the inverse
  semigroup~\(\Bis(H)\) on a \(\Cst\)\nb-algebra~\(A\) by Hilbert
  \(A\)\nb-bimodules comes from a Fell bundle over~\(H\) if and only
  if the map from \(E(\Bis(H))\cong \Open(H^0)\) to~\(\Ideal(A)\)
  commutes with suprema.  This Fell bundle over~\(H\) is unique up to
  isomorphism, and the Fell bundles over \(\Bis(H)\) and~\(H\) have
  the same section \(\Cst\)\nb-algebras.
\end{theorem}

\begin{proof}
  A map \(\Open(H^0)\to \Ideal(A)\) comes from a continuous map
  \(\Prim(A)\to H^0\) if and only if it commutes with finite infima
  and arbitrary suprema by \cite{Meyer-Nest:Bootstrap}*{Lemma~2.25};
  here we need~\(H^0\) to be a sober space, a very mild condition that
  certainly allows all locally Hausdorff spaces.  Compatibility with
  finite infima says that it is a morphism of semilattices, which we
  assume anyway; compatibility with suprema is an extra condition.  A
  continuous map \(\Prim(A)\to H^0\) is equivalent to an isomorphism
  between~\(A\) and the \(\Cst\)\nb-algebra of \(\Cont_0\)\nb-sections
  of an upper semicontinuous field~\((A_x)_{x\in H^0}\) of
  \(\Cst\)\nb-algebras over~\(H^0\) (see~\cite{Nilsen:Bundles}).  Thus
  the criterion in the theorem is necessary and sufficient for~\(A\)
  to come from such an upper semicontinuous field.  This gives a Fell bundle
  over \(H^0\subseteq H\).  It remains to extend this to all of~\(H\).

  Let \(t\in\Bis(H)\).  Then~\(\Hilm_t\) is a Hilbert
  \(A\)\nb-bimodule.  For \(h\in t\subseteq H^1\), we define
  \(\Hilm_{h,t} \defeq \Hilm_t \otimes_A A_{\s(h)}\); this is a
  Hilbert \(A_{\s(h)}\)-module.  If \(\xi\in\Hilm_t\), then
  \(\norm{\xi}^2 = \norm{\langle \xi,\xi\rangle}\), and for \(\langle
  \xi,\xi\rangle\in A\), the norm is the supremum of the norms of its
  images in~\(A_x\) for all \(x\in H^0\).  Therefore, the canonical
  map from~\(\Hilm_t\) to \(\prod_{h\in t} \Hilm_{h,t}\) is isometric.
  Thus we view~\(\Hilm_t\) as a space of sections of the bundle of
  Banach spaces~\(\Hilm_{h,t}\) over~\(t\).  This is an upper
  semicontinuous bundle on~\(t\) because~\((A_x)_{x\in H^0}\) is and
  the norm on~\(\Hilm_t\) is given by \(\norm{\xi}^2 = \norm{\langle
    \xi,\xi\rangle}\) with \(\langle \xi,\xi\rangle\in A\).

  If \(t,u\in \Bis(H)\) and \(h\in t\cap u\), then both
  \(\Hilm_{h,t}\) and~\(\Hilm_{h,u}\) are candidates for the
  fibre~\(\Hilm_h\) of our Fell bundle at~\(h\).  These are isomorphic
  through the canonical isomorphisms \(j_{t,t\cap u}\colon
  \Hilm_{t\cap u}\to \Hilm_t|_{\s(t\cap u)}\) and \(j_{u,t\cap u}
  \colon \Hilm_{t\cap u}\to \Hilm_u|_{\s(t\cap u)}\) from
  Theorem~\ref{the:S_act_Cstar_Fell_bundle}.

  For each \(h\in H^1\), choose some \(t_h\in\Bis(H)\) with \(h\in
  t_h\) and define \(\Hilm_h\defeq \Hilm_{h,t_h}\).  If \(t\in\Bis(H)\), then
  there are canonical isomorphisms \(\Hilm_h\cong \Hilm_{h,t}\) for all \(h\in
  t\).  We use them to transport the topology on the
  bundle~\((\Hilm_{h,t})_{h\in t}\) to the bundle~\((\Hilm_h)_{h\in t}\).
  These topologies are compatible on~\(t\cap u\) for all
  \(t,u\in\Bis(H)\).  Since the subsets \(t\in\Bis(H)\) form an open
  cover of~\(H^1\), there is a topology on the whole
  bundle~\((\Hilm_h)_{h\in H^1}\) that coincides with the topology
  on~\((\Hilm_h)_{h\in t}\) described above for each \(t\in\Bis(H)\).  In
  particular, the space of \(\Cont_0\)\nb-sections of~\((\Hilm_h)_{h\in
    H^1}\) on~\(t\) coincides naturally with~\(\Hilm_t\).

  Let \(A(U)\) for \(U\in\Open(G)\) be the ideal of
  \(\Cont_0\)\nb-sections of~\((A_x)\) vanishing outside~\(U\).  Then
  \(A(U)=\Hilm_U\) if we view \(U\in E(\Bis(G))\).  We have
  \begin{equation}
    \label{eq:compatible_HtAU}
    \Hilm_t\otimes_A A(U) = \Hilm_{t\cdot U} = \Hilm_{t(U)\cdot t}
    = A(t(U))\otimes_A \Hilm_t
  \end{equation}
  for all \(t\in\Bis(H)\), \(U\in\Open(H^0)\) with \(U\subseteq
  \s(t)\).  Here we view each \(t\in \Bis(H)\) as a partial
  homeomorphism \(\s(t)\to \rg(t)\) and write \(t(U)\) for the image
  of~\(U\) under this map.  This is exactly how \(\Bis(H)\) acts
  on~\(H^0\).  Equation~\eqref{eq:compatible_HtAU} implies that
  \(\Hilm_{h,t}\cong A_{\rg(h)}\otimes_A \Hilm_t\).  Thus~\(\Hilm_h\) is a
  Hilbert \(A_{\rg(h)}\)-\(A_{\s(h)}\)-bimodule.  The isomorphism
  \(\Hilm_t\otimes_A \Hilm_u \to \Hilm_{tu}\) is \(A\)\nb-linear and
  hence \(\Cont_0(H^0)\)-linear.  Thus it restricts to an isomorphism
  on the fibres, \(\Hilm_{g,t} \otimes_A \Hilm_{h,u} \to \Hilm_{gh,tu}\) for all
  \(g\in t\), \(h\in u\) with \(\s(g)=\rg(h)\).  The compatibility of
  the multiplication with the inclusion maps from
  Theorem~\ref{the:S_act_Cstar_Fell_bundle} shows that these maps on
  the fibres do not depend on the choice of \(t\) and~\(u\) with
  \(h\in t\) and \(h\in u\).  Thus we get well-defined isomorphisms
  \(\Hilm_g\otimes_{A_{\s(g)}} \Hilm_h\to \Hilm_{gh}\) for all \(g,h\in H^1\) with
  \(\s(g)=\rg(h)\).  Since they can be put together to maps
  \(\Hilm_t\otimes_A \Hilm_u\to \Hilm_{tu}\) for all \(t,u\in\Bis(H)\)
  and since~\(\Bis(H)\) covers~\(H^1\), they are locally continuous,
  hence continuous.  Similarly, the isomorphisms \(\Hilm_t^*\cong
  \Hilm_{t^*}\) must come from well-defined, continuous maps
  \(\Hilm_h^*\to \Hilm_{h^{-1}}\) for \(h\in H^1\) by restricting them to
  fibres.  The remaining algebraic conditions needed for a Fell bundle
  over the groupoid~\(H^1\) all follow easily because
  \((\Hilm_t,\mu_{t,u})\) gives a Fell bundle over~\(\Bis(H)\).

  If we turn the Fell bundle over~\(H\) constructed above into a Fell
  bundle over~\(\Bis(H)\) again, we clearly get back the original Fell
  bundle over~\(\Bis(H)\) because~\(\Hilm_t\) is the space of
  \(\Cont_0\)\nb-sections of~\((\Hilm_h)_{h\in t}\).  Conversely, if
  we start with a Fell bundle over~\(H\), turn it into a Fell bundle
  over~\(\Bis(H)\), and then use the above construction to go back, we
  get an isomorphic Fell bundle over~\(H\).  Hence we get a bijection
  between isomorphism classes of the two types of Fell bundles.
  Theorem~\ref{the:iterated_crossed_2} shows that the passage from
  Fell bundles over~\(H\) to Fell bundles over~\(\Bis(H)\) does not
  change the section \(\Cst\)\nb-algebras.
\end{proof}

We assumed~\(G^0\) to be Hausdorff and locally compact so far because
Fell bundles over groupoids have not yet been defined in greater
generality.  We suggest to use the necessary and sufficient criterion
in Theorem~\ref{the:action_Bis_groupoid} as a definition:

\begin{definition}
  \label{def:action_etale_sober}
  Let~\(G\) be an étale topological groupoid for which~\(G^0\) (and
  hence~\(G^1\)) is sober.  An action of~\(G\) on a
  \(\Cst\)\nb-algebra~\(A\) is an action of~\(\Bis(G)\) by Hilbert
  bimodules for which the resulting map \(\Open(G^0)\to \Ideal(A)\)
  commutes with arbitrary suprema.
\end{definition}

Sobriety of~\(G^0\) is needed to turn a map \(\Open(G^0)\to
\Ideal(A)\) that commutes with suprema into a continuous map
\(\Prim(A)\to G^0\) (see \cite{Meyer-Nest:Bootstrap}*{Lemma~2.25}).

Let~\(G\) be a sober space~\(G^0\) viewed as a groupoid.  Then an
action of~\(G\) is the same as a continuous map \(\Prim(A)\to G^0\).
In the notation of~\cite{Meyer-Nest:Bootstrap}, this turns~\(A\) into
a \(\Cst\)\nb-algebra over~\(G^0\).  It is unclear what the ``fibres''
of such a \(\Cst\)\nb-algebra over~\(G^0\) should be if~\(G^0\) is
badly non-Hausdorff.  Therefore, it is not clear how to describe
actions of étale sober groupoids in the sense of
Definition~\ref{def:action_etale_sober} as Fell bundles over~\(G\).
If~\(G^0\) is locally Hausdorff and locally quasi-compact, then
Definition~\ref{def:action_etale_sober} seems to work quite well; we
plan to discuss this in greater detail elsewhere.

The criterion in Theorem~\ref{the:action_Bis_groupoid} also suggests
how to define actions of étale groupoids on other groupoids:

\begin{definition}
  \label{def:sober_etale_groupoid_action_on_groupoid}
  Let~\(G\) be an étale topological groupoid for which~\(G^0\) (and
  hence~\(G^1\)) is sober, and let~\(H\) be an arbitrary topological
  groupoid.  An \emph{action} of~\(G\) on~\(H\) is an action
  of~\(\Bis(G)\) on~\(H\) by partial equivalences for which the map
  \(\Open(G^0)\to\Open(H^0/H)\) that describes the action of
  \(E(\Bis(G))\) commutes with arbitrary suprema.
\end{definition}

The extra assumption in
Definition~\ref{def:sober_etale_groupoid_action_on_groupoid} and
\cite{Meyer-Nest:Bootstrap}*{Lemma~2.25} ensure that the map
\(\Open(G^0)\to\Open(H^0/H)\) for an action of~\(G\) on~\(H\) comes
from a continuous map \(H^0/H\to G^0\) or, equivalently, an
\(H\)\nb-invariant continuous map \(H^0\to G^0\).

\begin{proposition}
  \label{pro:action_from_groupoid_to_Cstar}
  Let~\(H\) be a locally quasi-compact, locally Hausdorff groupoid
  with Hausdorff object space and with a Haar system.  An action
  of~\(G\) on~\(H\) induces an action of~\(G\) on~\(\Cst(H)\) as
  well.
\end{proposition}

\begin{proof}
  In Section~\ref{sec:construct_Fell}, we turn an action
  of~\(\Bis(G)\) on~\(H\) into an action of~\(\Bis(G)\)
  on~\(\Cst(H)\).  For any open \(H\)\nb-invariant subset~\(U\)
  of~\(H^0\), the closure of~\(\Sect(H_U)\) in~\(\Cst(H)\) is an ideal
  \(\Cst(H_U)\) in~\(\Cst(H)\).  The map
  \(\Open(H^0/H)\to\Ideal(\Cst(H))\), \(U\mapsto \Cst(H_U)\), commutes
  with suprema.  Hence Theorem~\ref{the:action_Bis_groupoid} applies
  to the action of~\(\Bis(G)\) on~\(\Cst(H)\) if the action
  of~\(\Bis(G)\) satisfies the condition in
  Definition~\ref{def:sober_etale_groupoid_action_on_groupoid}.
\end{proof}

\subsection{The motivating example}
\label{sec:motivating_example}

Now we consider our motivating example: an action of a locally
Hausdorff, locally quasi-compact, étale groupoid~\(H\) on a locally
Hausdorff, locally quasi-compact space~\(Z\).  Let~\(\OCover\) be
a Hausdorff open covering of~\(Z\) and let~\(G_\OCover\) be the
associated covering groupoid, which is étale, locally compact and
Hausdorff.  Its \(\Cst\)\nb-algebra
\(\Cst(G_\OCover)\) is our noncommutative model for the non-Hausdorff
space~\(Z\).  We want to construct an ``action'' of~\(H\) on it that
models the given action of~\(H\) on~\(Z\).

To construct it, we use the inverse semigroup \(S\defeq\Bis(H)\) of
bisections of~\(H\).  First we turn the action of~\(H\) on~\(Z\) into
an action of~\(S\) on~\(Z\) by partial homeomorphisms in the usual
way: a bisection \(t\in S\) acts by the homeomorphism
\(\rg^{-1}(\s(t))\to \rg^{-1}(\rg(t))\), \(z\mapsto g_{\rg(z)}\cdot
z\), where~\(g_x\) is the unique arrow in~\(t\) with \(\s(g_x)=x\).

We have seen in Corollary~\ref{cor:action_on_covering_groupoid} that
the \(S\)\nb-action on~\(Z\) induces an \(S\)\nb-action
on~\(G_\OCover\) by partial equivalences.  The transformation
groupoid \(G_\OCover\rtimes S\) for this action is Hausdorff,
étale and locally compact.  It is equivalent to \(Z\rtimes S\) by
Corollary~\ref{cor:action_on_covering_groupoid}.

Let \(p\colon X\defeq \bigsqcup_{U\in\OCover} U \to Z\) be the
canonical map.  Then \(G_\OCover = p^*(Z)\).  An idempotent
\(U\in\Open(H^0)\) in~\(\Bis(H)\) acts on~\(G_\OCover\) by the
identity map on the open invariant subgroupoid \(G_\OCover|_{(\rg\circ
  p)^{-1}(U)}\).  That is, \(\Open(H^0)\) acts on~\(G_\OCover\)
through the map \(\Open(H^0)\to \Open(G_\OCover^0/G_\OCover)\),
\(U\mapsto (\rg\circ p)^{-1}(U)\); this commutes with suprema and
infima.  Thus our action of~\(\Bis(H)\) on~\(G_\OCover\) is also an
action of~\(H\) in the sense of
Definition~\ref{def:sober_etale_groupoid_action_on_groupoid}.

We may identify
\(Z\rtimes S\cong Z\rtimes H\) using the obvious \(S\)\nb-grading
on~\(Z\rtimes H\) and Theorem~\ref{the:graded_groupoid_versus_action},
so \(G_\OCover\rtimes S\) is equivalent to~\(Z\rtimes H\).

The \(S\)\nb-action on~\(G_\OCover\) induces a Fell bundle over~\(S\)
with unit fibre~\(\Cst(G_\OCover)\), which we view as an action
of~\(S\) on~\(\Cst(G_\OCover)\).  Theorem~\ref{the:iterated_crossed_1}
gives an isomorphism between its section \(\Cst\)\nb-algebra
\(\Cst(G_\OCover)\rtimes S\) and the groupoid \(\Cst\)\nb-algebra
\(\Cst(G_\OCover\rtimes S)\).  We may turn our Fell bundle
over~\(\Bis(H)\) into a Fell bundle over the groupoid~\(H\) by
Proposition~\ref{pro:action_from_groupoid_to_Cstar}.

Theorem~\ref{the:action_Bis_groupoid} also says that the section
\(\Cst\)\nb-algebra of the Fell bundle over~\(H\) is isomorphic to
\(\Cst(G_\OCover)\rtimes S \cong \Cst(G_\OCover\rtimes S)\).
The restriction to the unit fibre is~\(\Cst(G_\OCover)\) by
construction.  We are going to describe this Fell bundle over~\(H\).

We have \(G_\OCover\rtimes S\cong p^*(Z\rtimes H)\),
that is, the object space of~\(G_\OCover\rtimes S\) is~\(X\) and
the arrow space is homeomorphic to the space of triples
\((x_1,h,x_2)\), \(x_1,x_2\in X\), \(h\in H^1\) with
\(\rg(p(x_1))=\rg(h)\) and \(\rg(p(x_2))=\s(h)\) in~\(H^0\).  Here
\((x_1,h,x_2)\) is an arrow from~\(x_2\) to~\(x_1\), and the
multiplication is \((x_1,h_1,x_2)\cdot (x_2,h_2,x_3) =
(x_1,h_1h_2,x_3)\).  For \(h\in H^1\), let~\(K_h\) be the subspace of
triples \((x_1,h,x_2)\) for \(x_1,x_2\in X\), \(\rg(p(x_1))=\rg(h)\)
and \(\rg(p(x_2))=\s(h)\).  Since \(p\) and~\(H\) are étale, this is a
discrete
set.  The fibre at~\(h\) of our Fell bundle over~\(H\) is the
completion of the space~\(\Contc(K_h)\) of finitely supported
functions on~\(K_h\) to a Hilbert bimodule over
\(\Cst(K_{1_{\rg(h)}})\) and \(\Cst(K_{1_{\s(h)}})\).

\begin{proposition}
  \label{pro:Cstar_action_on_G1}
  Let~\(Z\) be a basic action of~\(H\) with Hausdorff quotient
  space~\(H\backslash Z\), for instance, \(Z=H^1\) with the action by
  left or right multiplication and quotient space~\(H^0\).  Then the
  groupoid \(G_\OCover\rtimes S\) is equivalent to \(H\backslash
  Z\) and \(\Cst(G_\OCover)\rtimes S\) is Morita equivalent to
  \(\Cont_0(H\backslash Z)\).
\end{proposition}

\begin{proof}
  The groupoid \(G_\OCover\rtimes S\) is equivalent to~\(Z\rtimes
  S\).  This is the same as \(Z\rtimes H\) by
  Theorem~\ref{the:graded_groupoid_versus_action}, using the evident
  \(S\)\nb-grading on~\(Z\rtimes H\).  Since the \(H\)\nb-action
  on~\(Z\) is basic, \(Z\rtimes H\) is equivalent to~\(H\backslash
  Z\).  This space is assumed to be Hausdorff, and
  \(G_\OCover\rtimes S\) is also a groupoid with Hausdorff object
  space.  So the equivalence between them is of the usual type,
  involving free and proper actions, by
  Proposition~\ref{pro:proper_action}.  Hence it induces a
  Morita--Rieffel equivalence from \(\Cont_0(H\backslash Z)\) to
  \(\Cst(G_\OCover)\rtimes S\).
\end{proof}

In the situation of Proposition~\ref{pro:Cstar_action_on_G1},
\(G_\OCover\rtimes S\) has Hausdorff arrow space because it must
be isomorphic to the covering groupoid of the open surjection
\(G_\OCover^0\to (G_\OCover\rtimes S)\backslash
G_\OCover^0 \cong H\backslash Z\) between two Hausdorff spaces.
In this case, it is also easy to see that any Fell bundle over the
groupoid \(G_\OCover\rtimes S\) is a pull-back of a Fell bundle
over~\(H\backslash Z\), which is the same as a \(\Cont_0(H\backslash
Z)\)-\(\Cst\)\nb-algebra~\(\Banb\).  The section \(\Cst\)\nb-algebra
of the Fell bundle over \(G_\OCover\rtimes S\) is Morita--Rieffel
equivalent to this \(\Cont_0(H\backslash
Z)\)-\(\Cst\)\nb-algebra~\(\Banb\).  By
Theorem~\ref{the:iterated_crossed_2}, this is also the section
\(\Cst\)\nb-algebra of the Fell bundle over~\(S\) associated
to~\(\Banb\).

Many properties like properness, amenability, essential principality
are shared by an action of a groupoid on a space and its
transformation groupoid.  This suggests how to extend these notions to
inverse semigroup actions on groupoids.  We take this as a definition
for proper actions of inverse semigroups on locally compact groupoids:

\begin{definition}
  \label{def:proper_action_S}
  An action of an inverse semigroup~\(S\) on a topological
  groupoid~\(G\) is \emph{proper} if the groupoid \(G\rtimes S\) is
  proper, that is, the following map is proper (that is, stably closed):
  \[
  (s,r)\colon (G\rtimes S)^1 \to G^0\times G^0,\qquad
  g\mapsto (s(g),r(g)).
  \]
  The action is called \emph{free} if this map is injective.
\end{definition}

Let~\(L\) be a proper groupoid such that~\(L^0\) is a locally compact
Hausdorff space.  Then the image of~\(L^1\) in \(L^0\times L^0\) is
locally compact and Hausdorff because it is a closed subspace of a
locally compact Hausdorff space.  Since this subspace is closed and
the orbit space projection \(L^0\to L\backslash L^0\) is open, it also
follows that \(L\backslash L^0\) is locally compact Hausdorff (see
Proposition~\ref{pro:Top_open_orbit}).  The
groupoid~\(L\) itself need not be Hausdorff: the non-Hausdorff group
bundle in Section~\ref{sec:explicit_example} is proper in this sense
because it is quasi-compact and the image of~\((s,r)\) is closed.
If~\(L\) acts \emph{freely} and properly on a Hausdorff space~\(L^0\),
however, then~\(L^1\) must be Hausdorff.  In this case, we also get
information about any open subgroupoid, which leads to the following
proposition:

\begin{proposition}
  \label{pro:proper_action_orbits}
  Let~\(S\) act properly and freely on a locally Hausdorff, locally
  quasi-compact groupoid~\(G\).  Then~\(G\) is a basic groupoid, so
  that~\(G\) is equivalent to the locally Hausdorff, locally
  quasi-compact space \(G\backslash G^0\).
\end{proposition}

\begin{proof}
  The map in Definition~\ref{def:proper_action_S} is a homeomorphism
  onto its image because it is continuous, injective, and closed.
  Hence its restriction to the open subspace \(G^1\subseteq (G\rtimes
  S)^1\) is still a homeomorphism onto its image.  This means
  that~\(G\) is a basic groupoid, so~\(G\) is equivalent to
  \(G^0/G\).  This is locally Hausdorff and locally quasi-compact by
  Proposition~\ref{pro:locally_Hausdorff_quotient}.
\end{proof}

Thus the free and proper actions of~\(S\) all come from actions on
locally Hausdorff spaces that are desingularised by replacing the
space by a Hausdorff groupoid~\(G\).

\subsection{Inverse semigroup models for étale groupoids}
\label{sec:ISGModels}

Let~\(G\) be an étale groupoid.  So far, we have described actions
of~\(G\) through actions of the inverse semigroup~\(\Bis(G)\).
Since~\(\Bis(G)\) is usually quite big, even uncountable, we now
replace it by smaller inverse semigroups.  The following definition
describes which inverse semigroups we allow as ``models'' for~\(G\):

\begin{definition}
  \label{def:isg_model}
  An \emph{inverse semigroup model} for an étale groupoid~\(G\)
  consists of an inverse semigroup~\(S\), an \(S\)\nb-action on the
  space~\(G^0\) by partial homeomorphisms, and an isomorphism
  \(G^0\rtimes S\cong G\) of étale groupoids that is the identity on
  objects.
\end{definition}

In particular, if \(S\subseteq \Bis(G)\) is a wide inverse
subsemigroup, then~\(S\) with its usual action on~\(G^0\) and the
canonical isomorphism \(G^0\rtimes S\cong G\) from
Corollary~\ref{cor:wide_action} is a model for~\(G\).

\begin{lemma}
  \label{lem:isg_model}
  An inverse semigroup model for~\(G\) is equivalent to an inverse
  semigroup~\(S\) with a homomorphism \(\varphi\colon S\to \Bis(G)\)
  that induces an isomorphism \(G^0\rtimes S\to G^0\rtimes\Bis(G)\cong
  G\), where we use the canonical action of~\(\Bis(G)\) on~\(G^0\)
  and~\(\varphi\) to let~\(S\) act on~\(G^0\).
\end{lemma}

\begin{proof}
  Let~\(S\) act on~\(G^0\).  There is a canonical homomorphism
  \(S\to\Bis(G^0\rtimes S)\), see~\cite{Exel:Inverse_combinatorial}.
  Combined with an isomorphism \(G^0\rtimes S\cong G\), we get a
  homomorphism \(\varphi\colon S\to\Bis(G)\).  Conversely, such a
  homomorphism induces an action of~\(S\) on~\(G^0\) and then a
  continuous groupoid homomorphism \(G^0\rtimes S\to G^0\rtimes
  \Bis(G) \cong G\).  Routine computations show that these two
  constructions are inverse to each other.
\end{proof}

The following lemma characterises inverse semigroup models more
concretely when we take \(\hat S= \Bis(G)\).

\begin{lemma}
  \label{lem:Z-isomorphism}
  Let \(S\) and~\(\hat{S}\) be inverse semigroups, let \(\varphi\colon
  S\to \hat S\) be a homomorphism, and let~\(\hat S\) act on~\(Z\) by
  partial homeomorphisms.  The induced groupoid homomorphism
  \(\tilde\varphi\colon Z\rtimes S\to Z\rtimes \hat S\) is an
  isomorphism if and only if
  \begin{enumerate}
  \item \label{enum:ZIso1} for all \(t_1,t_2\in S\) and every \(z\in
    Z\) with \(z\in D_{t_1^*t_1}\cap D_{t_2^*t_2}\) and every \(f\in
    E(\hat{S})\) with \(z\in D_f\) and \(\varphi(t_1)f=\varphi(t_2)f\), there is
    \(e\in E(S)\) with \(z\in D_e\) and \(t_1e=t_2e\);
  \item \label{enum:ZIso2} for every \(u\in \hat{S}\) and every \(z\in Z\)
    with \(z\in D_{u^*u}\), there is \(t\in S\) with \(z\in
    D_{t^*t}\) and there is \(f\in E(\hat{S})\) with \(z\in D_f\) and
    \(uf=\varphi(t)f\).
  \end{enumerate}
\end{lemma}

In this case, we call~\(\varphi\) a \emph{\(Z\)\nb-isomorphism}.

\begin{proof}
  The groupoid homomorphism~\(\tilde\varphi\) is the identity on objects
  and always continuous and open on arrows, so the only issue
  is whether~\(\tilde\varphi\) is bijective on arrows.  It is routine to
  check that~\ref{enum:ZIso1} is equivalent to injectivity
  and~\ref{enum:ZIso2} to surjectivity of~\(\tilde\varphi\).
\end{proof}

Let \(S\) and \(\varphi\colon S\to\Bis(G)\) be an inverse semigroup
model for an étale topological groupoid~\(G\).  Which actions of~\(S\)
on groupoids by partial equivalences or on \(\Cst\)\nb-algebras by
Hilbert bimodules come from actions of~\(G\)?

First we consider a trivial special case to see why we need more data.
Let~\(G\) be just a topological space, viewed as a groupoid.  In this
case, the trivial inverse semigroup~\(\{1\}\) is an inverse semigroup
model.  An action of~\(S\) contains no information.  An action
of~\(G\) on a topological groupoid~\(H\) or a \(\Cst\)\nb-algebra is
simply a continuous map \(\psi\colon H^0/H\to G^0\) or \(\psi\colon
\Prim(A)\to G^0\), respectively.

\begin{theorem}
  \label{the:model_action_groupoid}
  Let~\(G\) be a sober étale topological groupoid and let~\(S\) and
  \(\varphi\colon S\to\Bis(G)\) be an inverse semigroup model
  for~\(G\).  Let~\(H\) be a topological groupoid.  An action of~\(G\)
  on~\(H\) by partial equivalences is equivalent to a pair consisting
  of an action of~\(S\) on~\(H\) by partial equivalences and an
  \(S\)\nb-equivariant map \(\psi\colon H^0/H\to G^0\).  The
  transformation groupoid for an action of~\(G\) \textup{(}that is,
  \(\Bis(G)\)\textup{)} and its restriction to~\(S\) are the same.
\end{theorem}

The \(S\)\nb-equivariance of~\(\psi\) refers to the actions of~\(S\)
on \(H^0/H\) and~\(G^0\) by partial homeomorphisms induced by
the action on~\(H\) and by~\(\varphi\).

\begin{proof}
  First let~\(G\) act on~\(H\); more precisely, \(\Bis(G)\) acts
  on~\(H\) and the resulting map \(\Open(G^0)=
  E(\Bis(G))\to\Open(H^0/H)\) commutes with suprema
  (Definition~\ref{def:sober_etale_groupoid_action_on_groupoid}).
  \cite{Meyer-Nest:Bootstrap}*{Lemma~2.25} shows that it comes from a
  continuous map \(\psi\colon H^0/H\to G^0\).  This map is
  \(\Bis(G)\)-equivariant and hence \(S\)\nb-equivariant.

  Now let~\(S\) act on~\(H\) and let \(\psi\colon H^0/H\to G^0\) be an
  \(S\)\nb-equivariant map.  Let \(L\defeq H\rtimes S\) with its
  canonical \(S\)\nb-grading \((L_t)_{t\in S}\).  We claim that there
  is a unique \(\Bis(G)\)-grading \((\bar{L}_t)_{t\in\Bis(G)}\)
  on~\(L\) with \(\bar{L}_{\varphi(t)} = L_t\) for all \(t\in S\), and
  \(\bar{L}_U = H^1_{\psi^{-1}(U)}\) for \(U\in\Open(G^0)\).  These
  two conditions on the \(\Bis(G)\)-grading say exactly that it corresponds
  to the given \(S\)\nb-action and map~\(\psi\).  So the proof of the
  claim will finish the proof of the theorem.

  For \(t\in\Bis(G)\) and \(u\in S\), we may form \(t\cap
  \varphi(u)\in\Bis(G)\).  We have
  \[
  t\cap \varphi(u) = t\cdot V_{t,u} = \varphi(u)\cdot V_{t,u}\qquad
  \text{for }V_{t,u}= \s(t\cap \varphi(u))\in\Open(G^0);
  \]
  here we also view~\(V_{t,u}\) as an idempotent element
  of~\(\Bis(G)\).  Since~\(S\) models~\(G\), we have \(t =
  \bigcup_{u\in S} t\cap \varphi(u)\) and hence \(\s(t) =
  \bigcup_{u\in S} V_{t,u}\).  Any \(\Bis(G)\)-grading with
  \(\bar{L}_V = H^1_{\psi^{-1}(V)}\) for all \(V\in\Open(G^0)\)
  satisfies
  \[
  \bar{L}_t|_{\psi^{-1}(V_{t,u})}
  = \bar{L}_t \cdot \bar{L}_{V_{t,u}}
  = \bar{L}_{t \cap \varphi(u)}
  = \bar{L}_{\varphi(u)}|_{\psi^{-1}(V_{t,u})}
  \]
  for all \(t\in\Bis(G)\), \(u\in S\).  Since \(\s(t) = \bigcup_{u\in
    S} V_{t,u}\) and~\(\psi\) is \(S\)\nb-equivariant, this shows that
  there is at most one \(\Bis(G)\)-grading with the required
  properties, namely,
  \[
  \bar{L}_t = \bigcup_{u\in S} L_u|_{\psi^{-1}(V_{t,u})}.
  \]
  More explicitly, \(l\in\bar{L}_t\) if and only if \(l\in L_u\) for
  some \(u\in S\) for which \(t\) and~\(\varphi(u)\) have the same
  germ at~\(\psi(\s(l))\).  We must prove that
  \((\bar{L}_t)_{t\in\Bis(G)}\) is a grading with all desired
  properties.

  First we check \(\bar{L}_{\varphi(u)} = L_u\) for \(u\in
  S\).  The inclusion~\(\supseteq\) is trivial.  If \(l\in
  \bar{L}_{\varphi(u)}\), then \(l\in L_{u'}\) for some \(u'\in S\)
  for which \(\varphi(u)\) and~\(\varphi(u')\) have the same germ at
  \(\psi(\s(l))\in G^0\).  Hence there is an idempotent element \(e\in
  S\) with \(\psi(\s(l))\in \varphi(e)\) and \(ue=u'e\).  Since \(L_e
  = H^1_{\psi^{-1}(e)}\), we get \(l\in L_{u'}L_e = L_{u'e} = L_{ue} =
  L_uL_e\subseteq L_u\).  This finishes the proof that
  \(\bar{L}_{\varphi(u)} = L_u\) for all \(u\in S\).

  Next we check \(\bar{L}_W = H^1_{\psi^{-1}(W)}\) for
  \(W\in\Open(G^0)\).  The inclusion~\(\supseteq\) holds because
  \(V_{W,1} = W\).  Conversely, let \(l\in\bar{L}_W\).  Then \(l\in
  L_u\) for some \(u\in S\) for which \(\varphi(u)\) and~\(\Id_W\)
  have the same germ at~\(\psi(\s(l))\).  Since \(G^0\rtimes S\cong
  G\), there is an idempotent \(e\in S\) with \(\psi(\s(l))\in
  \varphi(e)\) and \(ue=e\).  An argument as in the previous paragraph
  shows that \(l\in L_uL_e = L_e \subseteq H^1\).  Thus \(\bar{L}_W =
  H^1_{\psi^{-1}(W)}\) for all \(W\in\Open(G^0)\).

  If \(t\in\Bis(G)\), \(u\in S\), then \((\varphi(u)\cap t)^* =
  \varphi(u^*)\cap t^*\).  Hence \(V_{t^*,u^*} = t(V_{t,u}) =
  \varphi(u)(V_{t,u})\).  This implies \(L_{t^*} = L_t^{-1}\) for
  all \(t\in\Bis(G)\).

  Let \(t_1,t_2\in\Bis(G)\).  We claim that \(\bar{L}_{t_1}\cdot
  \bar{L}_{t_1} = \bar{L}_{t_1t_2}\).  The inclusion~\(\subseteq\)
  follows because \((\varphi(u_1)\cap t_1)\cdot (\varphi(u_2)\cap t_2)
  \subseteq \varphi(u_1u_2)\cap t_1t_2\).  For the converse inclusion,
  take \(l\in \bar{L}_{t_1t_2}\).  Then \(t\in L_u\) for some \(u\in
  S\) for which \(t_1t_2\) and~\(\varphi(u)\) have the same germ
  at~\(\psi(\s(l))\).  Factor this germ as~\(g_1g_2\) with \(g_j\in
  t_j\) for \(j=1,2\).  There are \(u_j\in S\) with
  \(g_j\in\varphi(u_j)\) for \(j=1,2\) because \(G\cong G^0\rtimes
  S\).  Thus \(\varphi(u_1)\varphi(u_2)=\varphi(u_1u_2)\)
  and~\(t_1t_2\) have the same germ~\(g_1g_2\) at~\(\psi(\s(l))\).
  Then \(u_1u_2\) and~\(u\) also have the same germ there, and an
  argument as above shows that \(l\in L_{u_1u_2}\) as well.
  Using~\ref{enum:Gra1} for the \(S\)\nb-grading, we get \(l_j\in
  L_{u_j}\) for \(j=1,2\) with \(l=l_1l_2\).  Then \(\s(l_2)=\s(l)\)
  and \(\rg(l_1) = \rg(l)\).  This allows to prove \(l_2\in
  \bar{L}_{t_2}\) and \(l_1^{-1}\in \bar{L}_{t_1^*}\), so that
  \(l_1\in \bar{L}_{t_1}\).  Hence the \(\Bis(G)\)-grading
  satisfies~\ref{enum:Gra1}.

  It is clear that \(\bar{L}_{t_1}\subseteq \bar{L}_{t_2}\) if
  \(t_1\le t_2\) in \(\Bis(G)\), so \(\bar{L}_{t_1}\cap \bar{L}_{t_2}
  \supseteq \bigcup_{v\le t_1,t_2} \bar{L}_v = \bar{L}_{t_1\cap t_2}\)
  for all \(t_1,t_2\in\Bis(G)\).  For the converse inclusion, take
  \(l\in \bar{L}_{t_1}\cap \bar{L}_{t_2}\).  Then there are
  \(u_1,u_2\in S\) with \(l\in L_{u_1}\cap L_{u_2}\), such that
  \(t_j\) and~\(\varphi(u_j)\) have the same germ at~\(\psi(\s(l))\)
  for \(j=1,2\).  \ref{enum:Gra4} for the \(S\)\nb-grading gives
  \(v\in S\) with \(v\le u_1,u_2\) and \(l\in L_v\).  Since~\(\psi\)
  is \(S\)\nb-equivariant, \(\psi(\s(l))\) belongs to the domain
  of~\(\varphi(v)\), so the germs of \(\varphi(v)\)
  and~\(\varphi(u_i)\) at~\(\psi(\s(l))\) are equal.  Then the germs
  of \(t_1\) and~\(t_2\) at~\(\psi(\s(l))\) are equal as well, that is,
  \(t_1\cap t_2\) is defined at~\(\psi(\s(l))\) and has the same germ
  there as~\(\varphi(v)\).  This means that \(l\in \bar{L}_{t_1\cap t_2}\).
  This verifies~\ref{enum:Gra4} for the \(\Bis(G)\)-grading.

  Since \(\bar{L}_{\varphi(u)} = L_u\) for all \(u\in S\) and
  \(\bigcup_{u\in S} L_u = L^1\), we also get \(\bigcup_{t\in\Bis(G)}
  \bar{L}_t = L^1\), which is~\ref{enum:Gra6}.
\end{proof}

The following lemma is needed to formulate a similar result for
actions on \(\Cst\)\nb-algebras:

\begin{lemma}
  \label{lem:action_on_Prim}
  An action of~\(S\) on a \(\Cst\)\nb-algebra~\(A\) by Hilbert
  bimodules induces an action of~\(S\) on~\(\Prim(A)\) by partial
  homeomorphisms.
\end{lemma}

\begin{proof}
  The Rieffel Correspondence (see
  \cite{Raeburn-Williams:Morita_equivalence}*{Corollary~3.33}) says
  that an imprimitivity bimodule~\(\Hilm\) from~\(B\) to~\(A\) induces
  a homeomorphism \(\Prim(B)\congto \Prim(A)\).  The corresponding
  lattice isomorphism
  \[
  \Ideal(B) = \Open(\Prim(B)) \congto \Open(\Prim(A)) = \Ideal(A)
  \]
  sends an ideal \(J\subseteq B\) to the unique ideal \(I\subseteq A\)
  with \(I\cdot \Hilm=\Hilm\cdot J\).  A Hilbert \(A,B\)-bimodule
  induces a \emph{partial} homeomorphism \(\Prim(B)\to \Prim(A)\)
  because it is an imprimitivity bimodule between certain ideals in
  \(A\) and~\(B\), which correspond to open subsets of the primitive
  ideal spaces.  Isomorphic Hilbert bimodules induce the same partial
  homeomorphism, of course.  The partial homeomorphism associated to a
  tensor product bimodule \(\Hilm_1\otimes_B\Hilm_2\) is the composite
  of the partial homeomorphisms associated to \(\Hilm_1\)
  and~\(\Hilm_2\).  Thus the map from~\(S\) to \(\PHomeo(\Prim(A))\)
  induced by an action on~\(A\) by Hilbert bimodules is a
  homomorphism.
\end{proof}

\begin{theorem}
  \label{the:model_action_Cstar}
  Let~\(G\) be a sober étale topological groupoid and let~\(S\) and
  \(\varphi\colon S\to\Bis(G)\) be an inverse semigroup model
  for~\(G\).  Let~\(A\) be a \(\Cst\)\nb-algebra.  An action of~\(G\)
  on~\(A\) by Hilbert bimodules is equivalent to a pair consisting of
  an action of~\(S\) on~\(A\) by Hilbert bimodules and an
  \(S\)\nb-equivariant map \(\psi\colon \Prim(A)\to G^0\).  The
  section \(\Cst\)\nb-algebras of the corresponding Fell bundles over \(\Bis(G)\)
  and~\(S\) are the same.
\end{theorem}

The \(S\)\nb-equivariance of~\(\psi\) refers to the action of~\(S\) on
\(\Prim(A)\) from Lemma~\ref{lem:action_on_Prim}.

\begin{proof}
  Assume first that~\(G^0\) is locally compact Hausdorff.  In that
  case, an action of~\(G\) is the same as a Fell bundle over~\(G\) by
  Theorem~\ref{the:action_Bis_groupoid}.  This determines an action
  of~\(\Bis(G)\), which we may compose with~\(\varphi\) to get an
  action of~\(S\); we also get an \(S\)\nb-equivariant map~\(\psi\).
  Conversely, let an action of~\(S\) and a continuous
  \(S\)\nb-equivariant map \(\psi\colon \Prim(A)\to G^0\) be given.
  Since \(G\cong G^0\rtimes S\), we may carry over the proof of
  Theorem~\ref{the:action_Bis_groupoid}.  The \(S\)\nb-equivariance
  of~\(\psi\) gives the compatibility
  condition~\eqref{eq:compatible_HtAU}.  Hence literally the same
  argument still works.

  If~\(G^0\) is only a sober topological space, we need a different
  proof because we cannot describe \(G\)\nb-actions fibrewise.  We
  first construct the section \(\Cst\)\nb-algebra~\(B\) of the Fell
  bundle over~\(S\) corresponding to the action by
  Theorem~\ref{the:S_act_Cstar_Fell_bundle}.  This \(\Cst\)\nb-algebra
  is \(S\)\nb-graded by construction: it is the Hausdorff completion
  of the \Star{}algebra \(\bigoplus_{t\in S} \Hilm_t\) in the maximal
  \(\Cst\)\nb-seminorm that vanishes on \(j_{u,t}(\xi)\delta_u -
  \xi\delta_t\) for all \(t,u\in S\) with \(t\le u\) and all
  \(\xi\in\Hilm_t\), and we let \(B_t\subseteq B\) be the image
  of~\(\Hilm_t\) in~\(B\).  In particular, we may identify \(A=B_1\).
  Now we must construct a \(\Bis(G)\)-grading \((\bar{B}_t)_{t\in
    \Bis(G)}\) on~\(B\) with \(\bar{B}_{\varphi(t)} = B_t\) for all
  \(t\in S\) and \(\bar{B}_U = A(U)\) for all \(U\in \Open(G^0)\),
  where~\(A(U)\) denotes the ideal in~\(A\) corresponding to
  \(\psi^{-1}(U)\in\Open(\Prim(A))\).  This is done similarly to the
  proof of Theorem~\ref{the:model_action_groupoid}.  Since this is
  rather technical and we already have another proof in the locally
  compact Hausdorff case, we leave it to the determined reader to
  spell out the details of this argument.
\end{proof}

\section{Actions by automorphisms are not enough}
\label{sec:automorphisms_fail}

The following theorem shows that the multiplication action of a
non-Hausdorff groupoid on its own arrow space cannot be described by
a continuous groupoid action by automorphisms.

\begin{theorem}
  \label{the:no-go-theorem}
  Let~\(G\) be a locally quasi-compact, locally Hausdorff, étale
  groupoid with Hausdorff~\(G^0\), such that~\(G^1\) is not
  Hausdorff.  Let~\(A\) be a \(\Cst\)\nb-algebra with \(\Prim(A)
  \cong G^1\).  There is no continuous \textup{(}twisted\textup{)}
  action of~\(G\) on~\(A\) by automorphisms that induces the left
  multiplication action on~\(\Prim(A)\cong G^1\).
\end{theorem}

\begin{proof}
  Since \(\Prim(A)\cong G^1\), the lattice of ideals in~\(A\) is
  order-isomorphic to the lattice of open subsets in~\(G^1\).  Let
  \(A(U)\idealin A\) for an open subset \(U\subseteq G^1\) be the
  corresponding ideal in~\(A\).  Then \(\Prim(A(U))\cong U\).

  Part of a continuous action of~\(G\) on~\(A\) is a continuous map
  \(\Prim(A)\to G^0\).  (This is equivalent to a
  \(\Cont_0(G^0)\)-algebra structure.)  Since we want to have the left
  multiplication action of~\(G^1\) on \(\Prim(A)\), we assume that
  this map becomes the range map \(G^1\to G^0\) when we identify
  \(\Prim(A)\cong G^1\).  The fibre at \(x\in G^0\) is the restriction
  of~\(A\) to the closed subset \(G^x=\{g\in G^1\mid \rg(g)=x\}\),
  which we denote by~\(A|_{G^x}\); we have \(\Prim(A|_{G^x})=G^x\).
  A \(G\)\nb-action on~\(A\) must provide isomorphisms
  \(\alpha_g\colon A|_{G^{\s(g)}} \to A|_{G^{\rg(g)}}\) for \(g\in
  G^1\).  We assume that~\(\alpha_g\) induces the map
  \(G^{\s(g)}\to G^{\rg(g)}\), \(h\mapsto gh\), on the primitive
  ideal space.

  What does continuity of \(g\mapsto\alpha_g\) mean?  Let
  \(U,V\subseteq G^1\) be bisections, then \(U\cdot V\) is also a
  bisection.  If \(g\in U\), \(h\in V\) satisfy \(\s(g)=\rg(h)\),
  then~\(\alpha_g\) restricts to an isomorphism \(\alpha_{g,h}\colon
  A|_{h}\to A|_{gh}\).  Any element of~\(U\cdot V\) is of the
  form~\(g\cdot h\) for unique \(g\in U\), \(h\in V\).  Continuity
  of~\((\alpha_g)\) means that for all bisections \(U,V\) and all
  \(a=(a_h)_{h\in V}\) in~\(A(V)\), the section \((g\cdot h)\mapsto
  \alpha_{g,h}(a_h)\) for \(g\in U\), \(h\in V\) is continuous
  on~\(U\cdot V\), that is, it belongs to~\(A(U\cdot V)\) (see also
  \cite{Popescu:Equivariant_E}*{Definition~2.3}).  Thus we get
  isomorphisms \(\alpha_U\colon A(V)\to A(U\cdot V)\).  In brief,
  \(\Bis(G)\) acts on~\(A\) by partial isomorphisms.

  Since~\(G^1\) is non-Hausdorff, there are \(g_1,g_2\in G^1\) that
  cannot be separated by open subsets.  Then \(\rg(g_1)=\rg(g_2)\)
  and \(\s(g_1)=\s(g_2)\).  Let \(U_1\) and~\(U_2\) be bisections
  of~\(G\) containing \(g_1\) and~\(g_2\), respectively.  Shrinking
  them, we may achieve that \(\s(U_1)=\s(U_2)\).  Let
  \[
  V\defeq U_1^*U_1=\{1_x\mid x\in \s(U_1)\} = U_2^*U_2;
  \]
  then \(U_1V=U_1\) and \(U_2V=U_2\).  Since \(g_1\) and~\(g_2\)
  cannot be separated, there is a net~\((h_n)\) in~\(U_1\cap U_2\)
  that converges both to \(g_1\) and to~\(g_2\).

  Let \(f\in A(V)\) with \(f(1_{\s(g_1)})\neq0\).  Then
  \(\alpha_{U_1}(f)\in A(U_1V)\) and \(\alpha_{U_2}(f)\in A(U_2V)\) by
  our continuity assumption.  Thus
  \[
  \psi\defeq \alpha_{U_1}(f)\cdot \alpha_{U_2}(f)^*\in
  A(U_1V)\cap A(U_2V) = A(U_1\cap U_2),
  \]
  so~\(\psi\) vanishes at \(g_1\) and~\(g_2\).  At \(h_n\in U_1\cap
  U_2\), we have
  \[
  \alpha_{U_1}(f)(h_n)
  = \alpha_{U_1\cap U_2}(f)(h_n)
  = \alpha_{U_2}(f)(h_n)
  = \alpha_{h_n}(f(1_{\s(h_n)})).
  \]
  Since each~\(\alpha_{h_n}\) is an isomorphism, we get
  \[
  \norm{\psi(h_n)} =
  \norm{\alpha_{h_n}(f(1_{\s(h_h)})f(1_{\s(h_n)})^*)} =
  \norm{f(1_{\s(h_h)})}^2.
  \]
  If \(U\subseteq G^1\) is Hausdorff and \(a\in A(U)\), then \(U\ni
  x\mapsto \norm{a}_x\) is continuous
  by~\cite{Nilsen:Bundles}*{Corollary 2.2} because the map \(\Prim
  A(U)\to U\) is open and~\(U\) is Hausdorff and locally compact.
  Therefore, \(\norm{\psi(h_n)}\) converges towards
  \(\norm{\psi(g_1)}=0\).  At the same time, \(\norm{\psi(h_n)}\)
  converges towards \(\norm{f(1_{\s(g_1)})}^2\neq 0\) because
  \(\s(h_n)\to \s(g_1)\) inside the Hausdorff open subset~\(V\).  This
  contradiction shows that there is no continuous action of~\(G\)
  on~\(A\) that lifts the multiplication action on \(\Prim(A)\cong
  G^1\).
\end{proof}

\begin{remark}
  \label{rem:counterexample_more_general}
  More generally, if we only assume an open continuous surjection
  \(p\colon \Prim(A)\to G^1\), then there is no continuous action
  of~\(G\) on~\(A\) such that~\(p\) is \(G\)\nb-equivariant for
  the induced action of~\(G\) on~\(\Prim(A)\) and the left
  multiplication action on~\(G^1\); the proof is exactly the same.
\end{remark}

The proof of Theorem~\ref{the:no-go-theorem} does not care about the
multiplicativity of the action, so allowing ``twisted'' actions
of~\(G\) does not help.  There are only two ways around this.
First, we may allow Fell bundles over~\(G\).  Secondly, we may allow
actions of the inverse semigroup~\(\Bis(G)\).  After stabilisation
every Fell bundle becomes a twisted action by partial automorphisms
(see~\cite{BussExel:Regular.Fell.Bundle}).  We cannot remove the
twist, however, because an untwisted action of~\(\Bis(G)\) by
automorphisms would give an action of~\(G\) by automorphisms as
well, which cannot exist by Theorem~\ref{the:no-go-theorem}.

\section{A simple explicit example}
\label{sec:explicit_example}

Let~\(G\) be the group bundle over \(G^0=[0,1]\) with
trivial isotropy groups~\(G(x)\) for \(x\neq0\) and with
\(G(0)\cong\Z/2=\{1,-1\}\).  So, as a set, \(G\)~is \((0,1]\cup
\{0^+,0^-\}\) with \(0^+\)~corresponding to \(+1\in \Z/2\) and
\(0^-\)~to \(-1\in \Z/2\).  The topology on~\(G\) is the quotient
topology from \([0,1]\times\Z/2\), where we divide by the equivalence
relation generated by \((x,1)\sim (x,-1)\) for \(x\neq0\).  With this
topology, \(G\)~is an étale, quasi-compact, second countable, locally
Hausdorff, non-Hausdorff groupoid (even a group bundle).  The points
\(0^+\) and~\(0^-\) cannot be topologically separated: any net
in~\((0,1]\) converging to~\(0^+\) also converges to~\(0^-\), and vice
versa.

Let~\(H\) be the groupoid of the equivalence relation~\(\sim\) on
\([0,1]\times \Z/2\) just defined.  Its \(\Cst\)\nb-algebra
\(C^*(H)\cong C^*_\red(H)\) is
\[
A\defeq \{f\in \Cont([0,1],\Mat_2): f(0)\text{ is diagonal}\}.
\]
(This can be proved using the same idea as in
\cite{Clark-Huef-Raeburn:Fell_algebras}*{Example~7.1}.)  This is a
\(\Cst\)\nb-algebra over~\([0,1]\) with fibres \(A_x\cong\Mat_2\) at
\(x\neq0\) and \(A_0\cong\C^2\), and it has \(\hat{A} \cong
\Prim(A)\cong G^1\) (this is a special case of
\cite{Clark-Huef-Raeburn:Fell_algebras}*{Corollary~5.4}).
Theorem~\ref{the:no-go-theorem} shows that there is no action
of~\(G\) on~\(A\) by automorphisms that would model the left
multiplication action of~\(G\) on~\(G^1\).

Since~\(A\) is the groupoid \(\Cst\)\nb-algebra of the \v{C}ech
groupoid for the covering \([0^+,1]\cup [0^-,1] = H^1\), our main
results give an action of~\(G\) on~\(A\) by Hilbert bimodules.  We
first describe it as an inverse semigroup action for a very small
inverse semigroup~\(S\) that models~\(G\).  We consider three
special bisections of~\(G\):
\[
1= [0^+,1]=G^1\setminus\{0^-\},\quad
g=[0^-,1]=G^1\setminus\{0^+\},\quad
e=(0,1] = g\cap 1.
\]
The bisection~\(1\) is the unit bisection of~\(G\), so \(1x=x=x1\)
for all \(x\in\{1,g,e\}\).  Moreover, \(g^2=1\), \(e^2=e\), and
\(eg=ge=e\).  Thus \(S\defeq \{1,e,g\}\) is an inverse semigroup
with \(x^*=x\) for all \(x\in\{1,e,g\}\).  A bisection~\(t\)
of~\(G\) cannot contain both \(0^+\) and~\(0^-\).  Hence either
\(0^+\in t\subseteq 1\), \(0^-\in t\subseteq g\), or \(t\subseteq
e=1\cap g\).

The groupoid~\(G\) is the étale groupoid associated to the trivial
action of~\(S\) on~\(G^0\); here the trivial action has \(1\)
and~\(g\) acting by the identity on~\(G^0\) and~\(e\) acting by the
identity on \((0,1]\subseteq G^0\).  An action of~\(G\) on a
groupoid or a \(\Cst\)\nb-algebra is equivalent to an action
of~\(S\) together with a compatible action of \(G^0=[0,1]\)
(Theorem~\ref{the:model_action_Cstar}).

The transformation groupoid~\(\Link\) of the \(S\)\nb-action on~\(H\) may
be identified with the groupoid of the equivalence relation on
\([0,1]\sqcup [0,1]\) that identifies the two copies of~\((0,1]\),
so that
\begin{multline*}
\Link^1= [0,1]\times\{(+,+),(+,-),(-,+),(-,-)\} \\
\subseteq([0,1]\times\{(+,+),(+,-),(-,+),(-,-)\})^2.
\end{multline*}
The \(S\)\nb-grading on~\(\Link\) has
\begin{align*}
  \Link_1 &= (0,1]\times\{(+,+),(+,-),(-,+),(-,-)\} \sqcup
  \{0\}\times\{(+,+),(-,-)\},\\
  \Link_g &= (0,1]\times\{(+,+),(+,-),(-,+),(-,-)\} \sqcup
  \{0\}\times\{(+,-), (-,+)\},\\
  \Link_e &= (0,1]\times\{(+,+),(+,-),(-,+),(-,-)\} =
  \Link_1\cap \Link_g.
\end{align*}
So \(\Link_1\cong H\) is open but not closed.  The
\(\Cst\)\nb-algebra of~\(\Link\) is \(B\defeq \Cont([0,1],\Mat_2)\).

To let~\(S\) act on the \(\Cst\)\nb-algebra~\(A\) of~\(H\), we use
the transformation groupoid \(\Cst\)\nb-algebra~\(B\) and the involution
\(u \defeq \bigl(\begin{smallmatrix}0&1\\1&0\end{smallmatrix}\bigr)
\in B\).  We have \(u=u^*\) and \(u^2=1\), \(u\cdot A(0,1] = A(0,1]
= A(0,1]\cdot u\) and \(u A = A u\) as subsets of~\(B\).  Let
\(A_1\defeq A\), \(A_e\defeq A(0,1]\subseteq A_1\), and \(A_g\defeq
uA=Au\).  These subspaces~\(A_x\) for \(x\in S\) satisfy
\(A_x^*=A_x=A_{x^*}\) for all \(x\in S\) and \(A_x\cdot A_y =
A_{xy}\) for all \(x,y\in S\); in particular, \(A_g\)~is a full
Hilbert bimodule over~\(A_1\) with inner products given by the usual
formulas \(a_1^*\cdot a_2\) and \(a_1\cdot a_2^*\).  Furthermore,
\(A_1\cap A_g=A_e\) and \(A_1 + A_g=B\) because elements of~\(A_g\)
are precisely those \(f\in B\) with off-diagonal~\(f(0)\).  Hence
the map \(g\mapsto A_g\) defines an action of~\(S\) on~\(A\) by
Hilbert bimodules.  Since \(A_1+A_g\) is already
complete in the \(\Cst\)\nb-norm of~\(B\), there is only one
\(\Cst\)\nb-norm on~\(A_1+A_g\) that extends the given
\(\Cst\)\nb-norm on~\(A_1\).  Thus the sectional \(\Cst\)\nb-algebra
for the resulting Fell bundle over~\(S\) is~\(B\), which is
Morita--Rieffel equivalent to~\(\Cont[0,1]\).

The \(S\)\nb-action on~\(A\) extends to all bisections of~\(G\)
because they are all contained in \(1\) or~\(g\): if \(t\subseteq
G^1\) is a bisection, then let \(A_t=A_1|_{\s(t)}\) if \(t\subseteq
1\) and \(A_t=A_g|_{\s(t)}\) if \(t\subseteq g\); this is consistent
for \(t\subseteq 1\cap g=e\) because \(A_e=A_1\cap A_g\).

Next we describe a twisted \(S\)\nb-action by partial automorphisms
of~\(A\) that induces the \(S\)\nb-action by Hilbert bimodules
described above.  (This is possible by
\cite{BussExel:Regular.Fell.Bundle}*{Corollary~4.16} because our
saturated Fell bundle is regular in the notation
of~\cite{BussExel:Regular.Fell.Bundle}.)

A twisted \(S\)\nb-action by partial automorphisms is given by
ideals \(A_1=A\) and~\(A_e\) with isomorphisms \(\alpha_x\colon
A_{xx^*}\to A_{x^*x}\) and unitary multipliers (the twists)
\(\omega(x,y)\) in \(\Mult(A_{xyy^*x^*})\) for \(x,y\in S\).  For the
idempotent elements \(x=e,1\), the isomorphism~\(\alpha_x\) is the
identity; for \(x=g\), it is the order-\(2\)-automorphism
\(\alpha_g\colon A\to A\), \(a\mapsto uau\), because \(a_1\cdot ua_2
= u\cdot (ua_1u\cdot a_2)\) for all \(a_1\in A_1\), \(ua_2\in A_g\).
The automorphism~\(\alpha_g\) is \emph{not} inner on~\(A_1\) because
\(u\in B\) does not belong to~\(\Mult(A)\).  The restriction
of~\(\alpha_g\) to the ideal~\(A_e\) becomes inner, however, because
\(u\in\Mult(A(0,1])\).  This unitary~\(u\) enters in the twisting
unitaries~\(\omega(x,y)\) for \(x,y\in S\); they are \(1\) if
\(x=1\) or \(y=1\), or if \((x,y)\) is \((e,e)\) or~\((g,g)\)
(\(\alpha_g^2=\Id_A = \alpha_1\)).  The remaining cases are
\(\omega(e,g)=\omega(g,e) = u|_{A_e}\), that is, \(u\)~viewed as a
multiplier of the ideal \(A_e=A(0,1]\).  It is routine to check that
this data gives a twisted action of~\(S\) on~\(A\) in the sense of
\cite{BussExel:Regular.Fell.Bundle}*{Definition~4.1} and that the
resulting saturated Fell bundle over~\(S\) is isomorphic to the one
described above.  Incidentally, this is not a twisted action in the
sense of Sieben~\cite{SiebenTwistedActions} because \(\omega(e,g)\)
and~\(\omega(g,e)\) are non-trivial although~\(e\) is idempotent.

This twisted \(S\)\nb-action cannot be turned into a groupoid action
of~\(G\) by partial automorphisms because for \(x\in 1\cap g\),
the restrictions of \(\alpha_g\) and~\(\alpha_1\) to \(A|_{\s(x)}\)
differ by a non-trivial inner automorphism.  This impossibility is
in accord with Theorem~\ref{the:no-go-theorem}.

\begin{remark}
  \label{rem:Packer-Raeburn}
  The Packer--Raeburn Stabilisation Trick replaces a twisted group
  action by an untwisted action on a suitable
  \(\Cst\)\nb-stabilisation.  We claim that this cannot be done for
  the above inverse semigroup twisted action.  Let~\(D\) be a
  \(\Cst\)\nb-algebra with an untwisted action of~\(S\) by
  automorphisms.  Then \(1\) and~\(e\) act by the identity on~\(D\)
  and by some ideal \(D_e\idealin D\), respectively, and~\(g\)
  acts by some automorphism~\(\alpha_g\) on~\(D\).  If there is no
  twist, then \(\alpha_g|_{D_e} = \alpha_1|_{D_e}\) is the identity
  on~\(D_e\) because \(eg=e=ge\).  Suppose that~\(D\) is also a
  \(\Cst\)\nb-algebra over~\([0,1]\), with \(D((0,1]) = D_e\).  Then
  this allows to define an action of the groupoid~\(G\) on~\(D\) by
  letting elements of \(g\) or~\(1\) act by the fibre restrictions
  of \(\alpha_g\) and~\(\Id_D\), respectively.  This gives a
  well-defined, untwisted action of~\(G\) on~\(D\).
  Theorem~\ref{the:no-go-theorem} implies that
  \(\Prim(D)\not\cong G^1\), so that \(A\) and~\(D\) cannot be
  Morita--Rieffel equivalent.  This example therefore shows that the
  Packer--Raeburn Stabilisation Trick cannot be extended from groups
  to inverse semigroups or non-Hausdorff groupoids.
\end{remark}

\appendix
\section{Preliminaries on topological groupoids}
\label{sec:preliminaries}

This appendix defines topological groupoids and equivalences between
them, following~\cite{Meyer-Zhu:Groupoids}.  The point is that all
this works smoothly without assuming topological spaces to be
Hausdorff or locally (quasi)compact, if we choose appropriate
definitions.  The theory of possibly non-Hausdorff topological
groupoids becomes very natural if one treats topological groupoids,
Lie groupoids, infinite-dimensional Lie groupoids (modelled on
Banach or Fréchet manifolds), and other types of groupoids
simultaneously as in~\cite{Meyer-Zhu:Groupoids}.  Here we recall the
results and definitions from~\cite{Meyer-Zhu:Groupoids} that are
relevant for us.

The theory of topological groupoids and their principal bundles and
equivalences depends on a choice of ``covers'' in the category of
topological spaces (see~\cite{Meyer-Zhu:Groupoids}).  We choose the
open surjections as covers.  This means that we require \emph{the
  range and source maps in a topological groupoid, the bundle
  projection in a principal bundle, and the anchor maps in a (bibundle)
  equivalence to be open surjections}.

Following Bourbaki, we require compact and locally compact spaces to
be Hausdorff.  Since many authors allow non-Hausdorff locally
compact spaces, we usually speak of ``Hausdorff locally compact''
spaces to avoid confusion.  A topological space is \emph{locally
  quasi-compact} if every point has a neighbourhood basis consisting
of quasi-compact neighbourhoods.  This is strictly more than having
a single quasi-compact neighbourhood, but both notions coincide in
the locally Hausdorff case, which is the case we are interested in.
Recall that a topological space is \emph{locally Hausdorff} if every
point has a Hausdorff neighbourhood (and thus a neighbourhood basis
consisting of Hausdorff neighbourhoods).  A space is locally
Hausdorff, locally quasi-compact if and only if every point has a
compact (hence Hausdorff) neighbourhood.  It would make sense to call
such spaces ``locally compact,'' if it were not for the conflict
with other established notation.

\subsection{Topological groupoids,
principal bundles, and equivalences}
\label{sec:groupoids_bundles_equivalences}

We now specialise the general definitions of groupoids, groupoid
actions, principal bundles, basic groupoid actions and bibundle
equivalences in~\cite{Meyer-Zhu:Groupoids} to the category of (all)
topological spaces with open surjections as covers.

\begin{proposition}
  \label{pro:Top_open_groupoid}
  A topological groupoid consists of topological spaces \(G^0\)
  and~\(G^1\) and continuous maps \(\rg,\s\colon
  G^1\rightrightarrows G^0\) and \(m\colon
  G^1\times_{\s,G^0,\rg} G^1\to G^1\), \((g_1,g_2)\mapsto
  g_1\cdot g_2\), such that
  \begin{enumerate}[label=\textup{(G\arabic*)}]
  \item \label{enum:Gr1} \(\s(g_1\cdot g_2)=\s(g_2)\) and
    \(\rg(g_1\cdot g_2)=\rg(g_1)\) for all \(g_1,g_2\in G^1\);
  \item \label{enum:Gr2} \(m\) is associative: \((g_1\cdot
    g_2)\cdot g_3=g_1\cdot (g_2\cdot g_3)\) for all
    \(g_1,g_2,g_3\in G^1\) with \(\s(g_1)=\rg(g_2)\),
    \(\s(g_2)=\rg(g_3)\);
  \item \label{enum:Gr3} the following two maps are homeomorphisms:
    \begin{alignat*}{2}
      G^1\times_{\s,G^0,\rg} G^1&\to
      G^1\times_{\s,G^0,\s} G^1,
      &\qquad (g_1,g_2)&\mapsto (g_1\cdot g_2,g_2),\\
      G^1\times_{\s,G^0,\rg} G^1&\to
      G^1\times_{\rg,G^0,\rg} G^1,
      &\qquad (g_1,g_2)&\mapsto (g_1,g_1\cdot g_2),
    \end{alignat*}
  \item \label{enum:Gr4} \(\rg\) and~\(\s\) are open surjections.
  \end{enumerate}
  Then~\(m\) is open and surjective and there are continuous maps
  \(G^0\to G^1\) and~\(G^1\to G^1\) with the usual properties of
  unit and inversion.  Conversely, the maps in~\ref{enum:Gr3}
  are homeomorphisms if~\(G\) has continuous unit and inversion maps.
\end{proposition}

\begin{proof}
  Our definition of a groupoid is exactly
  \cite{Meyer-Zhu:Groupoids}*{Definition~3.4}.  It implies
  that~\(m\) is open and surjective and is equivalent to the usual
  one with unit and inverse by \cite{Meyer-Zhu:Groupoids}*{Proposition
    3.6}.
\end{proof}

Let~\(G\) be a topological groupoid as above.

\begin{proposition}
  \label{pro:Top_open_action}
  A \textup{(}right\textup{)} \(G\)\nb-action is a space~\(X\)
  with continuous maps \(\s\colon X\to G^0\) and \(m\colon
  X\times_{\s,G^0,\rg} G^1\to X\), \((x,g)\mapsto x\cdot
  g\), such that
  \begin{enumerate}[label=\textup{(A\arabic*)}]
  \item \label{enum:Act1} \(\s(x\cdot g)=\s(g)\) for all \(x\in X\),
    \(g\in G^1\) with \(\s(x)=\rg(g)\);
  \item \label{enum:Act2} \(m\) is associative: \((x\cdot g_1)\cdot
    g_2 = x\cdot (g_1\cdot g_2)\) for all \(x\in X\),
    \(g_1,g_2\in G^1\) with \(\s(x)=\rg(g_1)\) and
    \(\s(g_1)=\rg(g_2)\);
  \item \label{enum:Act3} \(m\) is surjective.
  \end{enumerate}
  Condition~\ref{enum:Act3} holds if and only if \(x\cdot
  1_{\s(x)}=x\) for all \(x\in X\), if and only if~\(m\) is an
  open surjection, if and only if the following map is a
  homeomorphism:
  \[
  X\times_{\s,G^0,\rg} G^1\to
  X\times_{\s,G^0,\s} G^1,
  \qquad (x,g)\mapsto (x\cdot g,g).
  \]
\end{proposition}

\begin{proof}
  This is contained in \cite{Meyer-Zhu:Groupoids}*{Definition and
    Lemma 4.1}.
\end{proof}

Left actions are defined similarly and are equivalent to right actions
by \(g\cdot x=x\cdot g^{-1}\).  The transformation groupoid
\(X\rtimes G\) of a groupoid action is a topological groupoid by
\cite{Meyer-Zhu:Groupoids}*{Definition and Lemma 4.11}.  Any groupoid
acts on~\(G^0\) by \(\rg(g)\cdot g \defeq \s(g)\) for all
\(g\in G^1\), and on~\(G^1\) both on the left and right by left and
right multiplication.

\begin{proposition}
  \label{pro:Top_open_orbit}
  For any \(G\)\nb-action on a topological space~\(X\), the orbit
  space projection \(X\to X/G\) is an open surjection, and
  \(X/G\)~is Hausdorff if and only if
  \begin{multline*}
    X\times_{X/G} X = \{(x_1,x_2)\in X\mid
    \\\text{there is }g\in G^1\text{ with }\s(x_1)=\rg(g)
    \text{ and }x_1\cdot g=x_2\}
  \end{multline*}
  is a closed subset of~\(X\times X\).
\end{proposition}

\begin{proof}
  The orbit space projection is open by
  \cite{Meyer-Zhu:Groupoids}*{Proposition 9.31} because the range and
  source maps of~\(G\) are open.  By
  \cite{Meyer-Zhu:Groupoids}*{Proposition 9.18}, \(X/G\)~is
  Hausdorff if and only if \(X\times_{X/G} X\) is closed
  in~\(X\times X\) (open surjections are clearly biquotient maps,
  see the discussion in~\cite{Meyer-Zhu:Groupoids}*{Section 9.6}).
\end{proof}

We now specialise the general concepts of basic actions and principal
bundles from~\cite{Meyer-Zhu:Groupoids} to our context.

\begin{proposition}
  \label{pro:Top_open_basic}
  A right \(G\)\nb-action is \emph{basic} if the map
  \begin{equation}
    \label{eq:basic_map}
    X\times_{\s,G^0,\rg} G^1\to
    X\times X,\qquad (x,g)\mapsto (x,x\cdot g),
  \end{equation}
  is a homeomorphism onto its image with the subspace topology.

  A \emph{principal right \(G\)\nb-bundle} is a space~\(X\) with
  continuous maps \(\s\colon X\to G^0\), \(p\colon
  X\to Z\), and \(m\colon X\times_{\s,G^0,\rg} G^1\to
  X\), \((x,g)\mapsto x\cdot g\), such that
  \begin{enumerate}[label=\textup{(Pr\arabic*)}]
  \item \label{enum:Prin1} \(\s(x\cdot g)=\s(g)\) and \(p(x\cdot
    g)=p(x)\) for all \(x\in X\), \(g\in G^1\) with
    \(\s(x)=\rg(g)\);
  \item \label{enum:Prin2} \(m\) is associative: \((x\cdot g_1)\cdot
    g_2 = x\cdot (g_1\cdot g_2)\) for all \(x\in X\),
    \(g_1,g_2\in G^1\) with \(\s(x)=\rg(g_1)\) and
    \(\s(g_1)=\rg(g_2)\);
  \item \label{enum:Prin3} the map
    \[
    X\times_{\s,G^0,\rg} G^1\to
    X\times_{p,Z,p} X,\qquad (x,g)\mapsto (x,x\cdot g),
    \]
    is a homeomorphism;
  \item \label{enum:Prin4} the map~\(p\) is open and surjective.
  \end{enumerate}
  Then \(x\cdot 1_{\s(x)}=x\) for all \(x\in X\), and there is a unique
  homeomorphism \(Z\cong X/G\) intertwining~\(p\) and the
  canonical projection \(X\to X/G\).  Thus a principal
  \(G\)\nb-bundle is equivalent to a basic \(G\)\nb-action with a
  homeomorphism \(X/G\cong Z\).
\end{proposition}

\begin{proof}
  A principal bundle in the sense above also satisfies \(x\cdot
  1_{\s(x)}=x\) for all \(x\in X\) because of~\ref{enum:Prin3} (see
  \cite{Meyer-Zhu:Groupoids}*{Lemma 5.3}).  Hence \(\s\) and~\(m\)
  give a right \(G\)\nb-action, and all conditions for a principal
  bundle in~\cite{Meyer-Zhu:Groupoids} are met.
  \cite{Meyer-Zhu:Groupoids}*{Lemma 5.3} also gives the unique
  homeomorphism \(X/G\cong Z\) intertwining~\(p\) and the
  canonical map \(X\to X/G\).

  A groupoid action is called \emph{basic}
  in~\cite{Meyer-Zhu:Groupoids} if it becomes a principal bundle with
  \(X\to X/G\) as bundle projection.  The canonical map
  \(X\to X/G\) is automatically \(G\)\nb-invariant, and it is
  an open surjection by \cite{Meyer-Zhu:Groupoids}*{Proposition 9.31}.
  Thus the second half of~\ref{enum:Prin1} and~\ref{enum:Prin4} hold
  for any \(G\)\nb-action with this choice of~\(p\).  The first
  half of~\ref{enum:Prin1} and~\ref{enum:Prin2} are part of the
  definition of a groupoid action.  The image of the map
  in~\eqref{eq:basic_map} is \(X\times_{X/G} X\) by the
  definition of~\(X/G\), so that~\ref{enum:Prin3} is equivalent
  to~\eqref{eq:basic_map} being a homeomorphism onto its image.
\end{proof}

Next, we consider the notion of equivalence between groupoids as
defined in~\cite{Meyer-Zhu:Groupoids}.  We will relate it to notions
of equivalence by other authors in
Appendix~\ref{sec:basic_vs_free_proper}.

\begin{proposition}
  \label{pro:equivalence}
  Let \(G\) and~\(H\) be topological groupoids.  A \emph{bibundle
    equivalence} from~\(H\) to~\(G\) consists of a topological
  space~\(X\), continuous maps \(\rg\colon X\to G^0\), \(\s\colon X\to
  H^0\) \emph{(anchor maps)}, \(G^1\times_{\s,G^0,\rg} X\to X\) and
  \(X\times_{\s,H^0,\rg} H^1\to X\) \emph{(multiplications)},
  satisfying the following conditions:
  \begin{enumerate}[label=\textup{(E\arabic*)}]
  \item \label{enum:Eq1} \(\s(g\cdot x)=\s(x)\), \(\rg(g\cdot
    x)=\rg(g)\) for all \(g\in G^1\), \(x\in X\) with
    \(\s(g)=\rg(x)\), and \(\s(x\cdot h)=\s(h)\), \(\rg(x\cdot
    h)=\rg(x)\) for all \(x\in X\), \(h\in H^1\) with
    \(\s(x)=\rg(h)\);
  \item \label{enum:Eq2} associativity: \(g_1\cdot (g_2\cdot x)=
    (g_1\cdot g_2)\cdot x\), \(g_2\cdot (x\cdot h_1)= (g_2\cdot
    x)\cdot h_1\), \(x\cdot (h_1\cdot h_2)= (x\cdot h_1)\cdot h_2\)
    for all \(g_1,g_2\in G^1\), \(x\in X\), \(h_1,h_2\in H^1\) with
    \(\s(g_1)=\rg(g_2)\), \(\s(g_2)=\rg(x)\), \(\s(x)=\rg(h_1)\),
    \(\s(h_1)=\rg(h_2)\);
  \item \label{enum:Eq3} the following two maps are homeomorphisms:
    \begin{align*}
      G^1\times_{\s,G^0,\rg} X &\to X\times_{\s,H^0,\s} X,&\qquad
      (g,x)&\mapsto (x,g\cdot x),\\
      X\times_{\s,H^0,\rg} H^1 &\to X\times_{\rg,G^0,\rg} X,&\qquad
      (x,h)&\mapsto (x,x\cdot h);
    \end{align*}
  \item \label{enum:Eq4} \(\s\) and~\(\rg\) are open;
  \item \label{enum:Eq5} \(\s\) and~\(\rg\) are surjective.
  \end{enumerate}
  Then \(1_{\rg(x)}\cdot x=x=x\cdot 1_{\s(x)}\) for all \(x\in X\),
  and the anchor maps descend to homeomorphisms \(G\backslash X\cong
  H^0\) and \(X/H\cong G^0\).
\end{proposition}

\begin{proof}
  Condition \ref{enum:Eq1} and~\ref{enum:Eq3} are equivalent to
  \ref{enum:Prin1} and~\ref{enum:Prin3} for both the left
  \(G\)\nb-action with \(p=\s\) and the right \(H\)\nb-action with
  \(p=\rg\), respectively.  Condition~\ref{enum:Eq2} means that
  the left \(G\)\nb-action and the right \(H\)\nb-action
  satisfy~\ref{enum:Prin2} and commute.  Conditions \ref{enum:Eq4}
  and~\ref{enum:Eq5} together are equivalent to~\ref{enum:Prin4} for
  both actions.  Thus the conditions \ref{enum:Eq1}--\ref{enum:Eq5}
  characterise bibundle equivalences in the notation
  of~\cite{Meyer-Zhu:Groupoids}.  The last sentence follows from the
  general properties of principal bundles, see
  Proposition~\ref{pro:Top_open_basic}.
\end{proof}

In the following, we abbreviate ``bibundle equivalence'' to
``equivalence'' because we do not use any other equivalences between
groupoids.

We have switched the direction of a bibundle equivalence compared
to~\cite{Meyer-Zhu:Groupoids} because this is convenient here.  Going
from right to left is also consistent with our notation \(\s\)
and~\(\rg\) for the right and left anchor maps.

\subsection{Basic actions versus free and proper actions}
\label{sec:basic_vs_free_proper}

We now compare our basic actions with free and proper actions.  A
continuous map \(f\colon X\to Y\) is \emph{closed} if it maps closed
subsets of~\(X\) to closed subsets of~\(Y\), and \emph{proper} if
\(\Id_Z\times f\colon Z\times X\to Z\times Y\) is closed for all
topological spaces~\(Z\) or, equivalently, \(f\)~is closed and
\(f^{-1}(y)\) is quasi-compact for all \(y\in Y\) (see
\cite{Bourbaki:Topologie_generale}*{Theorem~1 in I.10.2}).  A map from
a Hausdorff space~\(X\) to a Hausdorff locally compact space~\(Y\) is
proper if and only if preimages of compact subsets are compact.  In
this case, \(X\) is necessarily locally compact
(\cite{Bourbaki:Topologie_generale}*{Proposition 7 in I.10.3}).

\begin{definition}
  \label{def:proper_action}
  A right action of a topological groupoid~\(G\) on a topological
  space~\(X\) is \emph{proper} if the map in~\eqref{eq:basic_map} is
  proper.  The action is \emph{free} if the map~\eqref{eq:basic_map}
  is injective.
\end{definition}

Groupoids for which the action on its unit space is free (that is, for
which the map \(\s\times\rg\colon G^1\to G^0\times G^0\) is injective)
are often called \emph{principal} (see~\cite{Renault:Groupoid_Cstar}).
This terminology conflicts, however, with the usual notion of a
principal bundle, which requires extra topological conditions besides
freeness of the action.

We call a groupoid \emph{basic} if its canonical action on the object
space is basic, that is, the map \(\s\times\rg\colon G^1\to G^0\times
G^0\) is a homeomorphism onto its image.

\begin{proposition}
  \label{pro:proper_action}
  A groupoid action is free and proper if and only if it is basic and
  has Hausdorff orbit space.

  If \(G\) and~\(H\) are topological groupoids with Hausdorff
  object spaces, then an equivalence from \(H\) to~\(G\) in
  our sense is the same as a topological space~\(X\) with commuting
  free and proper actions of \(G\) and~\(H\), such that the
  anchor maps induce homeomorphisms \(G\backslash X\cong H^0\)
  and \(X/H\cong G^0\).
\end{proposition}

\begin{proof}
  The characterisation of free and proper actions is
  \cite{Meyer-Zhu:Groupoids}*{Corollary 9.32}; the main point of the
  proof is that the orbit space is Hausdorff if and only if the orbit
  equivalence relation is closed in \(X\times X\)
  (Proposition~\ref{pro:Top_open_orbit}).  The left and right actions
  on an equivalence are basic with \(X/H\cong G^0\) and
  \(G\backslash X\cong H^0\); hence they are free and proper
  if and only if \(G^0\) and~\(H^0\) are Hausdorff,
  respectively.  Conversely, if the actions of \(G\) and~\(H\)
  on~\(X\) are free and proper, then both actions are basic, and
  both anchor maps are open because they are equivalent to orbit space
  projections; thus we have an equivalence in our sense.
\end{proof}

For a general action of a groupoid~\(G\) on a space~\(X\), the
image of the map~\eqref{eq:basic_map} is the orbit equivalence
relation \(X\times_{X/G}X\subseteq X\times X\).  Thus
the map~\eqref{eq:basic_map} is a homeomorphism (the action is basic)
if and only if the action is free and the map that sends
\((x_1,x_2)\in X\times_{X/G}X\) to the unique \(g\in G^1\)
with \(\s(x_1)=\rg(g)\) and \(x_1\cdot g=x_2\) is continuous.

If \(G\), \(H\) and~\(X\) are locally compact Hausdorff, then an
equivalence in our sense is the same as a \((G,H)\)-equivalence in
the notation of~\cite{Muhly-Renault-Williams:Equivalence}; the main
result of~\cite{Muhly-Renault-Williams:Equivalence} is that such an
equivalence induces a Morita equivalence between the groupoid
\(\Cst\)\nb-algebras of \(G\) and~\(H\) (for any Haar systems).

For non-Hausdorff groupoids, Jean-Louis Tu defined a notion of
equivalence in~\cite{Tu:Non-Hausdorff}, using a technical variant of
proper actions: he calls a groupoid~\(G\) \emph{\(\rho\)\nb-proper}
with respect to a \(G\)\nb-invariant continuous map \(\rho\colon
G^0\to T\) if the map
\[
(\rg,\s)\colon G^1\to G^0\times_{\rho,T,\rho} G^0,\qquad
g\mapsto (\rg(g),\s(g)),
\]
is proper.  If~\(T\) is non-Hausdorff, then \(G^0\times_{\rho,T,\rho}
G^0\) need not be closed in~\(G^0\times G^0\), so that this is weaker
than properness.  In the definition of equivalence, he takes~\(\rho\)
to be the anchor map on the other side, so he requires the maps
in~\ref{enum:Eq3} to be proper.  These maps are continuous
bijections because the actions are free.  A
continuous, proper bijection, being closed, must be a homeomorphism.
Thus Tu's notion of equivalence is equivalent to ours.

\subsection{Covering groupoids and equivalence}
\label{sec:covering_groupoids}

\begin{definition}
  \label{def:covering_groupoid}
  Let \(f\colon X\to Z\) be a continuous, open surjection.  The
  \emph{covering groupoid}~\(G(f)\) has object space~\(X\), arrow
  space \(X\times_{f,Z,f} X\), range and source maps
  \(\rg(x_1,x_2)\defeq x_1\), \(\s(x_1,x_2)\defeq x_2\), and
  multiplication \((x_1,x_2)\cdot (x_2,x_3)\defeq (x_1,x_3)\) for all
  \(x_1,x_2,x_3\in X\) with \(f(x_1)=f(x_2)=f(x_3)\).
\end{definition}

The assumption on~\(f\) implies that it is a quotient map, that is, we
may identify~\(Z\) with the quotient space~\(X/{\sim}\) by the
following equivalence relation: \(x\sim y\) if and only if
\(f(x)=f(y)\); and \(f\) becomes the quotient map \(X\to X/{\sim}\).
The covering groupoid~\(G(f)\) is the groupoid associated to this
equivalence relation.  In particular, \(Z\) can be identified with the
orbit space~\(X/G(f)\) for the canonical action of~\(G(f)\) on its
unit space~\(X\).

Every covering groupoid is basic, that is, its action on the unit
space is basic.  Conversely, if~\(G\) is a basic groupoid, then it is
isomorphic to a covering groupoid.  The map \(\rg\times\s\colon G^1\to
G^0\times G^0\) gives a homeomorphism from \(G^1\) onto
\(G^0\times_{f,G^0/G,f} G^0\), where \(f\colon G^0\to G^0/G\) denotes
the quotient map.  This yields an isomorphism of topological groupoids
\(G\cong G(f)\).

\begin{example}[\v{C}ech groupoids]
  \label{exa:Cech}
  Let~\(Z\) be a topological space and let~\(\OCover\) be an open
  covering of~\(Z\).  Let \(X\defeq \bigsqcup_{U\in\OCover} U\)
  and let \(f\colon X\to Z\) be the canonical map: \(f\)~is the
  inclusion map on each \(U\in\OCover\).  This map is an open
  surjection.  It is even étale, that is, a local homeomorphism.  We
  denote the covering groupoid of~\(f\) by~\(G_\OCover\) and call
  it the \emph{\v{C}ech groupoid} of the covering.

  Assume that~\(Z\) is locally Hausdorff and choose the open
  covering~\(\OCover\) to consist of Hausdorff open subsets
  \(U\subset Z\).  Then the \v{C}ech groupoid \(G_\OCover\) is a
  Hausdorff, étale topological groupoid (see also
  \cite{Clark-Huef-Raeburn:Fell_algebras}*{Lemma~4.2}).  If, in
  addition, \(Z\) is locally quasi-compact, then \(G_\OCover\) is a
  (Hausdorff) locally compact, étale groupoid.  This is the situation
  we are mainly interested in.
\end{example}

\begin{proposition}
  \label{prop:MoritaInvariantCovGroupoid}
  Let \(f_i\colon X_i\to Z\) for \(i=1,2\) be two continuous, open
  surjections.  Then \(X_1\times_{f_1,Z,f_2} X_2\) with the obvious
  left and right actions of \(G(f_1)\) and~\(G(f_2)\) gives an
  equivalence from \(G(f_2)\) to~\(G(f_1)\).
\end{proposition}

\begin{proof}
  This is \cite{Meyer-Zhu:Groupoids}*{Example 6.4}.
\end{proof}

If \(G(f_1)\) and~\(G(f_2)\) are Hausdorff locally compact, then so is
the equivalence \(X_1\times_{f_1,Z,f_2} X_2\) between them.  If the
maps \(f_1\) and~\(f_2\) are both étale -- for instance, if they
come from open coverings of~\(Z\) -- then the groupoids \(G(f_1)\)
and~\(G(f_2)\) are étale, and the anchor maps \(X_1\leftarrow
X_1\times_{f_1,Z,f_2} X_2\rightarrow X_2\), \(x_1\leftarrow
(x_1,x_2)\rightarrow x_2\), are étale as well.

\begin{proposition}
  \label{pro:CharacterizationCovGroupoids}
  The covering groupoid~\(G(f)\) of a continuous open surjection
  \(f\colon X\to Z\) is always equivalent \textup{(}as a topological
  groupoid\textup{)} to the space~\(Z\) viewed as a groupoid with only
  identity arrows.  In particular, the \v{C}ech groupoid of a covering
  of~\(Z\) is equivalent to~\(Z\).

  Conversely, if~\(X\) is an equivalence from a space~\(Z\) to a
  topological groupoid~\(G\), then \(G\)~is isomorphic to the covering
  groupoid of the anchor map \(\s\colon X\to Z\).

  Hence covering groupoids are exactly the groupoids that are
  equivalent to spaces.
\end{proposition}

\begin{proof}
  The first part is a consequence of
  Proposition~\ref{prop:MoritaInvariantCovGroupoid} applied to
  \(f_1=f\) and \(f_2=\Id_Z\) (see also
  \cite{Meyer-Zhu:Groupoids}*{Example 6.3}).  For the second part,
  observe that the action of~\(Z\) on~\(X\) is simply the anchor map
  \(\s\colon X\to Z\), which must be an open surjection.  The anchor
  map \(\rg\colon X\to G^0\) must be a homeomorphism (because it must
  be the projection map \(X\to Z\backslash X=X\)), so we may as well
  assume \(X=G^0\).  Then \(G^1\times_{\s,G^0,\rg} X \cong G^1\), and
  the first isomorphism in~\ref{enum:Eq3} identifies~\(G^1\) with
  \(X\times_{\s,Z,\s} X\).  This yields an isomorphism from~\(G\) to
  the covering groupoid~\(G(s)\) of \(\s\colon X\to Z\).
\end{proof}

Let~\(Z\) be a space, view~\(Z\) as a groupoid with only identity
arrows.  When is~\(Z\) equivalent to a locally compact, Hausdorff
groupoid?  If~\(Z\) is equivalent to a topological groupoid~\(G\),
then~\(G\) is necessarily the covering groupoid~\(G(f)\) of a cover
\(f\colon X\to Z\) by
Proposition~\ref{pro:CharacterizationCovGroupoids}.

Given a space~\(Z\), we thus seek a locally compact, Hausdorff
space~\(X\) and an open, continuous surjection \(f\colon X\to Z\) such
that \(X\times_{f,Z,f} X\) is locally compact.  The question when
\(X\times_{f,Z,f} X\) is locally compact is also asked
in~\cite{Clark-Huef-Raeburn:Fell_algebras} at the end of Section~4.
We answer this question in
Proposition~\ref{pro:locally_Hausdorff_quotient} below:
\(X\times_{f,Z,f} X\) is locally compact if and only if~\(Z\) is
locally Hausdorff.  Proposition~\ref{pro:orbit_principal_lc_groupoid}
says that the only topological spaces~\(Z\) that are equivalent to
locally compact Hausdorff groupoids are the locally Hausdorff, locally
quasi-compact ones; for them, Example~\ref{exa:Cech} gives such an
equivalence, where the groupoid is even étale.  We need some
preparation in order to prove
Proposition~\ref{pro:locally_Hausdorff_quotient}.

\begin{definition}[\cite{Bourbaki:Topologie_generale}*{I.3.3,
    Définition 2, Proposition 5}]
  \label{def:locally_closed}
  A subset~\(S\) of a topological space~\(X\) is \emph{locally closed}
  if it satisfies the following equivalent conditions:
  \begin{enumerate}
  \item any \(x\in S\) has a neighbourhood~\(U\) such that \(S\cap U\)
    is relatively closed in~\(U\);
  \item \(S\) is open in its closure;
  \item \(S\) is an intersection of an open and a closed subset
    of~\(X\).
  \end{enumerate}
\end{definition}

The following proposition generalises
\cite{Bourbaki:Topologie_generale}*{I.9.7, Propositions 12 and 13} to
the locally Hausdorff case.

\begin{proposition}
  \label{pro:locally_closed_compact}
  A subset~\(S\) of a locally Hausdorff, locally quasi-compact
  space~\(X\) is locally quasi-compact in the subspace topology if and
  only if it is locally closed.
\end{proposition}

\begin{proof}
  First let~\(S\) be locally closed.  Write \(S=A\cap U\) with \(A\)
  closed and~\(U\) open in~\(X\).  Let \(x\in S\).  Since~\(X\) is
  locally quasi-compact, the quasi-compact neighbourhoods of~\(x\)
  in~\(X\) form a neighbourhood basis of~\(X\).  Since \(x\in U\),
  those quasi-compact neighbourhoods of~\(x\) that are contained
  in~\(U\) form a neighbourhood basis in~\(U\).  Their intersections
  with~\(A\) remain quasi-compact because~\(A\) is closed in~\(X\).
  They form a neighbourhood basis of~\(x\) in~\(S\), proving
  that~\(S\) is locally quasi-compact.

  Conversely, assume that~\(S\) is locally quasi-compact in the
  subspace topology.  Let \(x\in S\).  Let~\(U\) be a Hausdorff open
  neighbourhood of~\(x\) in~\(X\).  Then \(S\cap U\) is a
  neighbourhood of~\(x\) in~\(S\) and hence contains a quasi-compact
  neighbourhood~\(K\) of~\(x\) in~\(S\) because~\(S\) is locally
  quasi-compact.  We have \(K=S\cap V\) for some neighbourhood~\(V\)
  of~\(x\) in~\(X\), and we may assume \(V\subseteq U\) because
  \(K\subseteq U\).  The subset \(S\cap V\) is relatively closed
  in~\(V\) because \(U\supseteq V\) is Hausdorff and \(S\cap
  V\) is quasi-compact.  Thus~\(S\) is locally closed.
\end{proof}

\begin{proposition}
  \label{pro:locally_Hausdorff_quotient}
  Let \(f\colon X\to Z\) be a continuous, open surjection.  The
  equivalence relation \(X_{f,Z,f} X\subseteq X\times X\) defined
  by~\(f\) is locally closed if and only if~\(Z\) is locally
  Hausdorff.  In particular, if~\(X\) is locally quasi-compact and
  locally Hausdorff, then \(X_{f,Z,f} X\) is locally quasi-compact if
  and only if~\(Z\) is locally Hausdorff.
\end{proposition}

\begin{proof}
  Assume~\(Z\) to be locally Hausdorff first.  Let \((x_1,x_2)\in
  X\times_{f,Z,f} X\) and let \(U\subseteq Z\) be a Hausdorff open
  neighbourhood of \(f(x_1)=f(x_2)\).  Then \(f^{-1}(U)\subseteq X\)
  is an open subset such that \(f\colon f^{-1}(U)\to U\) is an open
  map onto a Hausdorff space.  Hence
  \[
  f^{-1}(U)\times_{f,U,f} f^{-1}(U)
  = \bigl(X\times_{f,Z,f} X\bigr) \cap
  \bigl(f^{-1}(U)\times f^{-1}(U)\bigr)
  \]
  is relatively closed in \(f^{-1}(U)\times f^{-1}(U)\) by
  \cite{Meyer-Zhu:Groupoids}*{Proposition 9.15}.
  Thus~\(X\times_{f,Z,f} X\) is locally closed in~\(X\times X\).

  Conversely, assume~\(X\times_{f,Z,f} X\) to be locally closed
  in~\(X\times X\).  Let \(x\in X\).  Then \((x,x)\) has a
  neighbourhood in~\(X\times X\) so that~\(X\times_{f,Z,f} X\)
  restricted to it is relatively closed.  Shrinking this
  neighbourhood, we may assume that it is of the form \(U\times U\)
  for an open neighbourhood of~\(x\), by the definition of the product
  topology on~\(X\times X\).  The map \(f|_U\colon U\to f(U)\) is
  open, and \((X\times_{f,Z,f} X)\cap (U\times U) =
  U\times_{f|_U,f(U),f|_U} U\).  Since this is relatively closed by
  assumption, \cite{Meyer-Zhu:Groupoids}*{Proposition 9.15} shows
  that~\(f(U)\) is Hausdorff.  Since~\(x\) was arbitrary, this means
  that~\(Z\) is locally Hausdorff.

  The last sentence follows from the first one and
  Proposition~\ref{pro:locally_closed_compact}.
\end{proof}

\begin{corollary}
  \label{cor:Hausdorff_diagonal_closed}
  A topological space~\(X\) is locally Hausdorff if and only if the
  diagonal \(\{(x,x)\mid x\in X\}\) is a locally closed subset
  in~\(X\times X\).
\end{corollary}

\begin{proof}
  Apply Proposition~\ref{pro:locally_Hausdorff_quotient} to the
  identity map.
\end{proof}

\begin{proposition}
  \label{pro:orbit_principal_lc_groupoid}
  Let~\(G\) be a locally quasi-compact, locally Hausdorff groupoid and
  let~\(X\) be a basic right \(G\)\nb-action.  Then~\(X/G\) is locally
  quasi-compact and locally Hausdorff.

  If~\(X\) is an equivalence from a space~\(Z\) to~\(G\), then
  \(Z\cong X/G\) is locally quasi-compact and locally Hausdorff.
\end{proposition}

\begin{proof}
  Since \(G\) and~\(X\) are locally quasi-compact and locally
  Hausdorff, so is their product \(X\times G^1\).  Since~\(G^0\) is
  locally Hausdorff, the diagonal in~\(G^0\) is locally closed by
  Corollary~\ref{cor:Hausdorff_diagonal_closed}.  The fibre product
  \(X\times_{\s,G^0,\rg} G^1\) is the preimage of the diagonal
  in~\(G^0\times G^0\) under the continuous map \(\rg\times\s\colon
  X\times G^1 \to G^0\times G^0\); hence \(X\times_{\s,G^0,\rg} G^1\)
  is locally closed in~\(X\times G^1\).  Thus \(X\times_{\s,G^0,\rg}
  G^1\) is locally quasi-compact and locally Hausdorff by
  Proposition~\ref{pro:locally_closed_compact}.

  Since the \(G\)\nb-action on~\(X\) is basic, \(X\times_{\s,G^0,\rg}
  G^1\) is homeomorphic to the subset \(X\times_{X/G} X\subseteq
  X\times X\).  Now Proposition~\ref{pro:locally_closed_compact} shows
  that \(X\times_{X/G} X\) is locally closed in~\(X\times X\).
  Then~\(X/G\) is locally Hausdorff by
  Proposition~\ref{pro:locally_Hausdorff_quotient}.  Since continuous
  images of quasi-compact subsets are again quasi-compact, \(X/G\)~is
  also locally quasi-compact.

  An equivalence from a space~\(Z\) to~\(G\) is the same as a basic
  \(G\)\nb-action with a homeomorphism \(X/G\cong Z\).  If this
  exists, then~\(Z\) must be locally Hausdorff and locally
  quasi-compact by the above argument.
\end{proof}

\section{Fields of Banach spaces over locally Hausdorff spaces}
\label{sec:Banach_fields}

Let~\(X\) be a locally quasi-compact, locally Hausdorff space.  Thus
any Hausdorff open subset of~\(X\) is locally compact.

\begin{definition}[see~\cite{Muhly-Williams:Renaults_equivalence} and
  the references there]
  \label{def:continuous_Banach_field}
  An \emph{upper semicontinuous field of Banach spaces} on~\(X\) is a
  family of Banach spaces~\((\Banb_x)_{x\in X}\) with a topology on
  \(\Banb=\bigsqcup_{x\in X} \Banb_x\) such that, for each Hausdorff
  open subset~\(U\) of~\(X\), \(\Banb|_U\) is an upper semicontinuous
  field of Banach spaces on~\(U\).  In particular, the norm of any
  continuous section of~\(\Banb|_U\) is an upper semicontinuous
  scalar-valued function on~\(U\).

  Let \(\Sect(U,\Banb)\) denote the vector space of continuous,
  compactly supported sections of~\(\Banb|_U\).  This is the union
  (hence inductive limit) of the subspaces \(\Sect_0(K,\Banb)\) of
  continuous sections on~\(K\) vanishing on~\(\partial K\),
  where~\(K\) runs through the directed set of compact subsets
  of~\(U\) and \(\partial K=K\cap \cl{U\setminus K}\) is the boundary
  of~\(K\) in~\(U\).  Each \(\Sect_0(K,\Banb)\) is a Banach space for
  the supremum norm
  \[
  \norm{f}_\infty \defeq \sup \{\norm{f(x)}\mid x\in K\}.
  \]
  We call a subset of \(\Sect(U,\Banb)\) \emph{bounded} if it is the
  image of a norm-bounded subset of \(\Sect_0(K,\Banb)\) for
  some~\(K\).

  If \(f\in \Sect(U,\Banb)\) for a Hausdorff open subset~\(U\)
  of~\(X\), then we always extend~\(f\) to a section of~\(\Banb\) on
  all of~\(X\) by taking \(f(x)\defeq 0\) for \(x\notin U\).  Let
  \(\Sect(X,\Banb)\) be the vector space of all sections
  of~\(\Banb\) that may be written as finite linear combinations
  \(\sum_{i=1}^m f_i\) for \(f_i\in \Sect(U_i,\Banb)\) and Hausdorff
  open subset~\(U_i\) of~\(X\).  We call such sections of~\(\Banb\)
  \emph{quasi-continuous}.

  A subset~\(A\) of \(\Sect(X,\Banb)\) is \emph{bounded} if there are
  Hausdorff open subsets \(U_1,\dotsc,U_m\) of~\(X\) and bounded
  subsets \(A_i\subseteq \Sect(U_i,\Banb)\) for \(i=1,\dotsc,m\) such
  that every element of~\(A\) may be written as a sum \(\sum_{i=1}^m
  f_i\) with \(f_i\in A_i\) for \(i=1,\dotsc,m\).
\end{definition}

To simplify our proofs, we use bornological language, that is, we
speak of \emph{bounded} instead of open subsets.  For a Hausdorff
locally compact space~\(X\), \(\Sect(X,\Banb)\) with its usual
topology is an inductive limit of Banach spaces.  The inductive
limit topology is determined by its continuous seminorms.  A
seminorm is continuous if and only if it is bounded in the sense
that its supremum over each bounded subset is finite; this is so
because a seminorm on a Banach space is continuous if and only if it
is bounded.  For locally Hausdorff~\(X\), the bounded seminorms are
those that restrict to bounded seminorms on all the subspaces
\(\Sect(U,\Banb)\) for \(U\subseteq X\) open and Hausdorff; this is
the same as the quotient topology from the map \(\bigoplus_U
\Sect(U,\Banb)\to\Sect(X,\Banb)\), where~\(U\) runs through the
Hausdorff open subsets of~\(X\).  Thus the usual topology on
\(\Sect(X,\Banb)\) -- which is the quotient topology induced by the
inductive limit topologies on the direct sums of the spaces
\(\Sect(U,\Banb)\) -- is
the topology generated by all bounded seminorms.

Let~\(\OCover\) be a family of open subsets of~\(X\) with the
following two properties:
\begin{enumerate}
\item \(X=\bigcup_{U\in \OCover} U\), that is, for each \(x\in
  X\) there is \(U\in \OCover\) with \(x\in U\);
\item \(U_1\cap U_2 = \bigcup \{U\in \OCover \mid U\subseteq U_1\cap
  U_2\}\) for all \(U_1, U_2\in \OCover\); that is, if \(x\in U_1\cap
  U_2\), then there is \(U\in \OCover\) with \(U\subseteq U_1\cap
  U_2\) and \(x\in U\).
\end{enumerate}
In our main application, the open subsets in~\(\OCover\) will not be
Hausdorff.  Thus \(\Sect(U,\Banb)\) for \(U\in\OCover\) is defined in
the same way as \(\Sect(X,\Banb)\), by taking finite linear
combinations of continuous compactly supported sections on Hausdorff
open subsets of~\(U\).  We view \(\Sect(U,\Banb)\) as a subspace in
\(\Sect(X,\Banb)\) by extending functions on~\(U\) by~\(0\)
outside~\(U\).  This gives an injective, bounded linear map
\(\Sect(U,\Banb)\to\Sect(X,\Banb)\).  Being bounded means that it maps
bounded subsets to bounded subsets.

Let \(\iota_U\colon \Sect(U,\Banb)\to \bigoplus_{U\in\OCover}
\Sect(U,\Banb)\) for \(U\in\OCover\) denote the inclusion map of the
\(U\)\nb-summand.  We call a subset~\(A\) of \(\bigoplus_{U\in\OCover}
\Sect(U,\Banb)\) bounded if there are finitely many
\(U_1,\dotsc,U_m\in\OCover\) and bounded subsets~\(A_i\)
of~\(\Sect(U_i,\Banb)\) such that any element of~\(A\) may be written
as \(\sum_{i=1}^m \iota_{U_i}(f_i)\) with \(f_i\in A_i\).

\begin{proposition}
  \label{pro:Banb_and_cover}
  The map
  \[
  E\colon \bigoplus_{U\in\OCover} \Sect(U,\Banb)\to\Sect(X,\Banb)
  \]
  is bounded linear and a bornological quotient map in the sense that
  any bounded subset of \(\Sect(X,\Banb)\) is the image of a bounded
  subset of~\(\bigoplus_{U\in\OCover} \Sect(U,\Banb)\); in particular,
  it is surjective.

  The kernel of~\(E\) is the closed linear span of the set of elements
  of the form \(\iota_U(f)-\iota_V(f)\) for \(f\in\Sect(U,\Banb)\),
  \(U,V\in\OCover\) with \(U\subseteq V\).
\end{proposition}

The ``closure'' in the description of the kernel is the bornological
one, defined using Mackey's notion of convergence in a bornological
vector space.  For any element \(g\in \ker E\), we will find a bounded
subset \(A\subseteq \bigoplus_{U\in\OCover} \Sect(U,\Banb)\), and
linear combinations~\(g_n\) of \(\iota_U(f)-\iota_V(f)\) for
\(f\in\Sect(U,\Banb)\), \(U,V\in\OCover\) with \(U\subseteq V\) such
that \(g-g_n\in 2^{-n}\cdot A\).  This implies convergence in any
bounded seminorm.

\begin{remark}
  Proposition~\ref{pro:Banb_and_cover} implies that~\(E\) is a
  quotient map with respect to the canonical topologies on the spaces
  involved.  That is, a seminorm~\(p\) on~\(\Sect(X,\Banb)\) is
  continuous if and only if \(p\circ E\) is a continuous seminorm on
  \(\bigoplus_{U\in \OCover}\Sect(U,\Banb)\).  The proof uses that
  continuity and boundedness are equivalent for seminorms on both
  spaces and
  that~\(E\) is a bornological quotient map.  It seems inconvenient,
  however, to prove this directly without bornological language.
\end{remark}

\begin{proof}
  In the proof, we abbreviate \(\Sect(U)\defeq\Sect(U,\Banb)\) because
  the Banach space bundle is fixed throughout.  We first show
  that~\(E\) is a bornological quotient map.

  Let \(A\subseteq\Sect(X)\) be bounded.  By definition, there are
  finitely many Hausdorff open subsets \(V_1,\dotsc,V_m\subseteq X\),
  compact subsets \(K_i\subseteq V_i\) and scalars \(C_i>0\) such that
  any \(f\in A\) may be written as \(\sum_{i=1}^m f_i\) with
  \(f_i\in\Sect_0(K_i)\) having \(\norm{f_i}_\infty\le C_i\).

  Since the subsets \(U\in\OCover\) cover~\(X\), they cover the
  compact subset~\(K_i\).  Since compact spaces are paracompact, there
  is a finite subordinate partition of unity
  \((\psi_{i,U})_{U\in\OCover}\), that is, \(\psi_{i,U}\colon K_i\to
  [0,1]\) is continuous and has compact support~\(L_{i,U}\) contained
  in \(U\cap K_i\), only finitely many~\(\psi_{i,U}\) are non-zero, and
  \(\sum_{U\in\OCover} \psi_{i,U}(x)=1\).  If \(f_i\in\Sect_0(K_i)\),
  then \(f_i\cdot\psi_{i,U}\in\Sect_0(K_i \cap L_{i,U})\subseteq
  \Sect(V_i\cap U)\) and \(\norm{f_i\cdot\psi_{i,U}}_\infty\le
  \norm{f_i}_\infty\).

  Now write \(f\in A\) first as \(\sum_{i=1}^m f_i\) with
  \(f_i\in\Sect_0(K_i)\) having \(\norm{f_i}_\infty\le C_i\), and
  then as \(\sum_{i=1}^n \sum_{U\in\OCover} f_i\cdot \psi_{i,U}\).
  This sum is still finite because only finitely many~\(\psi_{i,U}\)
  are non-zero for each~\(i\), and each summand \(f_i\cdot\psi_{i,U}\)
  runs through
  a bounded subset of \(\Sect(V_i\cap U)\) and hence of~\(\Sect(U)\)
  because we have uniform control on the supports \(\supp
  f_i\psi_{i,U}\subseteq K_i\cap L_{i,U}\) and norms
  \(\norm{f_i\cdot\psi_{i,U}}_\infty\le C_i\) of the summands.
  Hence~\(A\) is contained in the \(E\)\nb-image of a bounded
  subset in \(\bigoplus \Sect(U)\).

  Now we describe the kernel of~\(E\).  Let~\(N\) be the linear span
  of elements of the form \(\iota_U(f)-\iota_V(f)\) for all
  \(f\in\Sect(U)\), \(U,V\in\OCover\) with \(U\subseteq V\).  Since
  \(E(\iota_U(f)-\iota_V(f))=0\), we have \(N\subseteq \ker E\).  If
  \(U_1,U_2,V\in\OCover\) satisfy \(V\subseteq U_1\cap U_2\) and
  \(f\in\Sect(V)\), then
  \[
  \iota_{U_1}(f)-\iota_{U_2}(f)
  = -(\iota_V(f)-\iota_{U_1}(f)) + (\iota_V(f)-\iota_{U_2}(f)) \in N.
  \]
  We are going to modify a given element of~\(\ker E\) by adding
  elements of~\(N\) so that the norms of its constituents become
  arbitrarily small, without enlarging their supports.

  A generic element \(f\in \bigoplus_{U\in\OCover} \Sect(U)\) is of
  the form \(f=\sum \iota_U(f_U)\) with \(f_U\in\Sect(U)\) and
  \(f_U=0\) for all but finitely many~\(U\).  Each non-zero~\(f_U\) is
  a sum \(f_U = \sum_{j=1}^{k_U} f_{U,j}\) with
  \(f_{U,j}\in\Sect(V_{U,j})\) for finitely many Hausdorff open
  subsets \(V_{U,1},\dotsc,V_{U,k_U}\subseteq U\).  We renumber the
  finitely many Hausdorff open subsets~\(V_{U,j}\) consecutively as
  \(V_1,\dotsc,V_m\) and relabel our sections \(f_i\in\Sect(V_i)\)
  accordingly.  Let \(U_i\in\OCover\) for \(i=1,\dotsc,m\) be such
  that \(f=\sum_{i=1}^m \iota_{U_i}(f_i)\); so \(V_i\subseteq U_i\).
  Let \(K_i\defeq \supp f_i \subseteq V_i\) and let~\(K_i^\circ\) be
  the interior
  of~\(K_i\) inside~\(V_i\); thus \(x\in K_i^\circ\) for all \(x\in
  X\) with \(f_i(x)\neq0\).

  Now assume \(f\in \ker(E)\) and let \(\epsilon>0\).  We will
  construct a finite sequence \(f^{(j)} =
  \sum_{i=1}^m \iota_{U_i}(f_i^{(j)})\) with \(f^{(0)}=f\),
  \(f^{(j+1)}-f^{(j)}\in N\), and \(\norm{f_i^{(m)}}<\epsilon\) for
  all \(i=1,\dotsc,m\).  Furthermore, our construction ensures that
  the support of~\(f_i^{(j)}\) is contained in~\(K_i\) for all
  \(i,j\).  Letting~\(\epsilon\) run through a sequence going
  to~\(0\), the differences \(f-f^{(m)}\) in~\(N\) will converge
  to~\(f\) in the sense explained above because each constituent
  \(f_i-f_i^{(m)}\) converges to~\(f_i\) in the normed
  space~\(\Sect_0(K_i)\).  Our construction will be such that
  \(f_i^{(j)} = f_i^{(i)}\) for \(j\ge i\), that is, in the \(j\)th
  step we keep \(f_1,\dotsc,f_{j-1}\) fixed.  To make the following
  steps possible, we aim for stronger norm estimates
  \(\norm{f_i^{(j)}}< 2^{j-m}\epsilon\).  Assume that we have already
  constructed \(f^{(j)} = \sum_{i=1}^m \iota_{U_i}(f_i^{(j)})\)
  with \(f-f^{(j)}\in N\) and \(\norm{f_i^{(j)}}<2^{j-m}
  \epsilon\) for \(i=1,\dotsc,j\); for \(j=0\), this is satisfied
  for \(f^{(0)}=f\).  We are going to construct \(f^{(j+1)} =
  \sum_{i=1}^m \iota_{U_i}(f_i^{(j+1)})\) with \(f^{(j)}-f^{(j+1)}\in
  N\) and hence \(f-f^{(j+1)}\in N\), with \(f_i^{(j+1)}=f_i^{(j)}\) for
  \(i=1,2,\dotsc,j\), and \(\norm{f_{j+1}^{(j+1)}}<2^{j+1-m} \epsilon\).

  Let \(A_{j+1} = \{x\in V_{j+1} \mid \norm{f_{j+1}^{(j)}(x)}\ge
  2^{j+1-m}\epsilon\}\).  This is a closed subset of~\(K_{j+1}^\circ\)
  because the norm function is upper semicontinuous.  Since~\(K_{j+1}\)
  is compact, \(A_{j+1}\) is compact.  Since \(E(f)=0\) and
  \(E(f-f^{(j)})=0\), we have \(\sum_{i=1}^m f_i^{(j)}(x)=0\)
  for all \(x\in X\).  If \(x\in A_{j+1}\), then this gives
  \[
  \biggl\lVert
    \sum_{i=j+2}^m f_i^{(j)}(x)
  \biggr\rVert
  =
  \biggl\lVert
    \sum_{i=1}^{j+1} f_i^{(j)}(x)
  \biggr\rVert
  \ge
  \norm{f_{j+1}^{(j)}(x)} - \sum_{i=1}^{j} \norm{f_i^{(j)}}_\infty
  >0.
  \]
  Hence there must be \(i>j+1\) with \(f_i^{(j)}(x)\neq0\), so that
  \(x\in K_i^\circ\).  Thus the open subsets~\(K_i^\circ\) for
  \(i>j+1\) cover~\(A_{j+1}\).  If \(x\in A_{j+1}\cap K_i^\circ\), then \(x\in
  U_i\cap U_{j+1}\).  By our assumption on~\(\OCover\), there is
  \(U\in\OCover\) with \(x\in U\) and \(U\subseteq U_i\cap U_{j+1}\).
  Thus the open subsets \(K_i^\circ\cap U\) for \(i>j+1\) and
  \(U\in\OCover\) with \(U\subseteq U_i\cap U_{j+1}\) cover~\(A_{j+1}\).

  Since~\(A_{j+1}\) is compact and contained in the Hausdorff locally
  compact space \(V_{j+1}\), there is a subordinate finite partition of
  unity~\((\psi_{i,U})\).  That is, all but finitely
  many~\(\psi_{i,U}\) are non-zero, \(\psi_{i,U}\colon A_{j+1}\to
  [0,1]\) is a continuous function with compact support contained in
  \(K_i^\circ\cap U\), and \(\sum \psi_{i,U}(x)=1\) for \(x\in
  A_{j+1}\).  We may extend each non-zero~\(\psi_{i,U}\) from~\(A_{j+1}\) to
  a continuous function \(\bar\psi_{i,U}\colon K_{j+1}\to[0,1]\)
  vanishing in a neighbourhood of~\(\partial K_{j+1}\) and on
  \(K_{j+1}\setminus (K_i^\circ\cap U)\) because these two compact
  subsets of~\(A_{j+1}\) are disjoint from the compact support
  of~\(\psi_{i,U}\) in \(K_i^\circ\cap U\).  If necessary, we
  multiply all~\(\bar\psi_{i,U}\) with a suitable cut-off function
  so that \(\sum \bar\psi_{i,U}(x)\le1\) for all \(x\in K_{j+1}\).

  Now we let
  \[
  f^{(j+1)} = f^{(j)} + \sum_{i, U}
  \iota_{U_i}(f_{j+1}^{(j)}\bar\psi_{i,U})
  - \iota_{U_{j+1}}(f_{j+1}^{(j)}\bar\psi_{i,U}).
  \]
  By construction, \(f_{j+1}^{(j)}\bar\psi_{i,U}\) is continuous and
  supported in a compact subset of \(K_i^\circ\cap U\) with
  \(U\subseteq U_i\cap U_{j+1}\), \(U\in\OCover\).  Hence
  \(\iota_{U_i}(f_{j+1}^{(j)}\bar\psi_{i,U}) -
  \iota_{U_{j+1}}(f_{j+1}^{(j)}\bar\psi_{i,U})\in N\), so
  \(f^{(j+1)}-f^{(j)}\in N\) as desired.  Since only \(i>j+1\) appear in
  the sum, \(f_i^{(j+1)}=f_i^{(j)}\) for \(i<j+1\).  We get
  \[
  f^{(j+1)}_{j+1}(x) = f^{(j)}_{j+1}(x)\cdot \biggl(1- \sum_{i,U}
    \bar\psi_{i,U}(x)\biggr).
  \]
  This has supremum norm less than \(2^{j+1-m}\epsilon\) because \(1-
  \sum_{i,U} \bar\psi_{i,U}(x)\) vanishes where
  \(\norm{f^{(j)}_{j+1}(x)}\ge 2^{j+1-m}\epsilon\) and is at most~\(1\)
  everywhere else.  The support of~\(f^{(j+1)}_{j+1}\) is still contained
  in~\(K_{j+1}\) by construction.

  For \(i>j+1\), we get
  \[
  f^{(j+1)}_i = f^{(j)}_i +  \sum_U f^{(j)}_{j+1}\cdot\bar\psi_{i,U}.
  \]
  This still has support~\(K_i\) because \(\bar\psi_{i,U}\) is
  supported there.  This completes the induction step and thus the
  proof.
\end{proof}

\begin{remark}
  If~\(X\) is Hausdorff, then a partition-of-unity argument as in the
  proof of
  \cite{BussExel:Fell.Bundle.and.Twisted.Groupoids}*{Theorem~2.13}
  shows that~\(\ker(E)\) is the linear span \emph{without closure} of
  \(\iota_U(f)-\iota_V(f)\) with \(U,V\in \OCover\).  Hence this
  linear span is already closed for the natural topology on
  \(\bigoplus_{U\in\OCover} \Sect(U,\Banb)\).  Convergent infinite
  series are needed to generate~\(\ker E\) from
  \(\iota_U(f)-\iota_V(f)\) with \(U,V\in \OCover\). This happens in simple
  examples, such as the space \(X=[0,1]\sqcup_{(0,1]}[0,1]\) discussed
  in Section~\ref{sec:explicit_example} with the trivial bundle~\(\C\)
  and the standard open cover by two Hausdorff open subsets with their
  intersection \((0,1]\).
\end{remark}

\subsection{Proof of Theorem~\ref{the:iterated_crossed_2}}
\label{sec:proof_iterated_crossed}

We apply Proposition~\ref{pro:Banb_and_cover} to \(X=L\),
the cover \((L_t)_{t\in S}\), and the given Fell bundle~\(\Banb\) as
in the statement of Theorem~\ref{the:iterated_crossed_2}.  The
subsets \(B^*\) and \(B_1*B_2\) for bounded subsets
\(B,B_1,B_2\subseteq\Sect(L,\Banb)\) are again bounded; this is
routine to check.  Thus \(\Sect(L,\Banb)\) is a bornological
\Star{}algebra.  (The continuity of the operations for the
``inductive limit topology'' is also known but somewhat more
difficult.)

We are going to cite some results of~\cite{Renault:Representations}
below, which follow from the Disintegration Theorem and the Morita
Equivalence Theorem.  We assume that they hold for the Fell
bundle~\(\Banb\) in question and its restriction to~\(G\); this is not
yet proved in the literature, see the discussion before
Theorem~\ref{the:iterated_crossed_2}.  Remark~\ref{rem:I-norm}
sketches a slightly more complicated proof that uses only the Morita
Equivalence Theorem, that is, the assumptions in
Theorem~\ref{the:iterated_crossed_2}.

\begin{lemma}
  \label{lem:Cstar-completion}
  The \(\Cst\)\nb-algebra \(\Cst(L,\Banb)\) is the completion
  of~\(\Sect(L,\Banb)\) in the maximal bounded \(\Cst\)\nb-seminorm.
\end{lemma}

\begin{proof}
  Usually, \(\Cst(L,\Banb)\) is defined as the completion
  of~\(\Sect(L,\Banb)\) in the maximal \(\Cst\)\nb-seminorm that is
  bounded with respect to the \(I\)\nb-norm, a certain norm
  on~\(\Sect(L,\Banb)\).  \cite{Renault:Representations}*{Corollaire
    4.8} shows that a representation of~\(\Sect(L,\Banb)\) that is
  continuous with respect to the ``inductive limit topology'' is
  bounded for the \(I\)\nb-norm.  Hence a \(\Cst\)\nb-seminorm
  on~\(\Sect(L,\Banb)\) is continuous with respect to the
  ``inductive limit topology'' if and only if it is bounded with
  respect to the \(I\)\nb-norm.  The topology on~\(\Sect(L,\Banb)\)
  called ``inductive limit topology''
  in~\cite{Renault:Representations} is really the quotient topology
  induced by the inductive limit topology on
  \(\bigoplus_{U\in\OCover} \Sect(U,\Banb)\),
  where~\(\OCover\) is the set of all Hausdorff open subsets
  of~\(L\) and \(\bigoplus_{U\in\OCover} \Sect(U,\Banb)\) is
  viewed as the inductive limit of the Banach subspaces
  \(\bigoplus_{U\in F} \Sect_0(K_U,\Banb)\) where~\(F\) is a finite
  subset of~\(\OCover\) and \(K_U\subseteq U\) for \(U\in F\)
  are compact subsets.  As we remarked above, a seminorm is
  continuous in this sense if and only if it is bounded in the
  canonical bornology on~\(\Sect(L,\Banb)\) introduced in
  Appendix~\ref{sec:Banach_fields}.
\end{proof}

Let \(D\defeq \bigoplus_{t\in S} \Sect(L_t,\Banb)\).  This carries a
canonical direct sum bornology as in
Appendix~\ref{sec:Banach_fields}.  The Fell bundle operations turn
it into a \Star{}algebra.  The multiplication and involution are
bounded, so we even have a bornological \Star{}algebra.  The map
\(E\colon D\to \Sect(L,\Banb)\) from
Proposition~\ref{pro:Banb_and_cover} is a bounded
\Star{}homomorphism.

Since~\(E\) is a bornological quotient map by
Proposition~\ref{pro:Banb_and_cover}, a
\(\Cst\)\nb-seminorm~\(p\) on~\(\Sect(L,\Banb)\) is bounded if and
only if \(p\circ E\) is a bounded \(\Cst\)\nb-seminorm on~\(D\).  A
bounded \(\Cst\)\nb-seminorm on~\(D\) is of the form \(p\circ E\) for
a \(\Cst\)\nb-seminorm~\(p\) on~\(\Sect(L,\Banb)\) if and only if it
vanishes on the kernel of~\(E\).  By
Proposition~\ref{pro:Banb_and_cover}, a bounded seminorm on~\(D\)
vanishes on~\(\ker E\) if and only if it vanishes on
\(\iota_t(f)-\iota_u(f)\) for all \(f\in\Sect(L_t,\Banb)\), \(t,u\in
S\), \(t\le u\).  Thus \(\Cst(L,\Banb)\) is isomorphic to the
completion of~\(D\) in the maximal \(\Cst\)\nb-seminorm~\(q\) on~\(D\)
that is bounded and vanishes on \(\iota_t(f)-\iota_u(f)\) for all
\(f,t,u\) as above.

The restriction of this \(\Cst\)\nb-seminorm~\(q\) to
\(\Sect(G,\Banb)\subseteq D\) is a bounded \(\Cst\)\nb-seminorm.
Since
\(\Cst(G,\Banb)\) is defined as the completion of~\(\Sect(G,\Banb)\)
with respect to the maximal bounded \(\Cst\)\nb-seminorm
on~\(\Sect(G,\Banb)\), \(q\) extends to a \(\Cst\)\nb-seminorm
on~\(\Cst(G,\Banb)\).  Since \(q(f)^2 = q(f^**f)\) for
\(f\in\Sect(L_t,\Banb)\), the restriction of~\(q\) to
\(\Sect(L_t,\Banb)\) is dominated by the Hilbert module norm from
\(\Cst(L_t,\Banb)\).  Thus~\(q\) automatically extends to the sum
\(\bigoplus_{t\in S} \Cst(L_t,\Banb)_{t\in S}\).  Furthermore, \(q\)
still annihilates \(\iota_t(f)-\iota_u(f)\) for all
\(f\in\Cst(L_t,\Banb)\), \(t,u\in S\), \(t\le u\) because
\(\Sect(L_t,\Banb)\) is dense in \(\Cst(L_t,\Banb)\).  Conversely, a
\(\Cst\)\nb-seminorm on \(\bigoplus_{t\in S} \Cst(L_t,\Banb)\) that
annihilates \(\iota_t(f)-\iota_u(f)\) for all
\(f\in\Cst(L_t,\Banb)\), \(t,u\in S\), \(t\le u\) restricts to a
\(\Cst\)\nb-seminorm~\(q\) on~\(D\) that annihilates
\(\iota_t(f)-\iota_u(f)\) for all \(f\in\Sect(L_t,\Banb)\), \(t,u\in
S\), \(t\le u\).  Since~\(D\) is dense in \(\bigoplus_{t\in S}
\Cst(L_t,\Banb)\), this says that \(\Cst(L,\Banb)\) is isomorphic to
the completion of \(\bigoplus_{t\in S} \Cst(L_t,\Banb)\) in the
maximal \(\Cst\)\nb-seminorm that annihilates
\(\iota_t(f)-\iota_u(f)\) for all \(f\in\Cst(L_t,\Banb)\), \(t,u\in
S\), \(t\le u\).  This is exactly the definition of the section
\(\Cst\)\nb-algebra of the Fell bundle \(\Cst(L_t,\Banb)_{t\in S}\)
over~\(S\).  This finishes the proof of
Theorem~\ref{the:iterated_crossed_2}.

\begin{remark}
  \label{rem:I-norm}
  We may also prove Theorem~\ref{the:iterated_crossed_2} without
  Lemma~\ref{lem:Cstar-completion}, using the usual definition of
  \(\Cst(L,\Banb)\) involving the \(I\)\nb-norm on~\(D\).  This
  variant of the proof has the advantage that it does not require the
  Disintegration Theorem.  We still need the Morita equivalence
  theorem for our Fell bundles, however, so that our inner
  products are positive and generate the expected ideals.

  We only
  explain the new points in this alternative proof.  The
  \(I\)\nb-norm on~\(\Sect(L,\Banb)\) restricts to the \(I\)\nb-norm
  on~\(\Sect(G,\Banb)\).  Consider a \(\Cst\)\nb-seminorm~\(q\)
  on~\(D\) that annihilates \(\iota_t(f)-\iota_u(f)\) for all
  \(f\in\Sect(L_t,\Banb)\), \(t,u\in S\), \(t\le u\) and satisfies
  \(q(f)\le \norm{f}_I\) for all \(f\in \Sect(G,\Banb)\).  Then
  \(q(f) = q(f^**f)^{1/2} \le \norm{f^**f}_I^{1/2} \le \norm{f}_I\)
  for all \(f\in\Sect(L_t,\Banb)\), \(t\in S\).  Thus~\(q\) is
  bounded with respect to our bornology as well, so it factors as
  \(\dot{q}\circ E\) for a bounded seminorm~\(\dot{q}\)
  on~\(\Sect(L,\Banb)\) by Proposition~\ref{pro:Banb_and_cover}.
  This seminorm satisfies \(\dot{q}(f)\le \norm{f}_I\) for all
  \(f\in\Sect(L_t,\Banb)\), \(t\in S\).  But then \(\dot{q}(f)\le
  \norm{f}_I\) follows for all \(f\in\Sect(L,\Banb)\), \(t\in S\).
\end{remark}

\end{document}